\documentclass[nointlimits]{amsart}

\usepackage[T1]{fontenc} 
\usepackage[utf8]{inputenc} 
\usepackage[USenglish]{babel} 

\usepackage{xparse}
\usepackage{ifthen}

\usepackage{lmodern}
\usepackage{mathrsfs}
\usepackage[nopatch=footnote]{microtype}

\usepackage{url}
\usepackage{natbib}

\usepackage{amsmath, amssymb, amsthm, enumerate}

\usepackage{xcolor}

\usepackage{enumitem}

\usepackage{hyperref}
\newcommand{\TITLE}{Compact Sobolev embeddings of radially symmetric functions}
\hypersetup{
	unicode=true,
	breaklinks=true,
	pdftitle={\TITLE},
	pdfauthor={Zdeněk Mihula},
	pdfsubject={We provide a complete characterization of compactness of Sobolev embeddings of radially symmetric functions on the entire space in the general framework of rearrangement-invariant function spaces. We avoid any unnecessary restrictions and cover also embeddings of higher order, providing a complete picture within this framework. To achieve this, we need to develop new techniques because the usual techniques used in the study of compactness of Sobolev embeddings in the general framework of rearrangement-invariant function spaces are limited to domains of finite measure, which is essential for them to work. Furthermore, we also study certain weighted Sobolev embeddings of radially symmetric functions on balls, where the weight is a nonnegative power of the distance from the origin. We completely characterize their compactness and also describe optimal target rearrangement-invariant function spaces in these weighted Sobolev embeddings.}
	}

\numberwithin{equation}{section}

\theoremstyle{plain}
\newtheorem{theorem}{Theorem}[section]
\newtheorem{corollary}{Corollary}[section]
\newtheorem{proposition}{Proposition}[section]

\theoremstyle{definition}
\newtheorem{remark}{Remark}[section]

\makeatletter
\let\c@corollary=\c@theorem
\let\c@proposition=\c@theorem
\let\c@remark=\c@theorem
\makeatother

\makeatletter
\let\save@mathaccent\mathaccent
\newcommand*\if@single[3]{%
  \setbox0\hbox{${\mathaccent"0362{#1}}^H$}%
  \setbox2\hbox{${\mathaccent"0362{\kern0pt#1}}^H$}%
  \ifdim\ht0=\ht2 #3\else #2\fi
  }
\newcommand*\rel@kern[1]{\kern#1\dimexpr\macc@kerna}
\newcommand*\widebar[1]{\@ifnextchar^{{\wide@bar{#1}{0}}}{\wide@bar{#1}{1}}}
\newcommand*\wide@bar[2]{\if@single{#1}{\wide@bar@{#1}{#2}{1}}{\wide@bar@{#1}{#2}{2}}}
\newcommand*\wide@bar@[3]{%
  \begingroup
  \def\mathaccent##1##2{%
    \let\mathaccent\save@mathaccent
    \if#32 \let\macc@nucleus\first@char \fi
    \setbox\z@\hbox{$\macc@style{\macc@nucleus}_{}$}%
    \setbox\tw@\hbox{$\macc@style{\macc@nucleus}{}_{}$}%
    \dimen@\wd\tw@
    \advance\dimen@-\wd\z@
    \divide\dimen@ 3
    \@tempdima\wd\tw@
    \advance\@tempdima-\scriptspace
    \divide\@tempdima 10
    \advance\dimen@-\@tempdima
    \ifdim\dimen@>\z@ \dimen@0pt\fi
    \rel@kern{0.6}\kern-\dimen@
    \if#31
      \overline{\rel@kern{-0.6}\kern\dimen@\macc@nucleus\rel@kern{0.4}\kern\dimen@}%
      \advance\dimen@0.4\dimexpr\macc@kerna
      \let\final@kern#2%
      \ifdim\dimen@<\z@ \let\final@kern1\fi
      \if\final@kern1 \kern-\dimen@\fi
    \else
      \overline{\rel@kern{-0.6}\kern\dimen@#1}%
    \fi
  }%
  \macc@depth\@ne
  \let\math@bgroup\@empty \let\math@egroup\macc@set@skewchar
  \mathsurround\z@ \frozen@everymath{\mathgroup\macc@group\relax}%
  \macc@set@skewchar\relax
  \let\mathaccentV\macc@nested@a
  \if#31
    \macc@nested@a\relax111{#1}%
  \else
    \def\gobble@till@marker##1\endmarker{}%
    \futurelet\first@char\gobble@till@marker#1\endmarker
    \ifcat\noexpand\first@char A\else
      \def\first@char{}%
    \fi
    \macc@nested@a\relax111{\first@char}%
  \fi
  \endgroup
}
\makeatother

\newcommand{\riSp}[1]{\ifthenelse{\equal{#1}{W}}{rearrangement-invariant function space}{\ifthenelse{\equal{#1}{w}}{rearrangement-invariant space}{r.i.\ space}}}
\newcommand{\riSps}[1]{\ifthenelse{\equal{#1}{W}}{rearrangement-invariant function spaces}{\ifthenelse{\equal{#1}{w}}{rearrangement-invariant spaces}{r.i.\ spaces}}}
\newcommand{\RiSps}[1]{\ifthenelse{\equal{#1}{W}}{Rearrangement-invariant function spaces}{\ifthenelse{\equal{#1}{w}}{Rearrangement-invariant spaces}{R.i.\ spaces}}}
\newcommand{\riNorm}[1]{\ifthenelse{\equal{#1}{w}}{rearrangement-invariant function norm}{r.i.\ function norm}}
\newcommand{\riNorms}[1]{\ifthenelse{\equal{#1}{w}}{rearrangement-invariant function norms}{r.i.\ function norms}}
\newcommand{\R}{\mathbb{R}}
\newcommand{\rn}{\R^n}
\newcommand{\N}{\mathbb{N}}
\newcommand{\M}{\mathscr{M}}
\newcommand{\Mpl}{\M^+}
\newcommand{\logsLZ}[1]{\mathbb{#1}}
\newcommand{\A}{\logsLZ{A}}
\newcommand{\B}{\logsLZ{B}}
\NewDocumentCommand{\LZ}{ O{(0, \infty)} m m m }{L^{{#2}, {#3}, {#4}}{#1}}

\newcommand{\ac}{\overset{*}{\hookrightarrow}}

\NewDocumentCommand{\aldx}{ O{\mu} O{\alpha} }{{#1}_{#2}}
\NewDocumentCommand{\alBvol}{ O{\mu} O{\alpha} O{n} }{{#1}_{{#2}, {#3}}}
\newcommand{\omN}{\omega_n}

\NewDocumentCommand{\wmrXB}{ O{m} O{X} O{R} }{W^{#1}_R{#2}(B_{#3})}
\NewDocumentCommand{\wmrXBvan}{ O{m} O{X} O{R} }{W^{#1}_{R,0}{#2}(B_{#3})}
\NewDocumentCommand{\wmrX}{ O{m} O{X} }{W^{#1}_R{#2}(\rn)}

\NewDocumentCommand{\aclocAB}{ O{0} m }{AC_{\text{loc}}(({#1},{#2}])}
\NewDocumentCommand{\aclocInt}{ m }{AC_{\text{loc}}({#1})}

\newcommand{\fundX}[1][X]{\varphi_{#1}}
\newcommand{\fundXrep}[1][X]{\varphi_{\widebar{#1}}}
\newcommand{\fundXAsocrep}[1][X]{\varphi_{\widebar{#1}'}}
\newcommand{\rep}[2][0, \infty]{\widebar{{#2}}({#1})}
\newcommand{\repFin}[2][0, 1]{\widebar{{#2}}({#1})}
\newcommand{\repAsoc}[2][0, \infty]{\widebar{{#2}}'({#1})}
\newcommand{\repAsocFin}[2][0, 1]{\widebar{{#2}}'({#1})}

\newcommand{\RR}{R}

\newcommand{\C}{\mathcal{C}}

\newcommand*\dd{\mathop{}\!\mathrm{d}}

\DeclareMathOperator{\spt}{supp}

\DeclareMathOperator*{\esssup}{ess\,sup}

\iffalse
    \usepackage{refcheck}
    \def\crefname#1#2#3{}
    \def\cref#1{\{\getrefs#1,\relax\}}
    \def\getrefs#1,#2\relax{%
        \ref{#1}%
        \ifx\relax#2\relax\else
            , \getrefs#2\relax
        \fi
    }
\else
\usepackage[capitalise]{cleveref} 
\fi

\title{\TITLE}
\author{Zden\v ek Mihula}
\address{Zden\v ek Mihula, Czech Technical University in Prague, Faculty of Electrical Engineering, Department of Mathematics, Technick\'a~2, 166~27 Praha~6, Czech Republic}
\email{mihulzde@fel.cvut.cz}
\urladdr{\href{https://orcid.org/0000-0001-6962-7635}{0000-0001-6962-7635}}

\begin{document}
\setcitestyle{numbers,square,aysep={,},notesep={, },citesep={,}}

\subjclass[2020]{46E35, 46E30}
\keywords{Sobolev spaces, Sobolev embeddings, compactness, rearrangement-invariant spaces, radial symmetry, optimal spaces}
\thanks{This research was partly supported by grant no.~23-04720S of the Czech Science Foundation}

\begin{abstract}
We provide a complete characterization of compactness of Sobolev embeddings of radially symmetric functions on the entire space $\rn$ in the general framework of rearrangement-invariant function spaces. We avoid any unnecessary restrictions and cover also embeddings of higher order, providing a complete picture within this framework. To achieve this, we need to develop new techniques because the usual techniques used in the study of compactness of Sobolev embeddings in the general framework of rearrangement-invariant function spaces are limited to domains of finite measure, which is essential for them to work. Furthermore, we also study certain weighted Sobolev embeddings of radially symmetric functions on balls, where the weight is a nonnegative power of the distance from the origin. We completely characterize their compactness and also describe optimal target rearrangement-invariant function spaces in these weighted Sobolev embeddings.
\end{abstract}

\maketitle


\section{Introduction}
It is notoriously known that Sobolev embeddings on unbounded domains are often noncompact (see~\cite[Chapter~6]{AFbook} for a nice introduction to this problem). This hindrance often makes analysis of partial differential equations or variational problems on unbounded domains very difficult. In particular, while the usual Sobolev space $W^{1,p}(\rn)$ is (continuously) embedded into the Lebesgue space $L^q(\rn)$ for every $q\in[p, np/(n-p)]$ when $p\in[1, n)$; and for every $q\in[p, \infty)$ when $p\geq n$, the Sobolev embedding
\begin{equation}\label{intro:Sobolev_emb_Lp_Lq_unrestricted}
W^{1,p}(\rn) \hookrightarrow L^q(\rn)
\end{equation}
is never compact. There are different phenomena causing this, depending on the values of $p$ and $q$, but one is common and always presented. Since both the Sobolev and the $L^q$ norms are translation invariant, by considering shifted copies of a nontrivial function $u\in W^{1,p}(\rn)$, we immediately see that the embedding cannot be compact. Loosely speaking, the mass can escape to infinity.

This obstacle was in particular faced by Strauss in his seminal paper \cite{S:77}, where he proved the existence of so-called solitary waves in the nonlinear Klein--Gordon equation in higher dimensions. A key step in his argument was to establish a certain compactness result for the Sobolev embedding \eqref{intro:Sobolev_emb_Lp_Lq_unrestricted} with $p=2$ but with $W^{1,p}(\rn)$ replaced by its subspace $W_R^{1,p}(\rn)$ consisting of functions with radial symmetry (cf.~\cite{BL:80, BL:83, CGM:78, EL:83}). More specifically, his argument yields that the Sobolev embedding
\begin{equation}\label{intro:Sobolev_emb_Lp_Lq_restricted}
W_R^{1,p}(\rn) \hookrightarrow L^q(\rn)
\end{equation}
with $p = 2$ is compact for every $q\in(p, np/(n-p))$, with $np/(n-p)$ to be interpreted as $\infty$ when $p\geq n$. This was later extended, among other things, to all values of $p\in[1, \infty)$ in \cite{L:82}. Before we proceed, let us briefly turn our attention to the restriction $q\in(p, np/(n-p))$. While the restriction $q \neq np/(n-p)$ is probably not surprising and is due to the usual concentration phenomenon, the restriction $q \neq p$ might be slightly surprising for some readers, who are not so familiar with Sobolev embeddings on the entire space. It is because of the vanishing phenomenon, which is in a sense an opposite phenomenon to the concentration one (see~\cite{L:84, L:85, S:95} and references therein for more information).

To establish the compactness of \eqref{intro:Sobolev_emb_Lp_Lq_restricted}, both papers \cite{S:77} and \cite{L:82} rely on a suitable ``radial lemma''. Its more general version from the latter paper reads as
\begin{equation}\label{intro:Lions_radial_lemma}
	|u(x)| \leq C_{n,p} \|u\|_{L^p(\rn)}^\frac{p-1}{p} \|\nabla u\|_{L^p(\rn)}^\frac1{p} |x|^{-\frac{n-1}{p}} \quad \text{for a.e.\ $x\in\rn$}
\end{equation}
and for every $u\in W_R^{1,p}(\rn)$, where $C_{n,p}$ is a positive constant, depending on $n$ and $p$, independent of $u$. In particular, this estimate, which does not in general hold for a function from $W^{1,p}(\rn)$, shows that the symmetry prevents the mass from escaping. The rest of the argument goes as follows. The radial lemma is exploited to establish that
\begin{equation}\label{intro:Lions_radial_lemma_norm_estimate}
\|u\|_{L^q(\rn\setminus B_R)} \leq \frac{C}{R^{(n-1)(1/p - 1/q)}} \|u\|_{W^{1,p}(\rn)} \quad \text{for every $u\in W_R^{1,p}(\rn)$}
\end{equation}
and for every $R>0$, where $B_R$ is the ball with radius $R$ centered at the origin. Note that the exponent $(n-1)(1/p - 1/q)$ is positive when $p < q \leq \infty$\textemdash here and in the rest, we implictly assume that $n\geq2$. Hence, given a bounded set $M\subseteq W_R^{1,p}(\rn)$, we see that
\begin{equation}\label{intro:uniform_decay_at_infinity_Lq_norms}
\lim_{R\to \infty} \sup_{u\in M} \|u\|_{L^q(\rn\setminus B_R)} = 0,
\end{equation}
provided that $p<q\leq\infty$. For $p\in[1, n]$, the compactness of \eqref{intro:Sobolev_emb_Lp_Lq_restricted} is then established by combining \eqref{intro:uniform_decay_at_infinity_Lq_norms} with the classical result that the Sobolev embedding $W^{1,p}(B_R) \hookrightarrow L^q(B_R)$ is compact if (and only if) $q\in[1, np/(n-p))$\textemdash this is where the restriction $q < np/(n-p)$ comes in (when $p>n$, the argument is the same, but we may also include $q=\infty$). Since the radial symmetry of functions appears (or ``can be forced'') in a surprisingly large number of problems and various compactness arguments are essential tools (it is virtually impossible to provide an exhaustive list of references, but the interested reader is referred to, e.g., \cite{A:24a, A:24b, D_FY:98, GdFMdSM:11, MV_S:15} and references therein), it is of interest to have at our disposal sharp compactness results for Sobolev spaces of radially symmetric functions. We note that compact embeddings of radially symmetric functions in function spaces measuring (integer or fractional) smoothness have been intensively studied also from quantitative points of view. The quality of compactness can be described by the rate of decay of various quantities, such as the approximation or entropy numbers. The interested reader is referred to \cite{ET:96, KLSS:03} and references therein for more information on this direction of research.

The principle result of this paper is a complete characterization of compactness of the Sobolev embedding
\begin{equation}\label{intro:Sobolev_emb_restricted}
\wmrX\hookrightarrow Y(\rn),
\end{equation}
where $X(\rn)$ and $Y(\rn)$ are \riSps{W} and $\wmrX$ is a Sobolev space of $m$th order, $m\in\N$, built upon $X(\rn)$ (see~\cref{sec:prel} for precise definitions). \RiSps{w} constitute a large class of function spaces whose norms are invariant with respect to symmetrizations preserving the measure of level sets of functions. Well-known examples of \riSps{w} are Lebesgue spaces, Lorentz spaces, or Orlicz spaces. In fact, this class of function spaces is rather wide and, besides these well-known types of function spaces, contains also more refine function spaces that may come handy when one deals with in a sense limiting situations. For example, it contains function spaces, sometimes called spaces of Brezis-Wainger-type, which appear in the study of sharp integrability properties of Sobolev functions in limiting Sobolev embeddings (\cite{BW:80, H:79}, see also~\cite{M:73}) and which improve limiting Sobolev embeddings into exponential Orlicz spaces (\cite{S:72, T:67}, see also~\cite{P:65,Y:61}).

It should be pointed out that our characterization of compactness of the Sobolev embedding \eqref{intro:Sobolev_emb_restricted}, see~\cref{thm:characterization_of_comp_emb_Y} below, does not impose any restrictions on the spaces $X$ or $Y$, apart from their rearrangement invariance. Therefore, it provides us with a complete picture within this framework. Moreover, it also covers embeddings of an arbitrary integer order. However, to obtain a complete picture, it is not sufficient to merely generalize the pointwise estimate~\eqref{intro:Lions_radial_lemma}. Instead, we need to develop new techniques and use subtler arguments to obtain sharp results. 

There are at least two reasons why a pointwise radial lemma is not suitable for our purposes. First, when dealing with a general function norm, there is no ``exponent'' inherently associated with it and ``playing with powers'' is not at our disposal. In fact, it is not clear in the first place what a suitable generalization of \eqref{intro:Lions_radial_lemma} should look like in this general setting. Second, the estimate is pointwise. Now, while pointwise estimates are often useful (as~\eqref{intro:Lions_radial_lemma} undoubtedly has been) and usually pleasant to work with, they may be too rough when one strives for sharp, optimal results even in a sense limiting situations in problems involving norms. Since establishing compactness means establishing norm convergence, we focus on establishing suitable norm inequalities from the very beginning, mixing new ideas with the contemporary theory of function spaces.

Our characterization of compactness of the Sobolev embedding~\eqref{intro:Sobolev_emb_restricted} reads as:
\begin{theorem}\label{thm:characterization_of_comp_emb_Y}
Let $m,n\in\N$, $n\geq2$. Let $X(\rn)$ and $Y(\rn)$ be \riSps{W}. The Sobolev embedding
\begin{equation}\label{thm:characterization_of_comp_emb_Y:eq:embedding}
\wmrX\hookrightarrow Y(\rn)
\end{equation}
is compact if and only if
\begin{equation}\label{thm:characterization_of_comp_emb_Y:eq:globally_ac}
\lim_{a\to\infty} \sup_{\|f\|_{\rep{X}}\leq 1} \|f^*\chi_{(a,\infty)}\|_{\rep{Y}} = 0
\end{equation}
and simultaneously either
\begin{equation}\label{thm:characterization_of_comp_emb_Y:eq:fund_Y_vanishes_0}
\lim_{t\to 0^+} \varphi_Y(t) = 0
\end{equation}
and one of the conditions
\begin{enumerate}[label=(A\arabic*), ref=(A\arabic*)]
	\item $m < n$ and
	\begin{equation}\label{thm:characterization_of_comp_emb_Y_locally_not_Linfty:eq:Hardy_op_compact}
		\lim_{a\to 0^+} \sup_{\|f\chi_{(0,1)}\|_{\rep{X}}\leq1} \Big\| \chi_{(0,a)}(t) \int_t^1 (f\chi_{(0,1)})^*(s) s^{-1 + \frac{m}{n}} \dd s \Big\|_{\rep{Y}} = 0;
	\end{equation}
	\item $m \geq n$
\end{enumerate}
is satisfied
or
\begin{equation}\label{thm:characterization_of_comp_emb_Y:eq:fund_Y_not_vanishing_at_0}
\lim_{t\to 0^+} \varphi_Y(t) > 0
\end{equation}
and one of the conditions
\begin{enumerate}[label=(B\arabic*), ref=(B\arabic*)]
	\item\label{enum:cond:m_less_than_n_Y_is_locally_Linfty}  $m < n$ and
	\begin{equation}\label{thm:characterization_of_comp_emb_Y_is_locally_Linfty:eq:Hardy_op_compact}
		\lim_{a\to 0^+} \sup_{\|f\chi_{(0,1)}\|_{\rep{X}}\leq1} \int_0^a (f\chi_{(0,1)})^*(s) s^{-1+\frac{m}{n}} \dd s = 0;
	\end{equation}
	\item  $m = n$ and
	\begin{equation}\label{thm:characterization_of_comp_emb_Y_is_locally_Linfty:eq:X_locally_not_L1}
	\lim_{t\to 0^+} \frac{t}{\varphi_X(t)} = 0;
	\end{equation}
	\item\label{enum:cond:m_bigger_than_n}  $m > n$
\end{enumerate}
is satisfied.
\end{theorem}
The notation is explained in detail in~\cref{sec:prel}. Here, we just mention that $f\mapsto f^*$ stands for the operation of nonincreasing rearrangement and $\fundX[Y]$ is the fundamental function of the space $Y$ (i.e., the norm of a characteristic function of a set with prescribed measure). For example, if $Y=L^q$, $q\in[1, \infty]$, then $\fundX[Y](t) = t^{1/q}$. Loosely speaking, \eqref{thm:characterization_of_comp_emb_Y:eq:fund_Y_vanishes_0} is true when $Y$ is not ``locally $L^\infty$'', whereas \eqref{thm:characterization_of_comp_emb_Y:eq:fund_Y_not_vanishing_at_0} is the case otherwise (cf.~\cite[Theorem~5.2]{S:12}). In particular, \cref{thm:characterization_of_comp_emb_Y} also characterizes when $\wmrX$ is compactly embedded into $L^\infty(\rn)$.

Important as general theorems are, concrete examples are often (at least) as valuable. Therefore, \cref{thm:characterization_of_comp_emb_Y} is also accompanied with a large number of concrete examples, which are provided in~\cref{sec:examples}. In particular, we provide a complete characterization of the compactness of the Sobolev embedding \eqref{intro:Sobolev_emb_restricted} when both $X$ and $Y$ are so--called Lorentz--Zygmund spaces (see~\cref{thm:compactness_entire_space_LZ_examples}, also~\cref{thm:compactness_entire_space_LLogL_examples}). The class of Lorentz--Zygmund spaces $L^{p,q,(\alpha_0,\alpha_\infty)}(\rn)$, $p,q\in[1,\infty]$, $(\alpha_0, \alpha_\infty)\in\R^2$, is a fine-grained scale of function spaces, which encompasses several different types of customary function spaces, but at the same time it is rather pleasant to work with. In particular, it contains Lebesgue spaces ($L^p(\rn) = L^{p,p,(0, 0)}(\rn)$), Lorentz spaces ($L^{p,q}(\rn) = L^{p,q,(0, 0)}(\rn)$) and some Orlicz spaces\textemdash in particular, those of ``logarithmic type'' (when $p=q < \infty$) and those of ``exponential type'' (when $p=q=\infty$). Using such a fine-grained scale of function spaces, we can also capture delicate limiting situations\textemdash such as when $X$ and $Y$ are ``close to each other'', when $X$ is ``close to $L^\frac{n}{m}$'', or when $Y$ is ``close to $L^\infty$''.

The main novelty in \cref{thm:characterization_of_comp_emb_Y}, besides the characterization itself, is the presence of the condition~\eqref{thm:characterization_of_comp_emb_Y:eq:globally_ac}. We will explain its role in more detail soon, but let us first focus on the conditions \eqref{thm:characterization_of_comp_emb_Y_locally_not_Linfty:eq:Hardy_op_compact} and \eqref{thm:characterization_of_comp_emb_Y_is_locally_Linfty:eq:Hardy_op_compact}. When suitably restricted to function spaces over finite measure spaces, they were shown in \cite{KP:08} (cf.~\cite{CR:07, P:06}) to be equivalent to the compactness of the Sobolev embedding
\begin{equation}\label{intro:Sobolev_emb_on_domains_unrestricted}
W^m X(\Omega) \hookrightarrow Y(\Omega),
\end{equation}
where $\Omega$ is a bounded regular domain in $\rn$, depending on whether $Y\neq L^\infty$ or not.  Their suitable restrictions to function spaces over finite measure spaces can be equivalently reformulated in terms of so-called almost compact embeddings between function spaces, a relation that we denote by $\ac$ (see~\cref{sec:prel} for the precise definition and more information). For example, when $Y\neq L^\infty$, the compactness of \eqref{intro:Sobolev_emb_on_domains_unrestricted} is equivalent to the fact that
\begin{equation}\label{intro:Sobolev_emb_on_domains_optimal_into_target}
Y_X(\Omega) \ac Y(\Omega),
\end{equation}
where $Y_X(\Omega)$ is the optimal rearrangement-invariant target function space in~\eqref{intro:Sobolev_emb_on_domains_unrestricted} (i.e., the smallest rearrangement-invariant target space in \eqref{intro:Sobolev_emb_on_domains_unrestricted} with which the embedding is still valid), which is described in \cite{KP:06} (cf.~\cite{EKP:00}). Loosely speaking, the almost compact embedding \eqref{intro:Sobolev_emb_on_domains_optimal_into_target} ensures that the norm of $Y(\Omega)$ is (locally) sufficiently weaker than the optimal one. For example, if $X(\Omega)=L^p(\Omega)$, $p\in[1, n/m)$, then $Y_X(\Omega) = L^{np/(n-mp),p}(\Omega)$, a Lorentz space strictly smaller than the Lebesgue space $L^{np/(n-mp)}(\Omega)$ (see~\cite{On:63, P:66, T:98}), and \eqref{intro:Sobolev_emb_on_domains_optimal_into_target} is true with $Y(\Omega) = L^q(\Omega)$ if and only if $q < np/(n-mp)$. Optimal function spaces and almost compact embeddings have also been successfully used to characterize compactness of other Sobolev embeddings\textemdash in particular, see~\cite{S:15} (Sobolev embeddings on domains equipped with probabilistic measures related to their isoperimetric profile) and~\cite{CM:19, CM:22} (Sobolev trace embeddings with respect to various measures). However, in all these papers the underlying domain is either bounded or at least of finite measure. In fact, there are no almost compact embeddings between \riSps{W} over nonatomic measure spaces of infinite measure (see~\cite[Theorem~4.5]{S:12}).

Let us finally turn our attention to the role of the condition~\eqref{thm:characterization_of_comp_emb_Y:eq:globally_ac}. Loosely speaking, it ensures that the norm of $Y(\rn)$ is ``globally sufficiently weaker'' than that of $X(\rn)$. For example, if $X(\rn) = L^p(\rn)$, $p\in[1, \infty]$, then~\eqref{thm:characterization_of_comp_emb_Y:eq:globally_ac} is satisfied with $Y(\rn) = L^q(\rn)$ if and only if $p<q\leq \infty$. Its role should be more clear from the follow proposition, which plays an important role in the proof of the sufficiency part of \cref{thm:characterization_of_comp_emb_Y}. It is also illustrative to compare the proposition with \eqref{intro:Lions_radial_lemma_norm_estimate}.

\begin{proposition}\label{prop:bounded_set_in_WmrX_tail_is_small}
Let $n\in\N$, $n\geq2$. Let $X(\rn)$ and $Y(\rn)$ be \riSps{W}. There is a constant $C>0$, depending only on the dimension $n$, such that
\begin{equation}\label{prop:bounded_set_in_WmrX_tail_is_small:quantitative_estimate}
\|u\chi_{\rn\setminus B_R}\|_{Y(\rn)} \leq C \Big( \sup_{\|f\|_{\rep{X}}\leq 1} \|f^*\chi_{(R^{n-1},\infty)}\|_{\rep{Y}} \Big) \|u\|_{W^1X(\rn)}
\end{equation}
for every $u\in \wmrX[1]$. In particular, if
\begin{equation}\label{prop:bounded_set_in_WmrX_tail_is_small:globally_ac}
\lim_{a\to\infty} \sup_{\|f\|_{\rep{X}}\leq 1} \|f^*\chi_{(a,\infty)}\|_{\rep{Y}} = 0,
\end{equation}
then
\begin{equation}\label{prop:bounded_set_in_WmrX_tail_is_small:tail_vanishes}
\lim_{R\to\infty} \sup_{u\in M} \|u\chi_{\rn\setminus B_R}\|_{Y(\rn)} = 0
\end{equation}
for every bounded set $M$ in $\wmrX[1]$.
\end{proposition}
Noteworthily, there is a simple geometric idea behind this proposition. Consider a radially symmetric function $u(x) = g(|x|)$ on $\rn$. When the size of $g$ is bigger than some $\lambda > 0$ on an interval $(R, R+\delta)$, $\delta>0$, then so is that of $u$ in the spherical shell $B_{R+\delta}\setminus B_{R}$. Now, since we assume that the dimension of $\rn$ is greater than $1$, for a fixed $\delta>0$, the ratio
\begin{equation*}
	\frac{|B_{R+\delta}\setminus B_{R}|}{|(R, R+\delta)|} = \frac{(R+\delta)^n - R^n}{\delta}\omega_n
\end{equation*}
goes to infinity as $R\to \infty$. Another possible point of view is that the maximum number of disjoint balls with radius $\delta>0$ centered on the sphere $\{x\in\rn: |x| = R\}$ goes to infinity as $R\to \infty$ (cf.~\cite{FKM:21}). These simple observations together with the fact that a~\riNorm{w} of a function depends only on the measure of its level sets are essential in our proof of~\cref{prop:bounded_set_in_WmrX_tail_is_small}.

Besides enabling us to recover compactness,  there is another important feature that the presence of radial symmetry brings in. It allows us to improve some weighted Sobolev embeddings on balls. Whereas the classical Sobolev embedding
\begin{equation}\label{intro:Sobolev_emb_Lp_Lq_ball_not_restricted}
W^{m,p}(B_R) \hookrightarrow L^q(B_R),
\end{equation}
where $p\in[1, n/m)$, holds if and only if $q\in[1, np/(n-mp)]$; and is compact if and only if $q\in[1, np/(n-mp))$, the weighted Sobolev embedding
\begin{equation}\label{intro:weighted_Sobolev_emb_Lp_Lq_restricted_ball}
W_R^{m,p}(B_R) \hookrightarrow L^q(B_R, \aldx)
\end{equation}
for radially symmetric functions, where $\dd \aldx(x) = |x|^\alpha \dd x$ with $\alpha>0$, holds if (and only if) $q\in[1, (n+\alpha)p/(n-mp)]$; and is compact if (and only if) $q\in[1, (n+\alpha)p/(n-mp))$ (see~\cite{GdFMdSM:11} and references therein). Note that the restriction to radially symmetric functions in~\eqref{intro:weighted_Sobolev_emb_Lp_Lq_restricted_ball} is essential\textemdash the critical exponent for the unrestricted weighted Sobolev embedding $W^{m,p}(B_R) \hookrightarrow L^q(B_R, \aldx)$ is still $np/(n-mp)$. This fact that \eqref{intro:weighted_Sobolev_emb_Lp_Lq_restricted_ball} is compact with bigger exponents $q$ than \eqref{intro:Sobolev_emb_Lp_Lq_ball_not_restricted} has been successfully used to deal with ``nonlinearities of higher order'' in the analysis of some partial differential equations\textemdash especially in analysis of the H\'{e}non equation (in particular, see~\cite{GS:08, GdFMdSM:11, N:82} and references therein). Motivated by this, we also study the weighted Sobolev embedding
\begin{equation}\label{intro:embedding_on_balls_Sob_emb}
		\wmrXB \hookrightarrow Y(B_R, \aldx)
\end{equation}
in the general framework of \riSps{W}. We note that the embedding \eqref{intro:embedding_on_balls_Sob_emb} does not involve any boundary conditions.

As with~\eqref{intro:Sobolev_emb_restricted}, our main goal is to characterize when the embedding \eqref{intro:embedding_on_balls_Sob_emb} is compact. This basically breaks down into two steps, each of independent interest. We first characterize what the optimal target \riSp{W} in \eqref{intro:embedding_on_balls_Sob_emb} is (see~\cref{cor:embedding_on_balls_optimal_target}) and then use it to characterize the compactness of \eqref{intro:embedding_on_balls_Sob_emb} (see~\cref{thm:embedding_on_balls_compactness_not_Linfty}). Unlike before, this is by now a standard argument, successfully used several times when dealing with Sobolev embeddings over finite measure spaces (recall our discussion about \eqref{intro:Sobolev_emb_on_domains_unrestricted} and \eqref{intro:Sobolev_emb_on_domains_optimal_into_target}). Nevertheless, although we do not need to develop a new approach to tackle \eqref{intro:embedding_on_balls_Sob_emb}, we still need to carefully supply suitably modified existing arguments with new ones to completely exploit the presence of radial symmetry, which, especially when $m=1$, essentially reduces \eqref{intro:embedding_on_balls_Sob_emb} to a one-dimensional weighted problem (the interested reader is referred to \cite[Chapter~1]{Mabook} for some classical results in this direction). Furthermore, even though the study of optimal \riSps{W} in Sobolev embeddings has been very active for a long time and by now covers a large number of different settings (e.g., see~\cite{ACPS:18, BC:21, BCS:25, CP:09, CP:16, CPS:15, CPS:20, CMP:23, D:25Online, EKP:00, KP:06, M:21} and references therein), the existing results do not cover the weighted Sobolev embedding~\eqref{intro:embedding_on_balls_Sob_emb}, to the best of the author's knowledge. Our general theorems concerning \eqref{intro:embedding_on_balls_Sob_emb} are also accompanied by concrete examples, which are provided in~\cref{sec:examples}. Finally, we note that (unrestricted) weighted or double weighted Sobolev embeddings with weights that are powers of the distance from the origin are closely connected to Caffarelli-Kohn-Nirenberg inequalities and the related (double) weighted isoperimetric problem. The interested reader is referred to \cite{ABCMP:17, Mc:23} and references therein for more information.

The paper is structured as follows. In~\cref{sec:prel}, we fix notation and summarize some basic theory that we will use later to make this paper self-contained for the most part. \cref{sec:functions_on_entire_space} contains auxiliary results, some of which may be of independent interest, which are later used in the proof of \cref{thm:characterization_of_comp_emb_Y}. It also contains the proof of \cref{prop:bounded_set_in_WmrX_tail_is_small}. In \cref{sec:weighted_Sobolev_on_balls}, we study the weighted Sobolev embedding \eqref{intro:embedding_on_balls_Sob_emb}. \cref{sec:proof_of_main_result} contains the proof of \cref{thm:characterization_of_comp_emb_Y}. Finally, \cref{sec:examples} contains concrete examples concerning both \eqref{intro:Sobolev_emb_restricted} and \eqref{intro:embedding_on_balls_Sob_emb}.


\section{Preliminaries}\label{sec:prel}
Throughout the entire section, $(\RR, \mu)$ is a $\sigma$-finite nonatomic nontrivial measure space. The set of all $\mu$-measurable functions on $(\RR, \mu)$ is denoted by $\M(\RR, \mu)$. We also denote by $\Mpl(\RR, \mu)$ the set of all those functions from $\M(\RR, \mu)$ that are nonnegative $\mu$-a.e.\ in $\RR$. In the following sections, $(\RR, \mu)$ will be either an interval (endowed with the Lebesgue measure) or (a subset of) $\rn$ endowed with a measure that is absolutely continuous with respect to the Lebesgue measure. In the entire paper, we adhere to the convention that $0\cdot\infty=0$.

\subsection{Rearrangements and r.i.\ spaces}

The \emph{distribution function} $f_*\colon (0, \infty) \to [0, \infty]$ of a function $f\in \M(\RR,\mu)$  is defined as
\begin{equation*}
f_*(\lambda) = \mu(\{x\in\RR: |f(x)| > \lambda\}),\ \lambda\in(0, \infty).
\end{equation*}
We say that two functions $f\in \M(\RR,\mu)$ and $g\in \M(S,\nu)$, where $(S,\nu)$ is a possibly different measure space, are \emph{equimeasurable} if $f_* = g_*$. The \textit{nonincreasing rearrangement} $f^* \colon  (0,\infty) \to [0, \infty ]$ of $f$  is
defined as
\begin{equation*}
f^*(t)=\inf\{\lambda\in(0,\infty) : f_*(\lambda)\leq t\},\ t\in(0,\infty).
\end{equation*}
The functions $f$ and $f^*$ are equimeasurable. The function $f^*$ is nonincreasing and right continuous. We have $\esssup_{x\in\RR}|f(x)| = f^*(0^+)$ and $f^*(t) = 0$ for every $t\geq \mu(\{x\in\RR: f(x) \neq 0\})$. If $f=\chi_E$ for some $\mu$-measurable $E\subseteq\RR$, then $f^* = \chi_{(0, \mu(E))}$. If $f$ is a nonnegative nonincreasing function on $(0, \infty)$, then $f = f^*$, up to a countable set of jumps.  If $|f|\leq |g|$ $\mu$-a.e.\ in $\RR$, then $f^*\leq g^*$. Two functions are equimeasurable if and only if their nonincreasing rearrangements coincide. We also define the \emph{maximal nonincreasing rearrangement} $f^{**} \colon (0,\infty) \to [0, \infty ]$ of $f\in \M(R,\mu)$ as
\begin{equation*}
f^{**}(t)=\frac1{t} \int_0^ t f^{*}(s) \dd s,\ t\in(0,\infty).
\end{equation*}
We always have $f^*\leq f^{**}$. The function $f^{**}$ is nonincreasing and continuous, but unlike $f^*$, $f^{**}$ need not be equimeasurable with $f$.

We say that a functional $\|\cdot\|_{X(\RR, \mu)}\colon \Mpl(\RR, \mu) \to [0, \infty]$ is a \emph{\riNorm{w}} (we will write \emph{an \riNorm{}}) if for all $f,g,f_k\in\Mpl(\RR, \mu)$, $k\in\N$, and $\lambda\geq0$:
\begin{enumerate}[label=(P\arabic*), ref=(P\arabic*)]
	\item $\|f + g\|_{X(\RR, \mu)} \leq \|f\|_{X(\RR, \mu)} + \|g\|_{X(\RR, \mu)}$; $\|f\|_{X(\RR, \mu)} = 0$ if and only if $f=0$ $\mu$-a.e.; $\|\lambda f\|_{X(\RR, \mu)} = \lambda \|f\|_{X(\RR, \mu)}$;
	\item $\|f\|_{X(\RR, \mu)} \leq \|g\|_{X(\RR, \mu)}$ whenever $f\leq g$ $\mu$-a.e.;
	\item $\|f_k\|_{X(\RR, \mu)} \nearrow \|f\|_{X(\RR, \mu)}$ whenever $f_k\nearrow f$ pointwise $\mu$-a.e.;
	\item $\|\chi_E\|_{X(\RR, \mu)} < \infty$ whenever $\mu(E) < \infty$;
	\item $\|f\chi_{E}\|_{L^1(\RR, \mu)}\leq C_{E,X} \|f\|_{X(\RR, \mu)}$ whenever $\mu(E) < \infty$, where $C_{E,X}\in(0, \infty)$ is a constant that is independent of $f$ (but it may depend on $E$ or $\|\cdot\|_{X(\RR, \mu)}$);
	\item $\|f\|_{X(\RR, \mu)} = \|g\|_{X(\RR, \mu)}$ whenever $f$ and $g$ are equimeasurable.
\end{enumerate}

Given an \riNorm{} $\|\cdot\|_{X(\RR, \mu)}$, we extend it to all functions $f\in\M(\RR,\mu)$ by
\begin{equation*}
\|f\|_{X(\RR, \mu)} = \|\,|f|\,\|_{X(\RR, \mu)},\ f\in\M(\RR,\mu).
\end{equation*}

The function space $X(\RR, \mu)$ defined as
\begin{equation*}
X(\RR, \mu) = \{f\in\M(\RR,\mu): \|f\|_{X(\RR, \mu)} < \infty\}
\end{equation*}
endowed with the functional $\|\cdot\|_{X(\RR, \mu)}$ is called a \emph{\riSp{W}} (we will write \emph{an \riSp{}}). It is a Banach space (in particular, $\|\cdot\|_{X(\RR, \mu)}$ is a norm on $X(\RR, \mu)$), and the functions from $X(\RR, \mu)$ are finite $\mu$-a.e. An \riSp{} always contains simple functions. By a simple function, we mean a (finite) linear combination of characteristic functions of sets of finite measure.

Textbook examples of \riSps{} are the usual $L^p(\RR, \mu)$ spaces (i.e., \emph{Lebesgue spaces}), $p\in[1, \infty]$, Lorentz spaces, or Orlicz spaces. The rearrangement invariance of the usual Lebesgue norm $\|\cdot\|_{L^p(\RR, \mu)}$, $p\in[1, \infty]$, follows from the layer cake formula (e.g., \cite[Chapter~2, Proposition~1.8]{BS} or \cite[Theorem~1.13]{LL:01}), which tells us that
\begin{equation*}
\|f\|_{L^p(\RR,\mu)} = \|f^*\|_{L^p(0, \infty)} \quad \text{for every $f\in\M(\RR, \mu)$}.
\end{equation*}
\emph{Lorentz spaces} $L^{p, q}(\RR, \mu)$ are an important generalization of Lebesgue spaces, where either $p\in(1, \infty)$ and $q\in[1, \infty]$ or $p = q = 1$ or $p = q = \infty$. The corresponding \riNorm{} $\|\cdot\|_{L^{p,q}(\RR, \mu)}$ is defined as
\begin{equation*}
\|f\|_{L^{p,q}(\RR,\mu)} = \|t^{\frac1{p} - \frac1{q}}f^*(t)\|_{L^q(0, \infty)},\ f\in\Mpl(\RR,\mu).
\end{equation*}
However, one needs to be more careful here. The functional $\|\cdot\|_{L^{p,q}(\RR,\mu)}$ is not an \riNorm{} when $1< p < q \leq\infty$, because it is not subadditive. When $1< p < q \leq\infty$, the functional $\|\cdot\|_{L^{p,q}(\RR,\mu)}$ is merely equivalent to an \riNorm{}. More precisely, the functional
\begin{equation*}
\|f\|_{L^{(p,q)}(\RR,\mu)} = \|f^{**}\chi_{(0, \mu(\RR))}\|_{L^{p,q}(0, \infty)},\ f\in\Mpl(\RR,\mu),
\end{equation*}
is an \riNorm{}, and there are positive constants $C_1$ and $C_2$ such that
\begin{equation*}
C_1\|f\|_{L^{(p,q)}(\RR,\mu)} \leq \|f\|_{L^{p,q}(\RR,\mu)} \leq C_2\|f\|_{L^{(p,q)}(\RR,\mu)} \quad \text{for every $f\in\Mpl(\RR,\mu)$},
\end{equation*}
provided that either $p\in(1, \infty)$ and $q\in[1 ,\infty]$ or $p=q=\infty$. The interested reader can find more information in \cite[Chapter~4, Section~4]{BS} or \cite{H:66}. In view of that, we will consider $L^{p,q}(\RR,\mu)$ an \riSp{} even when $1< p < q \leq\infty$. Note that
\begin{equation*}
\|\cdot\|_{L^p(\RR,\mu)} = \|\cdot\|_{L^{p,p}(\RR,\mu)} \quad \text{for every $p\in[1, \infty]$}.
\end{equation*}
Furthermore, when $p\in(1, \infty)$ and $1\leq q_1 < q_2\leq \infty$, we have
\begin{equation*}
L^{p, q_1}(\RR,\mu) \subsetneq L^{p, q_2}(\RR,\mu),
\end{equation*}
regardless of whether $\mu(\RR) < \infty$ or not. \emph{Orlicz spaces} $L^A(\RR, \mu)$ are another very important generalization of Lebesgue spaces. The corresponding \riNorm{} $\|\cdot\|_{L^A(\RR, \mu)}$ is defined as
\begin{equation*}
\|f\|_{L^A(\RR,\mu)} = \inf\Bigg\{ \lambda>0: \int_{\RR} A\Big( \frac{|f(x)|}{\lambda} \Big) \dd\mu(x) \leq 1 \Bigg\},\ f\in\Mpl(\RR,\mu),
\end{equation*}
where $A\colon[0, \infty] \to [0, \infty]$ is a Young function. A function $A\colon[0, \infty] \to [0, \infty]$ is called a Young function if it is convex, left-continuous, vanishing at $0$, and not constant on the entire interval $(0, \infty)$. For example, when $p\in[1, \infty)$, we have $\|\cdot\|_{L^p(\RR,\mu)} = \|\cdot\|_{L^A(\RR,\mu)}$ with $A(t) = t^p$, $t\geq0$. We also have $\|\cdot\|_{L^\infty(\RR,\mu)} = \|\cdot\|_{L^A(\RR,\mu)}$ with $A(t) = \infty\cdot\chi_{(1, \infty]}(t)$, $t\geq0$. Besides the classical textbooks \cite{BS,RR:91}, the interested reader can find more information on the contemporary theory of Orlicz spaces and in particular Orlicz--Sobolev spaces in~\cite{C:10, MPT:23}.

An analogue of Fatou's lemma is at our disposal in the framework of \riSps{}. More precisely, if $\M(\RR,\mu) \ni f_k \to f$ pointwise $\mu$-a.e., then
\begin{equation}\label{prel:Fatou_lemma}
\|f\|_{X(\RR, \mu)} \leq \liminf_{k\to\infty} \|f_k\|_{X(\RR, \mu)}.
\end{equation}

With any \riNorm{} $\|\cdot\|_{X(\RR, \mu)}$, there is associated another \riNorm{}, $\|\cdot\|_{X'(\RR, \mu)}$, defined for $g \in  \Mpl(\RR,\mu)$ as
\begin{equation}\label{prel:associate_norm}
\|g\|_{X'(\RR, \mu)}=\sup_{\|f\|_{X(\RR, \mu)}\leq1} \int_{\RR} |f(x)||g(x)| \dd \mu(x),\ g \in  \Mpl(\RR,\mu).
\end{equation}
The \riNorm{} $\|\cdot\|_{X'(\RR, \mu)}$ is called the \emph{associate norm} of $\|\cdot\|_{X(\RR, \mu)}$. The resulting \riSp{} $X'(\RR, \mu)$ is called the \emph{associate space} (of $X(\RR, \mu)$). The definition of $\|\cdot\|_{X'(\RR, \mu)}$ immediately gives us that the H\"older inequality
\begin{equation}\label{prel:Holder}
\int_{\RR} |f||g| \dd\mu \leq \|f\|_{X(\RR, \mu)} \|g\|_{X'(\RR,\mu)} \quad \text{for all $f,g\in\M(\RR,\mu)$}
\end{equation}
is true.

We always have $X''(\RR, \mu) = X(\RR, \mu)$ with equal norms, where $X''(\RR, \mu) = (X')'(\RR, \mu)$. In particular, we have
\begin{equation}\label{prel:double_associate_norm}
\|f\|_{X(\RR, \mu)} = \sup_{\|g\|_{X'(\RR, \mu)} \leq 1} \int_{\RR} |f(x)| |g(x)| \dd \mu(x) \quad \text{for every $f\in\M(\RR, \mu)$}.
\end{equation}
Moreover, we also have
\begin{equation}\label{prel:double_associate_norm_nonincreasing_functions}
\|f\|_{X(\RR, \mu)} = \sup_{\|g\|_{X'(\RR, \mu)} \leq 1} \int_0^\infty f^*(t) g^*(t) \dd t \quad \text{for every $f\in\M(\RR, \mu)$}.
\end{equation}
For example, if $X(\RR, \mu) = L^{p,q}(\RR, \mu)$, where either $p\in(1, \infty)$ and $q\in[1, \infty]$ or $p = q = 1$ or $p = q = \infty$, then $X'(\RR, \mu) = L^{p',q'}(\RR, \mu)$, up to equivalent norms (with the equivalence being equality when $p=q$).

The \emph{fundamental function} of an \riSp{} $X(\RR, \mu)$ is the function $\fundX\colon [0, \mu(\RR)) \to [0, \infty)$ defined as
\begin{equation*}
\fundX(t) = \|\chi_E\|_{X(\RR, \mu)},\ t\in[0, \mu(\RR)),
\end{equation*}
where $E\subseteq\RR$ is any $\mu$-measurable set with $\mu(E) = t$. The rearrangement invariance of $\|\cdot\|_{X(\RR, \mu)}$ ensures that the fundamental function is well defined. Furthermore, it is quasiconcave and satisfies
\begin{equation}\label{prel:fund_of_X_times_of_Xasoc}
\fundX(t) \fundX[X'](t) = t \quad \text{for every $t\in[0,\mu(\RR))$}.
\end{equation}
The quasiconcavity means that $\fundX(t) = 0$ if and only if $t = 0$, $\fundX$ is nondecreasing, and the function $(0,\infty)\ni t\mapsto \fundX(t)/t$ is nonincreasing. For example, if $X(\RR, \mu) = L^p(\RR, \mu)$, then $\fundX(t) = t^\frac1{p}$, $t\in(0,\mu(\RR))$. More generally, if $X(\RR, \mu) = L^{p,q}(\RR, \mu)$, where either $p\in(1, \infty)$ and $q\in[1, \infty]$ or $p = q = 1$ or $p = q = \infty$, then $\fundX$ coincides with $\fundX[L^p]$, up to a constant multiple depending only on $p$ and $q$.

Every \riSp{} $X(\RR, \mu)$ can be uniquely represented as an \riSp{} over an interval. When $\mu(\RR) = \infty$, there is an unique \riSp{} over $(0, \infty)$, denoted $\rep{X}$, satisfying
\begin{equation*}
\|f\|_{X(\RR, \mu)} = \|f^*\|_{\rep{X}} \quad \text{for every $f\in\M(\RR, \mu)$}.
\end{equation*}
Moreover, we have $\fundX = \fundXrep$. When $\mu(\RR) < \infty$, there is an unique \riSp{} over $(0, 1)$, denoted $\repFin{X}$, satisfying
\begin{equation*}
\|f\|_{X(\RR, \mu)} = \|f^*(\mu(\RR)t)\|_{\repFin{X}} \quad \text{for every $f\in\M(\RR, \mu)$}.
\end{equation*}
Moreover, we have $\fundX(t) = \fundXrep(t/\mu(\RR))$ for every $t\in(0, \mu(\RR))$. Note that the representation is over $(0,1)$, not $(0, \mu(\RR))$. The choice of the unit interval, no matter what the measure of $0<\mu(\RR)<\infty$ is, will make some statements and their proofs less technical. On the other hand, when $\mu(\RR) \neq 1$, the representation norm over $(0,1)$ may differ from the usual one over $(0, \mu(\RR))$ by a constant multiple depending on $\mu(\RR)$. In fact, if $Z(0,\mu(\RR))$ is the unique \riSp{} satisfying
\begin{equation*}
\|f\|_{X(\RR, \mu)} = \|f^*\|_{Z(0, \mu(\RR))} \quad \text{for every $f\in\M(\RR, \mu)$},
\end{equation*}
then
\begin{equation*}
\|g\|_{\repFin{X}} = \big\| g\big( t/\mu(R) \big) \big\|_{Z(0, \mu(\RR))} \quad \text{for every $g\in\M(0, 1)$}.
\end{equation*}
However, in view of the boundedness of dilation operators on \riSps{} (see~\eqref{prel:dilation_bounded} below), we see that this is immaterial for our purposes. For example, we have
\begin{equation*}
\|f\|_{L^{p,q}(\RR, \mu)} = \|f^*\|_{L^{p,q}(0, \mu(\RR))} \quad \text{for every $f\in\M(\RR, \mu)$}
\end{equation*}
or
\begin{equation*}
\|f\|_{L^A(\RR, \mu)} = \|f^*\|_{L^A(0, \mu(\RR))} \quad \text{for every $f\in\M(\RR, \mu)$}.
\end{equation*}

Let $L\in(0, \infty]$. For every $a>0$ and $f\in\M(0, L)$, the \emph{dilation operator} $D_a$ defined as
\begin{equation*}
D_af(t)= f\Big( \frac{t}{a} \Big)\chi_{(0, aL)}(t),\ f\in\Mpl(0, L),\ t\in(0, L),
\end{equation*}
is bounded on every \riSp{} $X(0, L)$. Moreover, we have
\begin{equation}\label{prel:dilation_bounded}
\|D_af\|_{X(0,L)} \leq \max\{1,a\} \|f\|_{X(0,L)}\quad\text{for every $f\in\M(0,L)$}.
\end{equation}

We say that a Banach space $A$ is \emph{embedded in} a Banach space $B$, and write $A\hookrightarrow B$, if $A\subseteq B$ (in the set-theoretic sense) and the inclusion is continuous\textemdash that is, there is a constant $C>0$ such that $\|f\|_B \leq C \|f\|_A$ for every $f\in A$. We say that $A$ and $B$ are equivalent, and write $A = B$, if $A\hookrightarrow B$ and $B\hookrightarrow A$. In other words, $A$ and $B$ coincide in the set-theoretic sense and their norms are equivalent. If $X(\RR, \mu)$ and $Y(\RR, \mu)$ are \riSps{}, then $X(\RR, \mu) \hookrightarrow Y(\RR, \mu)$ if and only if $X(\RR, \mu) \subseteq Y(\RR, \mu)$. Furthermore, we have
\begin{equation}\label{prel:embedding_iff_asoc_embedding}
X(\RR, \mu) \hookrightarrow Y(\RR, \mu) \qquad \text{if and only if} \qquad Y'(\RR, \mu) \hookrightarrow X'(\RR, \mu).
\end{equation}
We also have $X(\RR, \mu) \hookrightarrow Y(\RR, \mu)$ if and only if $\rep[0,L]{X} \hookrightarrow \rep[0,L]{Y}$, where $L$ is either $1$ or $\infty$ depending on whether $\mu(\RR)$ is finite or not. Every \riSp{} $X(\RR, \mu)$ lies somewhere between $(L^1 \cap L^\infty)(\RR,\mu)$ and $(L^1 + L^\infty)(\RR,\mu)$. More precisely, we have
\begin{equation*}
(L^1 \cap L^\infty)(\RR,\mu) \hookrightarrow X(\RR, \mu) \hookrightarrow (L^1 + L^\infty)(\RR,\mu).
\end{equation*}
Moreover, since $(L^1 \cap L^\infty)(\RR,\mu) = L^\infty(\RR,\mu)$ and $(L^1 + L^\infty)(\RR,\mu) = L^1(\RR,\mu)$ when $\mu(\RR) < \infty$, we have
\begin{equation}\label{prel:Linfty_emb_X_emb_L1_finite_measure}
L^\infty(\RR,\mu) \hookrightarrow X(\RR, \mu) \hookrightarrow L^1(\RR,\mu)
\end{equation}
provided that $\mu(\RR) < \infty$. The interested reader can find more information on the theory of \riSps{} in~\cite{BS}.

Besides $\hookrightarrow$, there is another important relation between function spaces, which is known to play a significant role in the study of compactness in \riSps{} (more generally, in Banach function spaces) when $\mu(\RR) < \infty$. Assume for the rest of this paragraph that $\mu(\RR) < \infty$. We say that an \riSp{} $X(\RR, \mu)$ is \emph{almost compactly embedded} in an \riSp{} $Y(\RR, \mu)$, and write $X(\RR, \mu) \ac Y(\RR, \mu)$, if
\begin{equation*}
\lim_{k\to\infty}\sup_{\|f\|_{X(\RR, \mu)\leq1}} \|f\chi_{E_k}\|_{Y(\RR, \mu)} = 0
\end{equation*}
for every sequence $\{E_k\}_{k = 1}^\infty\subseteq \RR$ of $\mu$-measurable sets such that $\chi_{E_k} \to 0$ pointwise $\mu$-a.e. In other words, bounded sets in $X(\RR, \mu)$ have uniformly absolutely continuous norm in $Y(\RR, \mu)$. Similarly to~\eqref{prel:embedding_iff_asoc_embedding}, we have
\begin{equation}\label{prel:ac_iff_asocAC}
X(\RR,\mu) \ac Y(\RR,\mu) \quad \text{if and only if} \quad Y'(\RR,\mu) \ac X'(\RR,\mu).
\end{equation}
Furthermore, we have
\begin{equation}\label{prel:ac_iff_repAC}
X(\RR,\mu) \ac Y(\RR,\mu) \quad \text{if and only if} \quad \lim_{a\to0^+} \sup_{\|f\|_{\repFin{X}}\leq1} \|f^*\chi_{(0,a)}\|_{\repFin{Y}}.
\end{equation}
For example, $L^{p,q}(\RR,\mu) \ac L^{r,s}(\RR,\mu)$ if and only if $p>r$. The interested reader is referred to \cite{S:12} (cf.~\cite{F-MMP:10, LZ:63}), where the relation $\ac$ is thoroughly studied.

We will sometimes need to work with restrictions of \riSps{}. If $X(\RR, \mu)$ is an \riSp{} and $E\subseteq \RR$ is a $\mu$-measurable set of positive measure, then the functional $\|\cdot\|_{Z(E, \mu)}$ defined as
\begin{equation*}
\| f \|_{Z(E, \mu)} = \|f\chi_E\|_{X(\RR, \mu)},\ f\in\Mpl(R, \mu),
\end{equation*}
is an \riNorm{} over $(E, \mu)$. We will denote the resulting \riSp{} $X(E, \mu)$. When $\mu(\RR) = \infty$ and $\mu(E) < \infty$, the representation norm of $Z(E, \mu) = X(E, \mu)$ and that of $X(\RR, \mu)$ satisfy
\begin{align}
\min\{\mu(E), 1\}\|f\chi_{(0,1)}\|_{\rep{X}}&\leq \|f\|_{\repFin{Z}} \nonumber\\
&\leq \max\{\mu(E), 1\}\|f\chi_{(0,1)}\|_{\rep{X}} \label{prel:representation_space_restriction}
\end{align}
for every $f\in\M(0,1)$.

In particular, we will be interested in the situation where $(\RR, \mu) = \rn$ (endowed with the Lebesgue measure), $n\geq2$, and $(E, \mu) = (B_R, \aldx)$, where $B_R\subseteq\rn$ is the (open) ball centered at the origin with radius $R\in(0, \infty)$ and $\aldx$ is the weighted measure
\begin{equation*}
\dd \aldx(x) = |x|^\alpha \dd x,
\end{equation*}
where $\alpha \in[0, \infty)$. The weighted measure of $B_R$ is
\begin{equation*}
\aldx(B_R) = \alBvol R^{n+\alpha},
\end{equation*}
where 
\begin{equation*}
\alBvol = \frac{n}{n+\alpha} \omN
\end{equation*}
and $\omN$ is the volume of the unit ball in $\rn$. Finally, the mapping $\rn\ni x \mapsto \alBvol |x|^{n+\alpha}$ is a measure-preserving transformation (in the sense of~\cite[Chapter~2, Section~7]{BS}) from $(\rn, \aldx)$ into $(0, \infty)$. In particular, we have
\begin{align}
\|f(\omN |x|^n)\|_{X(\rn)} &= \|f\|_{\rep{X}} \quad \text{for every $f\in\M(0,\infty)$} \label{prel:representation_measure_preserving_entire_space}
\intertext{and}
\Big\| f\Big( \Big( \frac{|x|}{R} \Big)^{n+\alpha} \Big) \Big\|_{X(B_R,\aldx)} &= \|f\|_{\repFin{X}} \quad \text{for every $f\in\M(0,1)$}. \label{prel:representation_measure_preserving_ball}
\end{align}

\subsection{Sobolev spaces}
Throughout the entire subsection, let $m,n\in\N$, $n\geq2$, and let $\Omega\subseteq\rn$ be an open set. When $u$ is a $m$-times weakly differentiable function on $\Omega$, we denote by $\nabla^k u$, $k\in\{0, 1,\dots, m\}$, the (arbitrarily arranged) vector of all weak derivatives of $u$ of order $k$, where $\nabla^0 u = u$. We also denote by $D^m u$ the vector of all weak derivatives of $u$ of order $k=1, \dots, m$. Let $X(\Omega)$ be an \riSp{}. The \emph{Sobolev space} $W^mX(\Omega)$ built upon $X(\Omega)$ is defined as
\begin{equation*}
W^mX(\Omega) = \{u\in X(\Omega): \text{$u$ is $m$-times weakly differentiable and $|D^m u|\in X(\Omega)$}\}.
\end{equation*}
We endow $W^mX(\Omega)$ with the norm defined as
\begin{equation*}
\|u\|_{W^mX(\Omega)} = \|u\|_{X(\Omega)} + \|\,|D^m u|\,\|_{X(\Omega)},\ u\in X(\Omega).
\end{equation*}
We will write $\|D^m u\|_{X(\Omega)}$ instead of $\|\,|D^m u|\,\|_{X(\Omega)}$. Similarly for $\nabla^k u$. Note that $W^m L^p(\Omega)$, where $p\in[1, \infty]$, is the usual Sobolev space $W^{m,p}(\Omega)$, up to equivalent norms.

Let $\Omega = B_R$, where $B_R$ is the open ball in $\rn$ centered at the origin with radius $R\in(0, \infty]$\textemdash when $R=\infty$, $B_R$ is to be interpreted as $\rn$. We denote by $\wmrXB$ the subspace of $W^mX(B_R)$ consisting of those functions from $W^mX(B_R)$ that are radially symmetric. If $u\in\wmrXB$, then there is a function $g\colon I_R \to \R$ such that (see~\cite{B:11, GdFMdSM:11})
\begin{equation*}
u(x) = g(|x|) \quad \text{for a.e.\ $x\in B_R$}
\end{equation*}
and $g^{(j)}\in\aclocInt{I_R}$ for $j=0,\dots,m-1$, where $I_R = (0,R] \cap \R$. By $\aclocInt{I_R}$, we denote the collection of locally absolutely continuous functions on the interval $I_R$. Furthermore, we have $|\nabla u(x)| = |g'(|x|)|$ for a.e.\ $x\in B_R$. As for higher order derivatives, it was shown in \cite[Theorem~2.2]{GdFMdSM:11} that
\begin{equation}\label{prel:ineq_between_norm_of_gradient_and_derivative_of_the_one_var_func}
|g^{(j)}(|x|)|\leq |\nabla^{j}u(x)| \quad \text{for every $j\in\{1,\dots,m\}$ and a.e.\ $x\in B_R$}.
\end{equation}


\section{Sobolev functions with radial symmetry on \texorpdfstring{$\rn$}{the entire space}}\label{sec:functions_on_entire_space}
In this section, we prove \cref{prop:bounded_set_in_WmrX_tail_is_small} and also equip ourselves with tools that we will need later to prove \cref{thm:characterization_of_comp_emb_Y}. We start with a proposition about averaging in \riSps{}.
\begin{proposition}\label{prop:averaging_lemma}
Let $M\in\N$ and $\{I_j\}_{j = 1}^{M+1}$ be a system of nonoverlapping bounded intervals in $(0, \infty)$, each of length $\delta>0$. For all \riSps{} $X(0, \infty)$ and $Y(0, \infty)$, we have
\begin{equation*}
\Big\| \sum_{j = 1}^M  \chi_{I_j}(t)\frac1{\delta} \int_{I_j\cup I_{j+1}} g \Big\|_{Y(0, \infty)} \leq 2 \Bigg( \frac{\varphi_Y(\delta)}{\varphi_X(\delta)} + \sup_{\|f\|_{X(0, \infty)\leq1}}\|f^*\chi_{(\delta, \infty)}\|_{Y(0, \infty)} \Bigg) \|g\|_{X(0, \infty)}
\end{equation*}
for every $g\in\Mpl(0, \infty)$.
\end{proposition}
\begin{proof}
Fix $g\in\Mpl(0, \infty)$ such that $\|g\|_{X(0, \infty)} < \infty$. In particular, $g\in L^1(I_j)$ for every $j\in\{1, \dots, M+1\}$. Since the intervals $\{I_j\}_{j = 1}^{M+1}$ are nonoverlapping and they have the same length, the simple functions $(0,\infty)\ni t \mapsto \sum_{j = 1}^M  \chi_{I_j}(t)\frac1{\delta} \int_{I_{j+1}} g$ and $(0,\infty)\ni t \mapsto \sum_{j = 1}^M  \chi_{I_{j+1}}(t)\frac1{\delta} \int_{I_{j+1}} g$ are equimeasurable. Hence
\begin{align}
\Big\| \sum_{j = 1}^M  \chi_{I_j}(t)\frac1{\delta} \int_{I_j\cup I_{j+1}} g \Big\|_{Y(0, \infty)} &\leq \Big\| \sum_{j = 1}^M  \chi_{I_j}(t)\frac1{\delta} \int_{I_j} g \Big\|_{Y(0, \infty)} + \Big\| \sum_{j = 1}^M  \chi_{I_j}(t)\frac1{\delta} \int_{I_{j+1}} g \Big\|_{Y(0, \infty)} \nonumber\\
&= \Big\| \sum_{j = 1}^M  \chi_{I_j}(t)\frac1{\delta} \int_{I_j} g \Big\|_{Y(0, \infty)} + \Big\| \sum_{j = 1}^M  \chi_{I_{j+1}}(t)\frac1{\delta} \int_{I_{j+1}} g \Big\|_{Y(0, \infty)} \nonumber\\
&\leq 2 \Big\| \sum_{j = 1}^{M + 1}  \chi_{I_j}(t)\frac1{\delta} \int_{I_j} g \Big\|_{Y(0, \infty)}. \label{prop:averaging_lemma:eq:1}
\end{align}
Let $\pi\colon \{1, \dots, M+1\} \to \{1, \dots, M+1\}$ be a permutation such that the finite sequence $\{\int_{I_{\pi(j)}} g\}_{j = 1}^{M+1}$ is nonincreasing.

Next, set
\begin{equation}\label{prop:averaging_lemma:eq:3}
h = \sum_{j = 1}^{M+1}  \chi_{I_j}(t)\frac1{\delta} \int_{I_j} g.
\end{equation}
Using the fact that the intervals are nonoverlaping and have the same length $\delta$ once more, we see that
\begin{equation*}
h^* = \chi_{(0, \delta)} \frac1{\delta} \int_{I_{\pi(1)}} g + \sum_{j = 2}^{M + 1} \chi_{[(j-1)\delta, j\delta)}(t) \frac1{\delta} \int_{I_{\pi(j)}} g.
\end{equation*}
Note that
\begin{align}
\Big\| \sum_{j = 1}^{M + 1}  \chi_{I_j}(t)\frac1{\delta} \int_{I_j} g \Big\|_{Y(0, \infty)} &= \|h\|_{Y(0, \infty)} = \|h^*\|_{Y(0, \infty)} \nonumber\\
&\leq \Big\| \chi_{(0, \delta)} \frac1{\delta} \int_{I_{\pi(1)}} g \Big\|_{Y(0, \infty)} + \|\chi_{(\delta, \infty)}h^*\|_{Y(0, \infty)}. \label{prop:averaging_lemma:eq:4}
\end{align}
As for the first term on the right-hand side, using the H\"older inequality~\eqref{prel:Holder} and \eqref{prel:fund_of_X_times_of_Xasoc}, we have
\begin{align}
\Big\| \chi_{(0, \delta)} \frac1{\delta} \int_{I_{\pi(1)}} g \Big\|_{Y(0, \infty)} &= \frac{\varphi_{Y}(\delta)}{\delta} \int_{I_{\pi(1)}} g \leq \frac{\varphi_{Y}(\delta)}{\delta} \|g\|_{X(0, \infty)} \varphi_{X'}(\delta) \nonumber\\
&= \frac{\varphi_{Y}(\delta)}{\varphi_{X}(\delta)} \|g\|_{X(0, \infty)}. \label{prop:averaging_lemma:eq:5}
\end{align}
As for the second term on the right-hand side of \eqref{prop:averaging_lemma:eq:4}, recalling \eqref{prop:averaging_lemma:eq:3}, we see that
\begin{align*}
\|\chi_{(\delta, \infty)}h^*\|_{Y(0, \infty)} &\leq \Big( \sup_{\|f\|_{X(0, \infty)\leq1}}\|f^*\chi_{(\delta, \infty)}\|_{Y(0, \infty)} \Big) \|h\|_{X(0, \infty)} \\
&= \Big( \sup_{\|f\|_{X(0, \infty)\leq1}}\|f^*\chi_{(\delta, \infty)}\|_{Y(0, \infty)} \Big) \Big\| \sum_{j = 1}^{M+1}  \chi_{I_j}(t)\frac1{\delta} \int_{I_j} g \Big\|_{X(0, \infty)}.
\end{align*}
Since the averaging operator $\M(0, \infty) \ni h \mapsto \sum_{j = 1}^{M+1}  \chi_{I_j}\frac1{\delta} \int_{I_j} h$ is bounded on every \riSp{} over $(0, \infty)$ and its norm is at most $1$ (see~\cite[Chapter~2, Theorem~4.8]{BS}), we have
\begin{equation*}
\Big\| \sum_{j = 1}^{M+1}  \chi_{I_j}(t)\frac1{\delta} \int_{I_j} g \Big\|_{X(0, \infty)} \leq \|g\|_{X(0, \infty)}.
\end{equation*}
Combining the last two inequalities together, we obtain
\begin{equation}\label{prop:averaging_lemma:eq:6}
\|\chi_{(\delta, \infty)}h^*\|_{Y(0, \infty)} \leq \Big( \sup_{\|f\|_{X(0, \infty)\leq1}}\|f^*\chi_{(\delta, \infty)}\|_{Y(0, \infty)} \Big) \|g\|_{X(0, \infty)}.
\end{equation}
Hence, thanks to \eqref{prop:averaging_lemma:eq:4}, \eqref{prop:averaging_lemma:eq:5}, and \eqref{prop:averaging_lemma:eq:6}, we arrive at
\begin{align*}
\Big\| \sum_{j = 1}^M  \chi_{I_j}(t)\frac1{\delta} \int_{I_j\cup I_{j+1}} g \Big\|_{Y(0, \infty)} &\leq \frac{\varphi_{Y}(\delta)}{\varphi_{X}(\delta)} \|g\|_{X(0, \infty)} \\
&\quad + \Big( \sup_{\|f\|_{X(0, \infty)\leq1}}\|f^*\chi_{(\delta, \infty)}\|_{Y(0, \infty)} \Big) \|g\|_{X(0, \infty)},
\end{align*}
whence the claim follows in view of \eqref{prop:averaging_lemma:eq:1}.
\end{proof}

The following simple proposition provides us with a useful inequality between the quantity $\sup_{\|f\|_{X(0, \infty)} \leq 1}\|f^*\chi_{(a, \infty)}\|_{Y(0, \infty)}$ and the ratio of the corresponding fundamental functions.
\begin{proposition}\label{prop:decay_at_infinity_implies_decay_of_fund_func}
Let $X(0, \infty)$ and $Y(0, \infty)$ be \riSps{}. Then
\begin{equation*}
\frac{\varphi_Y(a)}{\varphi_X(a)} \leq 2\sup_{\|f\|_{X(0, \infty)} \leq 1}\|f^*\chi_{(a, \infty)}\|_{Y(0, \infty)} \quad \text{for every $a>0$}.
\end{equation*}
\end{proposition}
\begin{proof}
Since the function $(0, \infty) \ni a \mapsto \varphi_X(a)/a$ is nonincreasing thanks to the quasiconcavity of the fundamental function of an \riSp{}, we have
\begin{equation*}
\varphi_X(2a) = 2a\frac{\varphi_X(2a)}{2a} \leq 2a \frac{\varphi_X(a)}{a}  = 2 \varphi_X(a) \quad \text{for every $a>0$}.
\end{equation*}
Using this, we see that
\begin{align*}
\frac{\varphi_Y(a)}{\varphi_X(a)} \leq 2 \frac{\varphi_Y(a)}{\varphi_X(2a)} = 2 \frac{\|\chi_{(0,2a)}\chi_{(a, \infty)}\|_{Y(0, \infty)}}{\|\chi_{(0,2a)}\|_{X(0, \infty)}}.
\end{align*}
Since the function $\chi_{(0,2a)}$ is nonincreasing, the claim follows.
\end{proof}

\begin{remark}
The preceding proposition shows that if
\begin{equation}\label{rem:relation_eq}
\lim_{a\to\infty} \sup_{\|f\|_{X(0, \infty)} \leq 1}\|f^*\chi_{(a, \infty)}\|_{Y(0, \infty)} = 0,
\end{equation}
then necessarily
\begin{equation*}
\lim_{a \to \infty} \frac{\varphi_Y(a)}{\varphi_X(a)} = 0.
\end{equation*}
However, the latter is not sufficient for the former in general (see~\cref{prop:ac_near_infinity_LZ_examples,rem:decay_of_fund_not_imply_ac_near_infinity_LZ_example}). Note that the validity of \eqref{rem:relation_eq} can be viewed as a relation between two function spaces. The interested reader is referred to \cite{MP:25preprint, P:25thesis}, where this relation and its generalization were systematically studied.
\end{remark}

We now prove \cref{prop:bounded_set_in_WmrX_tail_is_small}.
\begin{proof}[Proof of \cref{prop:bounded_set_in_WmrX_tail_is_small}]
Let $R>0$ and $u\in\wmrX[1]$. We will show that \eqref{prop:bounded_set_in_WmrX_tail_is_small:quantitative_estimate} is valid with a constant $C>0$, depending only on the dimension $n$. We start with two simple observations. First, we may without loss of generality assume that
\begin{equation}\label{prop:bounded_set_in_WmrX_tail_is_small:eqq:1}
\|u\|_{\wmrX[1]}\leq1.
\end{equation}
Second, we may assume that $\sup_{\|f\|_{\rep{X}}\leq 1} \|f^*\chi_{(R^{n-1},\infty)}\|_{\rep{Y}}<\infty$; otherwise, \eqref{prop:bounded_set_in_WmrX_tail_is_small:quantitative_estimate} trivially holds. Set
\begin{align}
r_0 &= R^{n-1}, \label{prop:bounded_set_in_WmrX_tail_is_small:eqq:3}\\
C_R &= \sup_{\|f\|_{\rep{X}\leq1}}\|f^*\chi_{(r_0, \infty)}\|_{\rep{Y}}, \label{prop:bounded_set_in_WmrX_tail_is_small:eqq:12}
\end{align}
and note that
\begin{equation}\label{prop:bounded_set_in_WmrX_tail_is_small:eqq:2}
\frac{\varphi_Y(r_0)}{\varphi_X(r_0)} + C_R \leq 3C_R
\end{equation}
thanks to \cref{prop:decay_at_infinity_implies_decay_of_fund_func}. Let $\tilde{R} > R$ be arbitrary. For brevity's sake, we also set
\begin{equation*}
a = \omN R^{n} \qquad \text{and} \qquad b = \omN \tilde{R}^n.
\end{equation*}

Since $u\in\wmrX[1]$, there is a locally absolutely continuous function $g\colon (0, \infty) \to \R$ such that
\begin{align*}
u(x) &= g(\omN^\frac1{n} |x|) \quad \text{for a.e.\ $x\in\rn$} \\
\intertext{and}
\nabla u(x) &= \omN^\frac1{n} g'(|x|) \frac{x}{|x|} \quad \text{for a.e.\ $x\in\rn$}.
\end{align*}
Note that
\begin{equation}\label{prop:bounded_set_in_WmrX_tail_is_small:eqq:6}
\|u\chi_{B_{\tilde{R}}\setminus B_R}\|_{Y(\rn)} = \|g(t^\frac1{n})\chi_{(a,b)}(t)\|_{\rep{Y}} 
\end{equation}
thanks to \eqref{prel:representation_measure_preserving_entire_space}. Furthermore, using \eqref{prel:representation_measure_preserving_entire_space} once more together with~\eqref{prop:bounded_set_in_WmrX_tail_is_small:eqq:1}, we see that
\begin{align}
\|g(t^\frac1{n})\|_{\rep{X}} &= \|u\|_{X(\rn)} \leq 1 \label{prop:bounded_set_in_WmrX_tail_is_small:eqq:4}\\
\intertext{and}
\omN^\frac1{n} \|g'(t^\frac1{n})\|_{\rep{X}} &= \|\nabla u\|_{X(\rn)} \leq 1. \label{prop:bounded_set_in_WmrX_tail_is_small:eqq:5}
\end{align}
Next, for every $j\in\N_0$, we define the intervals $I_j$ as
\begin{equation*}
I_j = \big( a + (j-1)r_0, a + jr_0 \big].
\end{equation*}
Note that $|I_j| = r_0$ for every $j\in\N_0$. Now, let $\{\eta_j\}_{j = 1}^\infty \subseteq \C^1_c(\R)$ be a sequence of cutoff functions such that
\begin{equation}\label{prop:bounded_set_in_WmrX_tail_is_small:eqq:7}
0\leq \eta_j \leq 1,\ \eta_j \equiv 1\ \text{in $I_j$},\ \spt \eta_j \subseteq I_{j-1} \cup I_j \cup I_{j+1},\ \text{and}\ |\eta_j'|\leq 2/r_0
\end{equation}
for every $j\in\N$. Fix $M\in\N$ so large that
\begin{equation}\label{prop:bounded_set_in_WmrX_tail_is_small:eqq:8}
(a,b)\subseteq  \bigcup_{j = 1}^{M} I_j.
\end{equation}
Using~\eqref{prop:bounded_set_in_WmrX_tail_is_small:eqq:6}, \eqref{prop:bounded_set_in_WmrX_tail_is_small:eqq:8}, and \eqref{prop:bounded_set_in_WmrX_tail_is_small:eqq:7}, we obtain
\begin{align}
\|u\chi_{B_{\tilde{R}}\setminus B_R}\|_{Y(\rn)} &= \|g(t^\frac1{n})\chi_{(a,b)}(t)\|_{\rep{Y}} \leq\Big\| \sum_{j = 1}^M g(t^\frac1{n}) \chi_{I_j}(t) \Big\|_{\rep{Y}} \nonumber\\
&= \Big\| \sum_{j = 1}^M g(t^\frac1{n}) \eta_j(t) \chi_{I_j}(t) \Big\|_{\rep{Y}} \nonumber\\
&\leq \frac1{n}\Big\| \sum_{j = 1}^M \chi_{I_j}(t) \int_{t}^{a + (j+1)r_0} |g'(s^\frac1{n})| s^{-1 + \frac1{n}} \eta_j(s) \dd s \Big\|_{\rep{Y}} \nonumber\\
&\quad+ \Big\| \sum_{j = 1}^M \chi_{I_j}(t) \int_{t}^{a + (j+1)r_0} |g(s^\frac1{n})\eta_j'(s)| \dd s \Big\|_{\rep{Y}} \nonumber\\
&\leq \frac1{n} \Big\| \sum_{j = 1}^M \chi_{I_j}(t) \int_{a + (j-1)r_0}^{a + (j+1)r_0} |g'(s^\frac1{n})| s^{-1 + \frac1{n}} \dd s \Big\|_{\rep{Y}} \nonumber\\
&\quad+ \Big\| \sum_{j = 1}^M \chi_{I_j}(t) \int_{a + (j-1)r_0}^{a + (j+1)r_0} |g(s^\frac1{n})\eta_j'(s)| \dd s \Big\|_{\rep{Y}}. \label{prop:bounded_set_in_WmrX_tail_is_small:eqq:9}
\end{align}
As for the first term, we have
\begin{align*}
&\Big\| \sum_{j = 1}^M \chi_{I_j}(t) \int_{a + (j-1)r_0}^{a + (j+1)r_0} |g'(s^\frac1{n})| s^{-1 + \frac1{n}} \dd s \Big\|_{\rep{Y}} \\
&\leq \Big\| \sum_{j = 1}^M \chi_{I_j}(t) a^{-1 + \frac1{n}}\int_{a + (j-1)r_0}^{a + (j+1)r_0} |g'(s^\frac1{n})| \dd s \Big\|_{\rep{Y}} \\
&\leq \omega_n^{\frac{1-n}{n}} \Big\| \sum_{j = 1}^M \chi_{I_j}(t) \frac1{r_0}\int_{a + (j-1)r_0}^{a + (j+1)r_0} |g'(s^\frac1{n})| \dd s \Big\|_{\rep{Y}}
\end{align*}
thanks to \eqref{prop:bounded_set_in_WmrX_tail_is_small:eqq:3}. Moreover, using \cref{prop:averaging_lemma}, \eqref{prop:bounded_set_in_WmrX_tail_is_small:eqq:2}, and \eqref{prop:bounded_set_in_WmrX_tail_is_small:eqq:5}, we obtain
\begin{equation*}
\omega_n^{\frac{1-n}{n}}\Big\| \sum_{j = 1}^M \chi_{I_j}(t) \frac1{r_0}\int_{a + (j-1)r_0}^{a + (j+1)r_0} |g'(s^\frac1{n})| \dd s \Big\|_{\rep{Y}} \leq \frac{6}{\omega_n}C_R.
\end{equation*}
Hence
\begin{equation}\label{prop:bounded_set_in_WmrX_tail_is_small:eqq:10}
\Big\| \sum_{j = 1}^M \chi_{I_j}(t) \int_{a + (j-1)r_0}^{a + (j+1)r_0} |g'(s^\frac1{n})| s^{-1 + \frac1{n}} \dd s \Big\|_{\rep{Y}} \leq \frac{6}{\omega_n}C_R.
\end{equation}
As for the second term in~\eqref{prop:bounded_set_in_WmrX_tail_is_small:eqq:9}, we have
\begin{align}
&\Big\| \sum_{j = 1}^M \chi_{I_j}(t) \int_{a + (j-1)r_0}^{a + (j+1)r_0} |g(s^\frac1{n}) \eta_j'(s)| \dd s \Big\|_{\rep{Y}} \nonumber\\
&\leq \Big\| \sum_{j = 1}^M \chi_{I_j}(t) \frac{2}{r_0}\int_{a + (j-1)r_0}^{a + (j+1)r_0} |g(s^\frac1{n})| \dd s \Big\|_{\rep{Y}} \nonumber \\
&\leq 12C_R \label{prop:bounded_set_in_WmrX_tail_is_small:eqq:11}
\end{align}
thanks to \eqref{prop:bounded_set_in_WmrX_tail_is_small:eqq:7}, \cref{prop:averaging_lemma}, \eqref{prop:bounded_set_in_WmrX_tail_is_small:eqq:2}, and \eqref{prop:bounded_set_in_WmrX_tail_is_small:eqq:4}. Therefore, combining \eqref{prop:bounded_set_in_WmrX_tail_is_small:eqq:9} with \eqref{prop:bounded_set_in_WmrX_tail_is_small:eqq:10} and \eqref{prop:bounded_set_in_WmrX_tail_is_small:eqq:11}, we arrive at
\begin{equation}\label{prop:bounded_set_in_WmrX_tail_is_small:eqq:13}
\|u\chi_{B_{\tilde{R}}\setminus B_R}\|_{Y(\rn)} \leq \Big( \frac{6}{n\omega_n} + 12 \Big)C_R.
\end{equation}
Letting $\tilde{R}$ go to $\infty$ in \eqref{prop:bounded_set_in_WmrX_tail_is_small:eqq:13}, we obtain
\begin{equation*}
\|u\chi_{\rn\setminus B_R}\|_{Y(\rn)} \leq \Big( \frac{6}{n\omega_n} + 12 \Big)C_R.
\end{equation*}
At last, this combined with \eqref{prop:bounded_set_in_WmrX_tail_is_small:eqq:12} yields \eqref{prop:bounded_set_in_WmrX_tail_is_small:quantitative_estimate}. 

Finally, having already established \eqref{prop:bounded_set_in_WmrX_tail_is_small:quantitative_estimate}, we immediately see that \eqref{prop:bounded_set_in_WmrX_tail_is_small:globally_ac} implies \eqref{prop:bounded_set_in_WmrX_tail_is_small:tail_vanishes}, which finishes the proof.
\end{proof}

We conclude this section with a proposition, probably of independent interest, which will play an important role in our proof of the necessity of the condition \eqref{thm:characterization_of_comp_emb_Y:eq:globally_ac} for the compactness of \eqref{thm:characterization_of_comp_emb_Y:eq:embedding}.
\begin{proposition}\label{prop:construction_of_radially_symmetric_function_vanishing_at_infinity}
Let $m,n\in\N$, $n\geq2$. Assume that $f$ is a nonnegative nonincreasing simple function in $(0, \infty)$ with bounded support. 

For every $a\geq1$, there is a nonnegative radially symmetric function $u_{f,a}$ in $\rn$ such that
\begin{align}
&u_{f,a}\in\wmrX \cap \C^{m-1}(\rn) \quad \text{for every \riSp{} $X(\rn)$}, \nonumber\\
&u_{f,a}(x) = 0 \quad \text{for every $x\in\rn$ with $\omega_n|x|^n < \frac{a}{8}$}, \label{prop:construction_of_radially_symmetric_function_vanishing_at_infinity:eq:function_vanishes_as_a_increases}\\
&\|u_{f, a}\|_{\wmrX} \leq C_1 \|f\|_{\rep{X}}, \label{prop:construction_of_radially_symmetric_function_vanishing_at_infinity:eq:sobolev_norm_bounded_by_norm_of_f}\\
\intertext{and}
&\|\chi_{(a, \infty)}f\|_{\rep{Y}} \leq C_2 \|u_{f,a}\|_{Y(\rn)} \quad \text{for every \riSp{} $Y(\rn)$}. \label{prop:construction_of_radially_symmetric_function_vanishing_at_infinity:eq:tail_of_f_bounded_by_norm_of_u_fa}
\end{align}
Here $C_1$ and $C_2$ are constants depending only on $m$ and $n$. In particular, they are independent of $f$, $a$, and the \riSps{}.
\end{proposition}
\begin{proof}
Throughout the entire proof, $C>0$ is an universal absolute constant that may depend only on $m$ and $n$.

Since $(a/4)^{1/n}-(a/8)^{1/n} = (4^{-1/n}-8^{-1/n})a^{1/n} \geq (4^{-1/n}-8^{-1/n})$ for every $a\geq1$, we can find a collection $\{\eta_a\}_{a\geq1}$ of cutoff functions such that $\eta_a\in\C_0^\infty(0, \infty)$, $0\leq \eta_a \leq 1$, $\eta_a \equiv 0$ in $(0, (a/8)^{1/n})$, $\eta_a\equiv 1$ in $((a/4)^{1/n}, \infty)$, and $|\eta_a^{(k)}(t)|\leq C$ for every $t\in(0, \infty)$ and every $k=1, \dots, m$ with a constant $C$ independent of $a$. Let $\{\phi_a\}_{a\geq1}$ be defined as
\begin{equation*}
\phi_a(x) = \eta_a(\omega_n^\frac{1}{n}|x|),\ x\in\rn.
\end{equation*}
It is easy to see that $\phi_a\in\C_0^\infty(\rn)$ and
\begin{equation}\label{prop:construction_of_radially_symmetric_function_vanishing_at_infinity:eq:6}
|D^m \phi_a(x)|\leq C \quad \text{for every $x\in\rn$}.
\end{equation}
Note that the function $\tilde{f} = f(a)\chi_{(0,a)} + f\chi_{[a, \infty)}$ is a nonnegative nonincreasing simple function in $(0, \infty)$ with bounded support that is constant in $(0,a)$. Clearly, $\|\chi_{(a,\infty)}f\|_{\rep{Y}} = \|\chi_{(a,\infty)}\tilde{f}\|_{\rep{Y}}$ and $\|\tilde f\|_{\rep{X}}\leq \|f\|_{\rep{X}}$. Therefore, we may assume that $f$ is constant in the interval $(0,a)$. Hence, it can be expressed as
\begin{equation}\label{prop:construction_of_radially_symmetric_function_vanishing_at_infinity:eq:f_expressed_as_sum}
f = \sum_{j = 1}^M \gamma_j \chi_{(0, b_j)}
\end{equation}
for some $M\in\N$, $\gamma_j >0$, and $a\leq b_1 < b_2 < \dots < b_M < \infty$. Set
\begin{equation}\label{prop:construction_of_radially_symmetric_function_vanishing_at_infinity:eq:g_definition}
g = \sum_{j = 1}^M \frac{\gamma_j}{b_j^{m/n}}\chi_{(0, b_j^{1/n})}.
\end{equation}
Next, we inductively define the functions $v_0, v_1, \dots, v_m\colon (0, \infty) \to [0, \infty)$ as
\begin{equation*}
v_j(t) = (T^jg)(t),\ t\in(0, \infty),\ j\in\{0, 1, \dots, m\},
\end{equation*}
where $T$ is an auxiliary operator, mapping $\Mpl(0, \infty)$ to $\Mpl(0, \infty)$, defined as
\begin{equation*}
(Th)(t) = \int_t^\infty h(s) \dd s,\ t\in(0, \infty),\ h\in\Mpl(0, \infty).
\end{equation*}
The expression $T^j$ is to be understood as $T^j = \overbrace{T\circ \dots \circ T}^{\text{$j$-times}}$. Thanks to the Fubini theorem, for every $j = 1, \dots, m$, we have
\begin{equation}\label{prop:construction_of_radially_symmetric_function_vanishing_at_infinity:eq:v_j_after_fubini}
v_j(t) = \int_t^\infty \frac{g(\tau)(\tau - t)^{j-1}}{(j-1)!} \dd\tau \quad \text{for all $t\in(0, \infty)$ and $j = 1, \dots, m$}.
\end{equation}
Hence, it is easy to see that
\begin{equation}\label{prop:construction_of_radially_symmetric_function_vanishing_at_infinity:eq:5}
v_j(t) \leq C \int_{t}^\infty g(\tau) \tau^{m-1} \dd\tau \quad \text{for all $j\in\{1, \dots, m\}$ and $t\geq\frac1{8^{1/n}}$}.
\end{equation}

Now, we define the function $u_{f,a}$ as
\begin{equation}\label{prop:construction_of_radially_symmetric_function_vanishing_at_infinity:eq:def_of_u_fa}
u_{f,a}(x) = \eta_a(\omega_n^\frac1{n}|x|) v_m(\omega_n^\frac1{n}|x|),\ x\in\rn.
\end{equation}
Clearly, we have $u(x) = 0$ for every $x\in\rn$ with $\omega_n|x|^n < a/8$. Furthermore, it is easy to see that $u_{f,a}\in\C^{m-1}(\rn)$ and its $m$th order weak derivatives exist in $\rn$. Next, we claim that \eqref{prop:construction_of_radially_symmetric_function_vanishing_at_infinity:eq:tail_of_f_bounded_by_norm_of_u_fa} is true. To this end, thanks to \eqref{prop:construction_of_radially_symmetric_function_vanishing_at_infinity:eq:def_of_u_fa} together with \eqref{prop:construction_of_radially_symmetric_function_vanishing_at_infinity:eq:v_j_after_fubini} with $j=m$, we have
\begin{align}
\|u_{f,a}\|_{Y(\rn)} &= \Big\| \eta_a(\omega_n^\frac1{n}|x|)\int_{\omega_n^\frac1{n}|x|}^\infty \frac{g(\tau)(\tau - \omega_n^\frac1{n}|x|)^{m-1}}{(m-1)!} \dd\tau \Big\|_{Y(\rn)} \nonumber\\
&= \Big\| \eta_a(t^\frac1{n})\int_{t^{\frac1{n}}}^\infty \frac{g(\tau)(\tau - t^{\frac1{n}})^{m-1}}{(m-1)!} \dd\tau\Big\|_{\rep{Y}} \nonumber\\
&\geq \Big\| \chi_{(a, \infty)}(4t)\int_{(2t)^{\frac1{n}}}^\infty \frac{g(\tau)(\tau - t^{\frac1{n}})^{m-1}}{(m-1)!} \dd\tau\Big\|_{\rep{Y}} \nonumber\\
&\geq \Big\| \chi_{(a, \infty)}(4t)\int_{(2t)^{\frac1{n}}}^\infty \frac{g(\tau)(\tau - \frac{\tau}{2^{1/n}})^{m-1}}{(m-1)!} \dd\tau\Big\|_{\rep{Y}} \nonumber\\
&= \frac{(1 - 2^{-1/n})^{m-1}}{(m-1)!}\Big\| \chi_{(a, \infty)}(4t)\int_{(2t)^{\frac1{n}}}^\infty g(\tau) \tau^{m-1} \dd\tau\Big\|_{\rep{Y}}. \label{prop:construction_of_radially_symmetric_function_vanishing_at_infinity:eq:1}
\end{align}
Furthermore, recalling \eqref{prop:construction_of_radially_symmetric_function_vanishing_at_infinity:eq:g_definition} and \eqref{prop:construction_of_radially_symmetric_function_vanishing_at_infinity:eq:f_expressed_as_sum}, we obtain
\begin{align}
\Big\| \chi_{(a, \infty)}(4t)\int_{(2t)^{\frac1{n}}}^\infty g(\tau) \tau^{m-1} \dd\tau\Big\|_{\rep{Y}} &= \Big\| \chi_{(a, \infty)}(4t)\sum_{j = 1}^M \frac{\gamma_j}{b_j^{m/n}} \int_{(2t)^{\frac1{n}}}^\infty \chi_{(0, b_j^{1/n})}(\tau) \tau^{m-1} \dd\tau\Big\|_{\rep{Y}} \nonumber\\
&= \Big\| \chi_{(a, \infty)}(4t)\sum_{j = 1}^M \frac{\gamma_j}{b_j^{m/n}} \chi_{(0, b_j/2)}(t) \int_{(2t)^{\frac1{n}}}^{b_j^{\frac1{n}}}\tau^{m-1} \dd\tau\Big\|_{\rep{Y}} \nonumber\\
&\geq \Big\| \chi_{(a, \infty)}(4t)\sum_{j = 1}^M \frac{\gamma_j}{b_j^{m/n}} \chi_{(0, b_j/4)}(t) \int_{(\frac{b_j}{2})^\frac1{n}}^{b_j^{\frac1{n}}}\tau^{m-1} \dd\tau\Big\|_{\rep{Y}} \nonumber\\
&= \frac{1 - 2^{-\frac{m}{n}}}{m} \Big\| \chi_{(a, \infty)}(4t)\sum_{j = 1}^M \gamma_j \chi_{(0, b_j/4)}(t) \Big\|_{\rep{Y}} \nonumber\\
&= \frac{1 - 2^{-\frac{m}{n}}}{m} \Big\| \chi_{(a, \infty)}(4t)f(4t) \Big\|_{\rep{Y}}. \label{prop:construction_of_radially_symmetric_function_vanishing_at_infinity:eq:2}
\end{align}
Since the norm of the dilation operator $D_4$ on $\rep{Y}$ is at most~$4$ (see~\eqref{prel:dilation_bounded}), we have
\begin{equation}
\|\chi_{(a, \infty)}f\|_{\rep{Y}} \leq 4 \|\chi_{(a, \infty)}(4t)f(4t)\|_{\rep{Y}}. \label{prop:construction_of_radially_symmetric_function_vanishing_at_infinity:eq:3}
\end{equation}
Hence, combining \eqref{prop:construction_of_radially_symmetric_function_vanishing_at_infinity:eq:1}, \eqref{prop:construction_of_radially_symmetric_function_vanishing_at_infinity:eq:2}, and \eqref{prop:construction_of_radially_symmetric_function_vanishing_at_infinity:eq:3}, we obtain \eqref{prop:construction_of_radially_symmetric_function_vanishing_at_infinity:eq:tail_of_f_bounded_by_norm_of_u_fa}.

Finally, we prove \eqref{prop:construction_of_radially_symmetric_function_vanishing_at_infinity:eq:sobolev_norm_bounded_by_norm_of_f}, which will conclude the proof. We claim that
\begin{equation}\label{prop:construction_of_radially_symmetric_function_vanishing_at_infinity:eq:7}
\|u_{f,a}\|_{X(\rn)} \leq \frac1{m!} \|f\|_{\rep{X}}.
\end{equation}
To this end, using \eqref{prop:construction_of_radially_symmetric_function_vanishing_at_infinity:eq:g_definition}, we see that
\begin{align}
\Big\| \int_{t^\frac1{n}}^\infty g(\tau)\tau^{m-1} \dd\tau \Big\|_{\rep{X}} &\leq \Big\| \sum_{j = 1}^M \frac{\gamma_j}{b_j^{m/n}} \chi_{(0, b_j)}(t) \int_0^{b_j^\frac1{n}} \tau^{m-1} \dd\tau \Big\|_{\rep{X}} \nonumber\\
&= \frac1{m}\Big\| \sum_{j = 1}^M \frac{\gamma_j}{b_j^{m/n}} \chi_{(0, b_j)}(t) b_j^\frac{m}{n} \Big\|_{\rep{X}} = \frac1{m}\|f\|_{\rep{X}}. \label{prop:construction_of_radially_symmetric_function_vanishing_at_infinity:eq:8}
\end{align}
Furthermore, since \eqref{prop:construction_of_radially_symmetric_function_vanishing_at_infinity:eq:def_of_u_fa} together with \eqref{prop:construction_of_radially_symmetric_function_vanishing_at_infinity:eq:v_j_after_fubini} with $j=m$ clearly implies
\begin{align*}
\|u_{f,a}\|_{X(\rn)} &\leq \frac1{(m-1)!}\Big\| \int_{\omega_n^\frac1{n}|x|}^\infty g(\tau)\tau^{m-1} \dd\tau \Big\|_{X(\rn)} \\
&= \frac1{(m-1)!}\Big\| \int_{t^\frac1{n}}^\infty g(\tau)\tau^{m-1} \dd\tau \Big\|_{\rep{X}} ,
\end{align*}
we obtain \eqref{prop:construction_of_radially_symmetric_function_vanishing_at_infinity:eq:7}.
Next, we claim that
\begin{equation}\label{prop:construction_of_radially_symmetric_function_vanishing_at_infinity:eq:4}
|D^m u_{f,a}(x)| \leq C \Big( \int_{\omega_n^\frac1{n}|x|}^\infty g(\tau) \tau^{m-1} \dd\tau + g(\omega_n^\frac1{n}|x|) \Big) \quad \text{for a.e.\ $x\in\rn$}.
\end{equation}
It can be verified by straightforward computations that all the (weak) derivatives of the function $u_{f,a}$ up to order $m$ are linear combinations of (some of) the functions
\begin{equation*}
\rn\ni x \mapsto \eta_a^{(i)}(\omega_n^\frac1{n}|x|)\frac{x_{k_1} \cdots x_{k_j}}{|x|^{2l - 1}} v_{m-k}(\omega_n^\frac1{n}|x|)
\end{equation*}
with
\begin{equation*}
i,k\in\{0, \dots, m\},\ l\in\{1, \dots, m\},\ \text{and}\ j\in\{0, \dots, 2l - 1\}.
\end{equation*}
Moreover, the coefficients and the number of terms in the linear combinations depend only on $m$ and $n$. Now, the desired inequality \eqref{prop:construction_of_radially_symmetric_function_vanishing_at_infinity:eq:4} will follow from this, \eqref{prop:construction_of_radially_symmetric_function_vanishing_at_infinity:eq:6}, and \eqref{prop:construction_of_radially_symmetric_function_vanishing_at_infinity:eq:5} once we make two observations. On the one hand, since the functions $\eta_a^{(i)}$, $i=0,\dots, m$, vanish in the interval $(0, 8^{-1/n}]\subseteq (0, (a/8)^{1/n}]$, the inequality \eqref{prop:construction_of_radially_symmetric_function_vanishing_at_infinity:eq:4} trivially holds for a.e.~$x\in\rn$ with $\omega_n^{1/n}|x|\leq 8^{-1/n}$. On the other hand, for all $l\in\{1, \dots, m\}$ and $j\in\{0, \dots, 2l - 1\}$,
\begin{equation*}
\Big| \frac{x_{k_1} \cdots x_{k_j}}{|x|^{2l - 1}} \Big| \leq C |x|^{1 - 2l} \leq C \quad \text{for every $x\in\rn$ with $\omega_n^{1/n}|x| > 8^{-\frac1{n}}$}.
\end{equation*}
Therefore, the inequality \eqref{prop:construction_of_radially_symmetric_function_vanishing_at_infinity:eq:4} is indeed valid. At last, combining \eqref{prop:construction_of_radially_symmetric_function_vanishing_at_infinity:eq:4} with \eqref{prop:construction_of_radially_symmetric_function_vanishing_at_infinity:eq:8}, we obtain
\begin{equation*}
\|D^m u_{f,a}\|_{X(\rn)} \leq C (\|f\|_{\rep{X}} + \|g(t^\frac1{n})\|_{\rep{X}}).
\end{equation*}
Furthermore, using \eqref{prop:construction_of_radially_symmetric_function_vanishing_at_infinity:eq:g_definition} together with the fact that $b_j^{m/n}\geq a^{m/n} \geq 1$ for every $j\in\{1, \dots, M\}$, and \eqref{prop:construction_of_radially_symmetric_function_vanishing_at_infinity:eq:f_expressed_as_sum}, we see that
\begin{equation*}
\|g(t^\frac1{n})\|_{\rep{X}} \leq \Big\| \sum_{j = 1}^M \gamma_j \chi_{(0, b_j^{1/n})}(t^\frac1{n})\Big\|_{\rep{X}} = \Big\| \sum_{j = 1}^M \gamma_j \chi_{(0, b_j)}(t)\Big\|_{\rep{X}} = \|f\|_{\rep{X}}.
\end{equation*}
Hence, we have
\begin{equation*}
\|D^m u_{f,a}\|_{X(\rn)} \leq C \|f\|_{\rep{X}}.
\end{equation*}
Combining this and \eqref{prop:construction_of_radially_symmetric_function_vanishing_at_infinity:eq:7}, we see that \eqref{prop:construction_of_radially_symmetric_function_vanishing_at_infinity:eq:sobolev_norm_bounded_by_norm_of_f} is true, which concludes the proof.
\end{proof}


\section{Weighted embeddings on balls}\label{sec:weighted_Sobolev_on_balls}
In this section, we investigate the weighted Sobolev embedding
\begin{equation}\label{E:embedding_on_balls_Sob_emb}
		\wmrXB \hookrightarrow Y(B_R, \aldx),
\end{equation}
where $\aldx$ is the weighted measure $\dd \aldx(x) = |x|^\alpha \dd x$, where $\alpha \in[0, \infty)$. Our first goal is to characterize the optimal target \riSp{} in \eqref{E:embedding_on_balls_Sob_emb}, that is, the smallest \riSp{} $Y(B_R, \aldx)$ with which the embedding is valid. We then use our description of the optimal target space to characterize when the embedding is compact. Results from this section will also be useful in our proof of \cref{thm:characterization_of_comp_emb_Y}.

\subsection{Optimal target space in the weighted Sobolev embedding}
The following theorem is a so-called reduction principle for the weighted Sobolev embedding~\eqref{E:embedding_on_balls_Sob_emb}. It characterizes \eqref{E:embedding_on_balls_Sob_emb} by means of an inequality involving an integral operator (sometimes called a Hardy operator or a Copson operator) and functions of one variable (see~\eqref{thm:embedding_on_balls_characterization:eq:Hardy} below). By (K\"othe) dualizing this inequality as in \cite{KP:06}, where the optimal \riSps{} in classical Sobolev embeddings on bounded Lipschitz domains in $\rn$ are characterized, we can describe the optimal \riSp{} in~\eqref{E:embedding_on_balls_Sob_emb}.
\begin{theorem}\label{thm:embedding_on_balls_characterization}
Let $m,n\in\N$, $n\geq2$. Let $\alpha\geq0$ and $R\in(0, \infty)$. For all \riSps{} $X(B_R)$ and $Y(B_R,\aldx)$, the following two statements are equivalent.
\begin{enumerate}[label=(\arabic*), ref=(\arabic*)]
	\item There is a constant $C_{B,S}>0$ such that
	\begin{equation}\label{thm:embedding_on_balls_characterization:eq:Sob}
		\|u\|_{Y(B_R,\aldx)} \leq C_{B,S} \|u\|_{\wmrXB} \quad \text{for every $u\in\wmrXB$}.
	\end{equation}
	\item There is a constant $C_{B,H}>0$ such that
	\begin{equation}\label{thm:embedding_on_balls_characterization:eq:Hardy}
		\Big\| \int_{t^{\frac{n}{n+\alpha}}}^1 f(s) s^{-1 + \frac{m}{n}} \dd s \Big\|_{\repFin{Y}} \leq C_{B,H} \|f\|_{\repFin{X}} \quad \text{for every $f\in\Mpl(0,1)$}.
	\end{equation}
\end{enumerate}
\end{theorem}
Before we prove our reduction principle for \eqref{E:embedding_on_balls_Sob_emb}, we need to establish a couple of auxiliary propositions. But first, we use it to describe the optimal target \riSp{} in \eqref{E:embedding_on_balls_Sob_emb}.

\begin{corollary}\label{cor:embedding_on_balls_optimal_target}
Let $m,n\in\N$, $n\geq2$. Let $\alpha\geq0$ and $R\in(0, \infty)$. For an \riSp{} $X(B_R, \aldx)$, let $Y_{X}(B_R, \aldx)$ be the \riSp{} whose function norm is defined as
\begin{equation*}
\|u\|_{Y_X(B_R, \aldx)} = \|u^*(\aldx(B_R)t)\|_{\repFin{Y_X}},\ u\in\Mpl(B_R, \aldx),
\end{equation*}
where $\|\cdot\|_{\repFin{Y_X}}$ is the function norm satisfying
\begin{equation}\label{cor:embedding_on_balls_optimal_target:eq_asoc_norm_def}
\|g\|_{\repAsocFin{Y_X}} = \|t^{\frac{m+\alpha}{n}}g^{**}(t^\frac{n+\alpha}{n})\|_{\repAsocFin{X}} \quad \text{for every $g\in\Mpl(0, 1)$}.
\end{equation}
Then $Y_{X}(B_R, \aldx)$ is optimal in the embedding \eqref{E:embedding_on_balls_Sob_emb} in the following sense. The embedding \eqref{E:embedding_on_balls_Sob_emb} is valid with $Y=Y_{X}$ and if $Z(B_R, \aldx)$ is an \riSp{} such that \eqref{E:embedding_on_balls_Sob_emb} is valid with $Y=Z$, then $Y_{X}(B_R, \aldx) \hookrightarrow Z(B_R, \aldx)$.
\end{corollary}
\begin{proof}
It can be easily verified that the functional $\|\cdot\|_{\repAsocFin{Y_X}}$ defined by \eqref{cor:embedding_on_balls_optimal_target:eq_asoc_norm_def} is an \riNorm{} on $(0,1)$ (cf.~\cite[Proposition~3.1]{M:23}). For every \riSp{} $Z(0,1)$, we have
\begin{align}
\sup_{\|f\|_{\repFin{X}}\leq1} \Big\| \int_{t^{\frac{n}{n+\alpha}}}^1 |f(s)| s^{-1 + \frac{m}{n}} \dd s \Big\|_{Z(0,1)} &= \sup_{\substack{\|f\|_{\repFin{X}}\leq1\\ \|g\|_{Z'(0,1)}\leq1}} \int_0^1 g^*(t) \Big( \int_{t^{\frac{n}{n+\alpha}}}^1 |f(s)| s^{-1 + \frac{m}{n}} \dd s \Big) \dd t \nonumber\\
&=\sup_{\substack{\|f\|_{\repFin{X}}\leq1\\ \|g\|_{Z'(0,1)}\leq1}} \int_0^1 |f(s)| s^{-1 + \frac{m}{n}} \Big( \int_0^{s^\frac{n + \alpha}{n}} g^*(t) \dd t \Big) \dd s \nonumber\\
&=\sup_{\substack{\|f\|_{\repFin{X}}\leq1\\ \|g\|_{Z'(0,1)}\leq1}} \int_0^1 |f(s)| s^{\frac{m + \alpha}{n}} g^{**}(s^\frac{n + \alpha}{n}) \dd s \nonumber\\
&=\sup_{\|g\|_{Z'(0,1)}\leq1} \|t^{\frac{m + \alpha}{n}} g^{**}(t^\frac{n + \alpha}{n})\|_{\repAsocFin{X}} \label{cor:embedding_on_balls_optimal_target:eq:1}
\end{align}
thanks to \eqref{prel:double_associate_norm_nonincreasing_functions}, the Fubini theorem, and \eqref{prel:associate_norm}. Now, on the one hand, using \eqref{cor:embedding_on_balls_optimal_target:eq:1} with $Z(0,1) = \repFin{Y_X}$, we obtain
\begin{align*}
\sup_{\|f\|_{\repFin{X}}\leq1} \Big\| \int_{t^{\frac{n}{n+\alpha}}}^1 |f(s)| s^{-1 + \frac{m}{n}} \dd s \Big\|_{\repFin{Y_X}} &= \sup_{\|g\|_{\repAsocFin{Y_X}}\leq1} \|t^{\frac{m + \alpha}{n}} g^{**}(t^\frac{n + \alpha}{n})\|_{\repAsocFin{X}} \\
&= \sup_{\|g\|_{\repAsocFin{Y_X}}\leq1} \|g\|_{\repAsocFin{Y_X}} = 1.
\end{align*}
Hence, the Sobolev embedding $\wmrX\hookrightarrow Y_X(B_R, \aldx)$ is valid thanks to \cref{thm:embedding_on_balls_characterization}. On the other hand, if $\wmrX\hookrightarrow Y(B_R, \aldx)$ is valid for an \riSp{} $Y(B_R, \aldx)$, then \cref{thm:embedding_on_balls_characterization} and \eqref{cor:embedding_on_balls_optimal_target:eq:1} with $Z(0,1) = \repFin{Y}$ imply that there is a constant $C\in(0, \infty)$ such that
\begin{equation*}
\sup_{\|g\|_{\repAsocFin{Y}}\leq1} \|g\|_{\repAsocFin{Y_X}} = \sup_{\|g\|_{\repAsocFin{Y}}\leq1} \|t^{\frac{m + \alpha}{n}} g^{**}(t^\frac{n + \alpha}{n})\|_{\repAsocFin{X}} \leq C.
\end{equation*}
In other words, $\repAsocFin{Y} \hookrightarrow \repAsocFin{Y_X}$. It follows that $Y_X(B_R, \aldx) \hookrightarrow Y(B_R, \aldx)$ thanks to \eqref{prel:embedding_iff_asoc_embedding}, which finishes the proof.
\end{proof}

\begin{remark}\label{rem:embedding_on_balls_into_Linfty}
Note that
\begin{equation*}
\wmrXB \hookrightarrow L^\infty(B_R, \aldx)
\end{equation*}
if and only if
\begin{equation*}
\sup_{\|f\|_{\repFin{X}\leq1}} \int_0^1 f(s) s^{-1 + \frac{m}{n}} \dd s < \infty
\end{equation*}
thanks to \cref{thm:embedding_on_balls_characterization}. It is easy to see that this is the case if and only if either $m<n$ and $\repFin{X} \hookrightarrow L^{\frac{n}{m},1}(0,1)$ or $m\geq n$. Since the Lebesgue measure and the weighted measure $\aldx$ are absolutely continuous with respect to each other, it follows that $L^\infty(B_R,\aldx) = L^\infty(B_R)$. In other words, the weight is completely immaterial as far as embeddings into $L^\infty$ are concerned.
\end{remark}

We now establish two auxiliary propositions and then finally prove \cref{thm:embedding_on_balls_characterization}.
\begin{proposition}\label{prop:interpolation_of_operators_auxiliary}
For $\alpha\neq0$ and $\beta\geq \max\{0, -\alpha\}$, define the operator $T_{\alpha, \beta}$, mapping $\M(0, 1)$ to $\Mpl(0,1)$, as
\begin{equation*}
T_{\alpha, \beta}f(t) = t^{\beta}\int_t^1 |f(s)| s^{-1+\alpha} \dd s,\ t\in(0, 1),\ f\in\M(0, 1).
\end{equation*}
Let $X(0, 1)$ be an \riSp{}. Then
\begin{equation*}
\|T_{\alpha, \beta}f\|_{X(0, 1)} \leq \max\Big\{ \frac1{|\alpha|}, \frac1{\beta + 1} \Big\} \|f\|_{X(0, 1)} \quad \text{for every $f\in\M(0, 1)$}.
\end{equation*}
\end{proposition}
\begin{proof}
By Calder\'on's interpolation result \cite[Theorem~2]{C:66} (see also~\cite[Chapter~3, Theorem~2.2]{BS}), in order to prove the boundedness of $T_{\alpha, \beta}$ on $X(0, 1)$ with the operator norm at most $\max\{ 1/|\alpha|, 1/(\beta + 1) \}$, it is sufficient to show that
\begin{align*}
\|T_{\alpha, \beta}f\|_{L^\infty(0, 1)} \leq \frac1{|\alpha|} \|f\|_{L^\infty(0, 1)} \quad \text{for every $f\in L^\infty(0, 1)$} \\
\intertext{and}
\|T_{\alpha, \beta}f\|_{L^1(0, 1)} \leq \frac1{\beta + 1}  \|f\|_{L^1(0, 1)} \quad \text{for every $f\in L^1(0, 1)$}.
\end{align*}
These estimates are, however, easy to establish. If $\alpha < 0$, using the fact that $-1+\alpha<-1$ and $\alpha+\beta\geq0$, we see that
\begin{align*}
\|T_{\alpha, \beta}f\|_{L^\infty(0, 1)} &\leq \|f\|_{L^\infty(0, 1)} \sup_{t\in(0, 1)} t^\beta \int_t^1 s^{-1+\alpha} \dd s \\
&\leq \|f\|_{L^\infty(0, 1)} \frac1{|\alpha|} \sup_{t\in(0, 1)} t^{\alpha + \beta} \\
&= \frac1{|\alpha|} \|f\|_{L^\infty(0, 1)}.
\end{align*}
If $\alpha>0$, we have
\begin{align*}
\|T_{\alpha, \beta}f\|_{L^\infty(0, 1)} &\leq \|f\|_{L^\infty(0, 1)} \sup_{t\in(0, 1)} t^\beta \int_0^1 s^{-1+\alpha} \dd s \\
&= \frac1{\alpha} \|f\|_{L^\infty(0, 1)},
\end{align*}
this time using the fact that $-1+\alpha>-1$ and $\beta\geq0$.
Finally, using the Fubini theorem together with the fact that $\beta \geq \max\{0,-\alpha\} > -1$,  we have
\begin{align*}
\|T_{\alpha, \beta}f\|_{L^1(0, 1)} &= \int_0^1 |f(s)| s^{-1+\alpha} \int_0^s t^{\beta} \dd t \dd s \\
&\leq \frac1{\beta + 1} \sup_{s\in(0, 1)} s^{\beta + \alpha} \|f\|_{L^1(0, 1)} \\
&= \frac1{\beta + 1}  \|f\|_{L^1(0, 1)},
\end{align*}
whether $\alpha$ is positive or negative.
\end{proof}

\begin{proposition}\label{prop:construction_of_radially_symmetric_function_local_necessity_Hardy}
Let $m,n\in\N$, $1\leq m < n$. Let $R\in(0, \infty)$ and $\alpha\in[0, \infty)$. There are positive constants $C_3$ and $C_4$ such that for all \riSps{} $X(B_R)$ and $Y(B_R,\aldx)$, the following is true. For every $a\in(0, 1]$ and every nonnegative $f\in L^\infty(0, 1)$ with support inside $[0, a)$, there is a nonnegative radially symmetric and nonincreasing function $u_{f,R,a}$ such that:
\begin{align}
&u_{f,R,a}\in\wmrXB \cap \C^{m-1}(B_R\setminus\{0\}) , \nonumber\\
&\text{$u_{f,R,a}$ has support inside $B_r$, where $r=a^\frac1{n}R$}, \nonumber\\
&\|u_{f,R,a}\|_{\wmrXB} \leq C_3 \max\{1, R^{-m}\} \|f\|_{\repFin{X}}, \label{prop:construction_of_radially_symmetric_function_local_necessity_Hardy:eq:sobolev_norm_bounded_by_norm_of_f}\\
\intertext{and}
&\Big\| \int_{t^\frac{n}{n+\alpha}}^1 f(s) s^{-1+\frac{m}{n}} \dd s \Big\|_{\repFin{Y}} \leq C_4 \|u_{f,R,a}\|_{Y(B_R,\aldx)}. \label{prop:construction_of_radially_symmetric_function_local_necessity_Hardy:eq:Hardy_bounded_by_u}
\end{align}
Moreover, the constant $C_3$ depends only on $m$ and $n$, and the constant $C_4$ depends only on $\alpha$, $m$, and $n$.
\end{proposition}
\begin{proof}
In the entire proof, $C>0$ is an universal constant that may depend only on $\alpha$, $m$, and $n$. Let $f\in L^\infty(0, 1)$ be nonnegative with support inside $[0, a)$. The construction of the function $u_{f,R,a}$ is inspired by the construction in~\cite[p.~562]{KP:06}. We define the functions $g_1, \dots, g_m\colon [0, R] \to [0, \infty)$ as
\begin{equation*}
g_k(t) = \chi_{(0, R)}(t) T^k(f_m)\big( (t/R)^n \big),\ t\in[0, R],\ k = 0, 1, \dots, m,
\end{equation*}
where $T$ is an auxiliary operator, mapping $\Mpl(0, 1)$ to $\Mpl(0, 1)$, defined as
\begin{equation*}
(Th)(t) = \chi_{(0, 1)}(t)\int_t^1 h(s) \dd s,\ t\in[0, 1],\ h\in\Mpl(0, 1),
\end{equation*}
and $f_m$ is defined as
\begin{equation}\label{prop:construction_of_radially_symmetric_function_local_necessity_Hardy:eq:1}
f_m(t) = f(t) t^{-m + \frac{m}{n}},\ t\in(0, 1).
\end{equation}
The expression $T^j$ is to be understood as $T^j = \overbrace{T\circ \dots \circ T}^{\text{$j$-times}}$. It is easy to see that $g_m\in \C^{m-1}((0, R])$ and that $g_m^{(m-1)}$ is locally Lipschitz continuous in $(0, R]$. Using the Fubini theorem, we can easily verify by induction that for every $k\in\{1, \dots, m\}$ and every $t\in(0, R)$,
\begin{equation}\label{prop:construction_of_radially_symmetric_function_local_necessity_Hardy:eq:2}
g_k(t) =  \frac{\chi_{(0, R)}(t)}{(k-1)!} \int_{(t/R)^n}^1 f(s) s^{-m + \frac{m}{n}} \big( s - (t/R)^n \big)^{k-1} \dd s.
\end{equation}
Therefore, for every $j\in\{0, 1, \dots, m-1\}$ we have
\begin{equation}\label{prop:construction_of_radially_symmetric_function_local_necessity_Hardy:eq:3}
g_{m-j}(t) \leq \frac{\chi_{(0, R)}(t)}{(m-j-1)!} \int_{(t/R)^n}^1 f(s) s^{-j - 1 + \frac{m}{n}} \dd s \quad \text{for every $t\in(0, R)$}.
\end{equation}
Furthermore, using \eqref{prop:construction_of_radially_symmetric_function_local_necessity_Hardy:eq:2} with $k = m$ together with the fact that $(t/R)^n\leq s/2$ for every $s\geq 2(t/R)^n$, we have
\begin{align}
g_{m}(t) &\geq	\frac{\chi_{(0, 2^{-n}R)}(t)}{(m-1)!} \int_{2(t/R)^n}^1 f(s) s^{-m + \frac{m}{n}} (s - (t/R)^n)^{m-1} \dd s \nonumber \\
&\geq \frac{2^{1-m}}{(m-1)!} \chi_{(0, 2^{-n}R)}(t) \int_{2(t/R)^n}^1 f(s) s^{-m + \frac{m}{n}} s^{m-1} \dd s \nonumber \\
&= \frac{2^{1-m}}{(m-1)!} \chi_{(0, 2^{-n}R)}(t) \int_{2(t/R)^n}^1 f(s) s^{-1 + \frac{m}{n}} \dd s. \label{prop:construction_of_radially_symmetric_function_local_necessity_Hardy:eq:4}
\end{align}
Next, set
\begin{equation*}
u_{f,R,a}(x) = g_m(|x|),\ x\in B_R.
\end{equation*}
Note that $u_{f,R,a}\in\C^{m-1}(B_R\setminus\{0\})$ and its weak derivatives of order $m$ exist in $B_R$. Furthermore, the function $u_{f,R,a}$ is clearly radially symmetric and nonincreasing. Moreover, since the support of $f$ is inside $[0, a)$, the support of $u_{f,R,a}$ is inside $B_r$, where $r=a^\frac1{n}R$.

Now, we claim that \eqref{prop:construction_of_radially_symmetric_function_local_necessity_Hardy:eq:sobolev_norm_bounded_by_norm_of_f} is valid. To this end, by~\eqref{prel:representation_measure_preserving_ball} with $\alpha = 0$ and by~\eqref{prop:construction_of_radially_symmetric_function_local_necessity_Hardy:eq:3} with $j = 0$, we have
\begin{equation*}
\|u_{f,R,a}\|_{X(B_R)} \leq \frac1{(m-1)!}\Big\|\int_t^1 f(s) s^{-1 + \frac{m}{n}} \dd s \ \Big\|_{\repFin{X}}.
\end{equation*}
Furthermore, using~\cref{prop:interpolation_of_operators_auxiliary} with $\alpha = m/n$, and $\beta = 0$, we obtain
\begin{align}
\|u_{f,R,a}\|_{X(B_R)} \leq \frac1{(m-1)!} \frac{n}{m} \|f\|_{\repFin{X}}. \label{prop:construction_of_radially_symmetric_function_local_necessity_Hardy:eq:7}
\end{align}
Straightforward computations show that all the (weak) derivatives of $u_{f,R,a}$ up to order $m$ are linear combinations of (some of) the functions
\begin{equation}\label{prop:construction_of_radially_symmetric_function_local_necessity_Hardy:eq:5}
B_R \ni x \mapsto R^{-jn} g_{m-j}(|x|)|x|^{jn-k-l}x_{i_1}\cdots x_{i_l},
\end{equation}
where $j\in\{1, \dots,m\}$, $k\in\{j,\dots, m\}$,  and $l\in\{0, 1, \dots, k\}$. Moreover, the coefficients depend only on $m$ and $n$. Hence, using \eqref{prop:construction_of_radially_symmetric_function_local_necessity_Hardy:eq:5}, \eqref{prop:construction_of_radially_symmetric_function_local_necessity_Hardy:eq:1}, \eqref{prop:construction_of_radially_symmetric_function_local_necessity_Hardy:eq:3}, and \eqref{prel:representation_measure_preserving_ball} with $\alpha = 0$, we see that
\begin{align}
\|D^m u_{f,R,a}\|_{X(B_R)} &\leq C \Big( R^{-mn} \|f_m((|x|/R)^n) |x|^{mn - m}\|_{X(B_R)} \nonumber\\
&\qquad+ \sum_{j = 1}^{m-1}\sum_{k = j}^m R^{-jn} \|g_{m-j}(|x|)|x|^{jn-k}\|_{X(B_R)} \Big) \nonumber\\
&\qquad\leq C \Big( R^{-m} \|f((|x|/R)^n)\|_{X(B_R)} \nonumber\\
&\qquad+ \sum_{j = 1}^{m-1}\sum_{k = j}^m R^{-jn} \Big\| |x|^{jn-k} \int_{(|x|/R)^n}^1 f(s) s^{-j - 1 + \frac{m}{n}} \dd s \Big\|_{X(B_R)} \Big) \nonumber\\
&= C \Big(R^{-m} \|f\|_{\repFin{X}} \nonumber\\
&\qquad+ \sum_{j = 1}^{m-1}\sum_{k = j}^m R^{-k}\Big\| t^{j-\frac{k}{n}}\int_t^1 f(s)s^{-j-1+\frac{m}{n}} \dd s \Big\|_{\repFin{X}} \Big). \label{prop:construction_of_radially_symmetric_function_local_necessity_Hardy:eq:6}
\end{align}
Furthermore, for all $j\in\{1,\dots,m-1\}$ and $k\in\{j, \dots, m\}$, we have
\begin{equation*}
\Big\| t^{j-\frac{k}{n}}\int_t^1 f(s)s^{-j-1+\frac{m}{n}} \dd s \Big\|_{\repFin{X}} \leq \frac{n}{nj-m} \|f\|_{\repFin{X}}
\end{equation*}
thanks to \cref{prop:interpolation_of_operators_auxiliary} with $\alpha = -j + m/n$ and $\beta = j-k/n$. Hence, combining this with \eqref{prop:construction_of_radially_symmetric_function_local_necessity_Hardy:eq:6}, we obtain
\begin{equation}\label{prop:construction_of_radially_symmetric_function_local_necessity_Hardy:eq:8}
\|D^m u_{f,R,a}\|_{X(B_R)} \leq C \max\{1, R^{-m}\} \|f\|_{\repFin{X}}.
\end{equation}
Therefore, thanks to \eqref{prop:construction_of_radially_symmetric_function_local_necessity_Hardy:eq:7} and \eqref{prop:construction_of_radially_symmetric_function_local_necessity_Hardy:eq:8}, we arrive at
\begin{equation*}
\|u_{f,R,a}\|_{\wmrXB} \leq C \max\{1, R^{-m}\} \|f\|_{\repFin{X}}.
\end{equation*}
Hence, \eqref{prop:construction_of_radially_symmetric_function_local_necessity_Hardy:eq:sobolev_norm_bounded_by_norm_of_f} is true. In particular, since $\|f\|_{\repFin{X}} < \infty$, we have $u_{f,R,a}\in\wmrXB$.

Finally, we have yet to show that \eqref{prop:construction_of_radially_symmetric_function_local_necessity_Hardy:eq:Hardy_bounded_by_u} is valid. Thanks to the boundedness of the dilation operator $D_{2^{(n+\alpha)/n}}$ on $\rep{Y}$ (see~\eqref{prel:dilation_bounded}) and to~\eqref{prop:construction_of_radially_symmetric_function_local_necessity_Hardy:eq:4}, we have
\begin{align*}
\Big\| \int_{t^\frac{n}{n+\alpha}}^1 f(s) s^{-1+\frac{m}{n}} \dd s \Big\|_{\repFin{Y}}  &\leq 2^\frac{n+\alpha}{n} \Big\| \chi_{(0, 2^{-(n+\alpha)/n})}(t) \int_{2t^\frac{n}{n+\alpha}}^1 f(s) s^{-1+\frac{m}{n}} \dd s \Big\|_{\repFin{Y}} \\
&= 2^\frac{n+\alpha}{n} \Big\| \chi_{(0, 2^{-n}R)}(|x|) \int_{2(|x|/R)^n}^1 f(s) s^{-1+\frac{m}{n}} \dd s \Big\|_{Y(B_R, \aldx)}\\
&\leq 2^\frac{n+\alpha}{n} 2^{m-1} (m-1)! \|g_m(|x|)\|_{Y(B_R, \aldx)} \\
&= 2^{\frac{\alpha}{n} + m} (m-1)! \|u_{f,R,a}\|_{Y(B_R, \aldx)}.
\end{align*}
Therefore, \eqref{prop:construction_of_radially_symmetric_function_local_necessity_Hardy:eq:Hardy_bounded_by_u} is indeed true, which concludes the proof.
\end{proof}

\begin{remark}\label{rem:continuation_outside_by_zero}
Since the function $u_{f,R,a}$ with $a\in(0,1)$ has support inside $B_r$ with $r < R$, its continuation outside $B_R$ by $0$ belongs to $\wmrX\cap\mathcal C^{m-1}(\rn\setminus\{0\})$, provided that $X(B_R)$ in \cref{prop:construction_of_radially_symmetric_function_local_necessity_Hardy} is the restriction of an \riSp{} $X(\rn)$. Furthermore, an inspection of the proof shows that the support of the function $u_{f,R,a}$ is the closed ball $\bar{B}_{\varrho^{1/n}R}$, where $\varrho = \esssup\{t\in(0,1): f(t) > 0\}$.
\end{remark}

Now, we finally prove \cref{thm:embedding_on_balls_characterization}.
\begin{proof}[Proof of~\cref{thm:embedding_on_balls_characterization}]
Assume that \eqref{thm:embedding_on_balls_characterization:eq:Hardy} is valid with a constant $C_{B,H}>0$. Let $u\in\wmrXB$ and $u(x)=g(|x|)$. Note that for every $j\in\{0,\dots, m-1\}$ and every $t\in(0,R)$, we have
\begin{align}
|g^{(j)}(t)| &\leq \int_t^R |g^{(j+1)}(s)| \dd s + |g^{(j)}(R)| \nonumber\\
&= \frac{R}{n}\int_{(t/R)^n}^1 |g^{(j+1)}(Rs^\frac1{n})| s^{-1 + \frac1{n}} \dd s + |g^{(j)}(R)| \label{thm:embedding_on_balls_characterization:eq:1}
\end{align}
thanks to the fact that $g^{(j)}\in \aclocAB{R}$ and the obvious change of variables. Furthermore, since $g^{(j)}\in\C([\frac{R}{2},R])$ for every $j\in\{0,\dots, m-1\}$, the elementary mean value theorem for integrals yields the existence of $t_0\in[\frac{R}{2},R]$ such that
\begin{equation*}
g^{(j)}(t_0) = \frac{2}{R} \int_{\frac{R}{2}}^{R} g^{(j)} (s) \dd s.
\end{equation*}
Hence, we have
\begin{align}
|g^{(j)}(R)| &= \Bigg| \int_{t_0}^R g^{(j+1)}(s) \dd s + \frac{2}{R} \int_{\frac{R}{2}}^{R} g^{(j)} (s) \dd s \Bigg| \nonumber\\
&\leq \int_{\frac{R}{2}}^R |g^{(j+1)}(s)| \dd s + \frac{2}{R} \int_{\frac{R}{2}}^{R} |g^{(j)}(s)| \dd s \nonumber\\
&= \frac{R}{n}\int_{2^{-n}}^1 |g^{(j+1)}(Rs^\frac1{n})| s^{-1+\frac1{n}} \dd s + \frac{2}{n} \int_{2^{-n}}^1 |g^{(j)}(Rs^\frac1{n})| s^{-1+\frac1{n}} \dd s \nonumber\\
&\leq \frac{2^{n-1}R}{n}\int_{2^{-n}}^1 |g^{(j+1)}(Rs^\frac1{n})| \dd s + \frac{2^{n}}{n} \int_{2^{-n}}^1 |g^{(j)}(Rs^\frac1{n})| \dd s \label{thm:embedding_on_balls_characterization:eq:2}
\end{align}
for every $j\in\{0,\dots, m-1\}$ and every $t\in(0,R)$. Therefore, thanks to \eqref{thm:embedding_on_balls_characterization:eq:2}, the H\"older inequality \eqref{prel:Holder}, and \eqref{prel:representation_measure_preserving_ball} with $\alpha = 0$, we obtain
\begin{equation}\label{thm:embedding_on_balls_characterization:eq:3}
|g^{(j)}(R)| \leq \frac{2^n\fundX[X'](\omega_nR^n)}{n} \Bigg(\|g^{(j)}(|x|)\|_{X(B_R)} + \frac{R}{2}\|g^{(j+1)}(|x|)\|_{X(B_R)}\Bigg).
\end{equation}
Now, we show by backward induction on $i=0,\dots, m-1$, that
\begin{align}
|g^{(i)}(t)| &\leq \frac{R^{m-i}}{n}\int_{(t/R)^n}^1 |g^{(m)}(Rs^\frac1{n})| s^{-1+\frac{m-i}{n}} \dd s \nonumber\\
&\quad+ 2^{n+m-i-1}\frac{\fundX[X'](\omega_nR^n)}{n}\max\{1, R^{m-i}\}\sum_{k=i}^m \|g^{(k)}(|x|)\|_{X(B_R)} \label{thm:embedding_on_balls_characterization:eq:4}
\end{align}
for every $t\in(0,R)$. For $i=m-1$, the validity of \eqref{thm:embedding_on_balls_characterization:eq:4} easily follows from \eqref{thm:embedding_on_balls_characterization:eq:1} and \eqref{thm:embedding_on_balls_characterization:eq:3} with $j=m-i$. Assuming that \eqref{thm:embedding_on_balls_characterization:eq:4} is valid for $i\in\{1,\dots,m-1\}$, we see that
\begin{align*}
&\int_{(t/R)^n}^1 |g^{(i)}(Rs^\frac1{n})| s^{-1 + \frac1{n}} \dd s \\
&\leq \frac{R^{m-i}}{n}\int_{(t/R)^n}^1 \Big( \int_s^1 |g^{(m)}(R\tau^\frac1{n})| \tau^{-1+\frac{m-i}{n}} \dd\tau \Big) s^{-1 + \frac1{n}} \dd s \\
&\quad+ 2^{n+m-i-1}\fundX[X'](\omega_nR^n)\max\{1, R^{m-i}\}\sum_{k=i}^m \|g^{(k)}(|x|)\|_{X(B_R)}
\end{align*}
for every $t\in(0, \infty)$. As for the integral on the right-hand side, we have
\begin{equation*}
\int_{(t/R)^n}^1 \Big( \int_s^1 |g^{(m)}(R\tau^\frac1{n})| \tau^{-1+\frac{m-i}{n}} \dd\tau \Big) s^{-1 + \frac1{n}} \dd s \leq n\int_{(t/R)^n}^1 |g^{(m)}(R\tau^\frac1{n})| \tau^{-1+\frac{m-i+1}{n}} \dd\tau
\end{equation*}
thanks to the Fubini theorem. Therefore, combining the last two inequalities, we have
\begin{align*}
&\int_{(t/R)^n}^1 |g^{(i)}(Rs^\frac1{n})| s^{-1 + \frac1{n}} \dd s \\
&\leq R^{m-i}\int_{(t/R)^n}^1 |g^{(m)}(Rs^\frac1{n})| s^{-1+\frac{m-i+1}{n}} \dd s \\
&\quad+ 2^{n+m-i-1}\fundX[X'](\omega_nR^n)\max\{1, R^{m-i}\}\sum_{k=i}^m \|g^{(k)}(|x|)\|_{X(B_R)}
\end{align*}
for every $t\in(0, \infty)$. Hence, by combining this and both \eqref{thm:embedding_on_balls_characterization:eq:1} and \eqref{thm:embedding_on_balls_characterization:eq:3} with $j=i-1$, it is easy to see that \eqref{thm:embedding_on_balls_characterization:eq:4} is valid for $i-1$, which finishes the inductive step.

Next, thanks to \eqref{prel:representation_measure_preserving_ball}, \eqref{thm:embedding_on_balls_characterization:eq:4} with $i=0$, and \eqref{thm:embedding_on_balls_characterization:eq:Hardy}, we see that
\begin{align*}
\|u\|_{Y(B_R, \aldx)} &= \|g(|x|)\|_{Y(B_R, \aldx)} = \|g(Rt^{\frac1{n+\alpha}})\|_{\rep[0,1]{Y}} \\
&\leq \frac{R}{n}\Big\| \int_{t^{\frac{n}{n+\alpha}}}^1 g(Rs^\frac1{n}) s^{-1 + \frac{m}{n}} \dd s \Big\|_{\rep[0,1]{Y}} \\
&\quad+ 2^{n+m-1}\frac{\fundX[X'](\omega_nR^n)}{n}\max\{1, R^{m-i}\}\fundX[Y](\aldx(B_R))\sum_{k=0}^m \|g^{(k)}(|x|)\|_{X(B_R)} \\
&\leq C \Big( \|g(Rt^{\frac1{n}})\|_{\rep[0,1]{X}} + \sum_{k=0}^m \|g^{(k)}(|x|)\|_{X(B_R)} \Big) \\
&=C \Big( \|g(|x|)\|_{X(B_R)} + \sum_{k=0}^m \|g^{(k)}(|x|)\|_{X(B_R)} \Big)
\end{align*}
with a constant $C$ depending only on $C_{B,H}$, $\alpha$, $m$, $n$, $R$, $\fundX$ and $\fundX[Y]$. Combining this with \eqref{prel:ineq_between_norm_of_gradient_and_derivative_of_the_one_var_func}, we obtain the validity of \eqref{thm:embedding_on_balls_characterization:eq:Sob} with $C_{B,S}=2C$.

Now, assume that \eqref{thm:embedding_on_balls_characterization:eq:Sob} is valid with a constant $C_{B,S}>0$. We will show that \eqref{thm:embedding_on_balls_characterization:eq:Hardy} is valid with a constant $C_{B,H}\in(0, \infty)$. To this end, it is clearly sufficient to consider nonnegative functions $f\in\repFin{X}$. Furthermore, since $\min\{|f|, N\}\nearrow |f|$ as $N\to\infty$, it is actually sufficient to prove \eqref{thm:embedding_on_balls_characterization:eq:Hardy} for nonnegative functions $f\in \repFin{X} \cap L^\infty(0,1)$. Moreover, since $|f|\chi_{(0,a)}\nearrow |f|$ as $a\to 1^+$, we may also assume that $f$ is zero in a left neighborhood of $1$.  Let $f\in \repFin{X} \cap L^\infty(0,1)$ be nonnegative and such that $f = f\chi_{(0,a)}$ for some $a\in(0,1)$. Using \eqref{prop:construction_of_radially_symmetric_function_local_necessity_Hardy:eq:Hardy_bounded_by_u}, \eqref{thm:embedding_on_balls_characterization:eq:Sob}, and \eqref{prop:construction_of_radially_symmetric_function_local_necessity_Hardy:eq:sobolev_norm_bounded_by_norm_of_f}, we obtain
\begin{align*}
\Big\| \int_{t^{\frac{n}{n+\alpha}}}^1 f(s) s^{-1 + \frac{m}{n}} \dd s \Big\|_{\rep[0, 1]{Y}} &\leq C_4 \|u_{f,R,1}\|_{Y(B_R, \aldx)} \leq C_4 C_{B, S} \|u_{f,R,1}\|_{\wmrXB} \\
&\leq C_4 C_{B, S} C_3\max\{1, R^{-m}\} \|f\|_{\repFin{X}},
\end{align*}
where $u_{f,R,1}\in\wmrXB$ is the function from \cref{prop:construction_of_radially_symmetric_function_local_necessity_Hardy}, and $C_3$ and $C_4$ are the constants from \eqref{prop:construction_of_radially_symmetric_function_local_necessity_Hardy:eq:sobolev_norm_bounded_by_norm_of_f} and \eqref{prop:construction_of_radially_symmetric_function_local_necessity_Hardy:eq:Hardy_bounded_by_u}, respectively. Hence, the inequality \eqref{thm:embedding_on_balls_characterization:eq:Hardy} is valid with $C_{B,H} = C_3 C_4 C_{B, S}\max\{1, R^{-m}\}$.
\end{proof}

\subsection{Compactness of the weighted Sobolev embedding}
Now, having a description of the optimal target space in the weighted Sobolev embedding \eqref{E:embedding_on_balls_Sob_emb} at our disposal, we use it to characterize the compactness of the weighted Sobolev embedding. We only state the characterization when $Y(B_R, \aldx)$ is not equivalent to $L^\infty(B_R, \aldx)$, which is equivalent to the fact that (see~\cite[Theorem~5.2]{S:12}, cf.~\cite[Chapter~2, Theorem~5.5]{BS})
\begin{equation}\label{thm:embedding_on_balls_compactness_not_Linfty:fund_Y_vanishes_0}
\lim_{t\to 0^+} \varphi_Y(t) = 0.
\end{equation}
The problem of characterizing compactness of Sobolev embeddings into $L^\infty$, being usually in a way easier, requires different techniques. In fact, by combining \cref{prop:construction_of_radially_symmetric_function_local_necessity_Hardy} with $\alpha = 0$ and \cite[Theorem~1.2]{KP:08}, it can be shown that $\wmrX\hookrightarrow L^\infty(B_R)$ is compact if and only if $W^mX(B_R)\hookrightarrow L^\infty(B_R)$ is. Note that the restriction $m<n$ imposed in \cite[Theorem~1.2]{KP:08} stems from the interpolation technique used there and is not necessary (cf.~\cite[Theorem~6.1]{S:15}, \cite[Theorem~1.2]{CM:19}). Since $L^\infty(B_R,\aldx) = L^\infty(B_R)$ (recall~\cref{rem:embedding_on_balls_into_Linfty}), it follows that $\wmrX\hookrightarrow L^\infty(B_R, \aldx)$ is compact if and only if $W^mX(B_R)\hookrightarrow L^\infty(B_R)$ is. Therefore, we can use the known characterizations of the compactness of $W^mX(B_R)\hookrightarrow L^\infty(B_R)$ (in particular, see~\cite{KP:08}) to characterize the compactness of $\wmrX\hookrightarrow L^\infty(B_R, \aldx)$. It turns out that $\wmrXB \hookrightarrow L^\infty(B_R,\aldx)$ is compact if and only if
\begin{equation}\label{E:embedding_on_balls_compactness_into_Linfty_char_AsocX}
\lim_{a\to0^+} \|t^{-1 + \frac{m}{n}}\chi_{(0,a)}(t)\|_{\repAsocFin{X}} = 0.
\end{equation}
When $m\leq n$, it is easy to see, by using \eqref{prel:associate_norm}, that \eqref{E:embedding_on_balls_compactness_into_Linfty_char_AsocX} is equivalent to
\begin{equation}\label{E:embedding_on_balls_compactness_into_Linfty_char_domain_ac_to_Lorentz}
\repFin{X} \ac L^{\frac{n}{m},1}(0,1).
\end{equation}
When $m>n$, \eqref{E:embedding_on_balls_compactness_into_Linfty_char_AsocX} is always satisfied.

\begin{theorem}\label{thm:embedding_on_balls_compactness_not_Linfty}
Let $m,n\in\N$, $n\geq2$. Let $\alpha\geq0$ and $R\in(0, \infty)$. Let $X(B_R)$ and $Y(B_R,\aldx)$ be \riSps{}. Assume that \eqref{thm:embedding_on_balls_compactness_not_Linfty:fund_Y_vanishes_0} is true. The following two statements are equivalent.
\begin{enumerate}[label=(\arabic*), ref=(\arabic*)]
	\item The embedding \eqref{E:embedding_on_balls_Sob_emb} is compact.
	\item We have
	\begin{equation}\label{thm:embedding_on_balls_compactness_not_Linfty:ac}
		Y_X(B_R, \aldx) \ac Y(B_R, \aldx),
	\end{equation}
	where $Y_X(B_R, \aldx)$ is the optimal target space from \cref{cor:embedding_on_balls_optimal_target}.
\end{enumerate}
\end{theorem}

\begin{proof}
First, assume that the embedding \eqref{E:embedding_on_balls_Sob_emb} is compact. By~\eqref{prel:ac_iff_asocAC} and \eqref{prel:ac_iff_repAC}, the validity of \eqref{thm:embedding_on_balls_compactness_not_Linfty:ac} is equivalent to that of
\begin{equation}\label{thm:embedding_on_balls_compactness_not_Linfty:eq1}
\lim_{j\to\infty} \sup_{\|g\|_{\repAsocFin{Y}}\leq1}\|g^*\chi_{(0,1/j)}\|_{\repAsocFin{Y_X}} = 0.
\end{equation}
Let $\varepsilon > 0$. Using \eqref{cor:embedding_on_balls_optimal_target:eq_asoc_norm_def}, \eqref{prel:associate_norm}, the Fubini theorem, and \eqref{prel:double_associate_norm_nonincreasing_functions}, we see that
\begin{align}
\sup_{\|g\|_{\repAsocFin{Y}}\leq1}\|g^*\chi_{(0,1/j)}\|_{\repAsocFin{Y_X}} &= \sup_{\substack{\|g\|_{\repAsocFin{Y}}\leq1\\\|f\|_{\repFin{X}}\leq1}} \int_0^1 |f(t)| t^{-1+\frac{m}{n}} \int_0^{t^\frac{n+\alpha}{n}} g^*(s)\chi_{(0,1/j)}(s) \dd s \dd t \nonumber\\
&= \sup_{\substack{\|g\|_{\repAsocFin{Y}}\leq1\\\|f\|_{\repFin{X}}\leq1}} \int_0^1 g^*(s) \chi_{(0,1/j)}(s) \int_{s^\frac{n}{n+\alpha}}^1 |f(t)| t^{-1+\frac{m}{n}} \dd t \dd s \nonumber\\
&= \sup_{\|f\|_{\repFin{X}}\leq1} \Big\| \chi_{(0,1/j)}(t) \int_{t^\frac{n}{n+\alpha}}^1 |f(s)| s^{-1+\frac{m}{n}} \dd s \Big\|_{\repFin{Y}}. \label{thm:embedding_on_balls_compactness_not_Linfty:eq2}
\end{align}
Since $\min\{|f(t)|,N\}\nearrow |f(t)|$ for a.e.\ $t\in(0,1)$ as $N\to\infty$ for every $f\in\M(0,1)$, we have
\begin{align*}
&\sup_{\|f\|_{\repFin{X}}\leq1} \Big\| \chi_{(0,1/j)}(t) \int_{t^\frac{n}{n+\alpha}}^1 |f(s)| s^{-1+\frac{m}{n}} \dd s \Big\|_{\repFin{Y}}\\
&= \sup_{\substack{\|f\|_{\repFin{X}}\leq1\\f\in L^\infty(0,1)}} \Big\| \chi_{(0,1/j)}(t) \int_{t^\frac{n}{n+\alpha}}^1 |f(s)| s^{-1+\frac{m}{n}} \dd s \Big\|_{\repFin{Y}}.
\end{align*}
Moreover, since the embedding \eqref{E:embedding_on_balls_Sob_emb} is compact (and so also bounded), the supremum is finite thanks to \cref{thm:embedding_on_balls_characterization}. Therefore, for each $j\in\N$, there is a function $f_j\in L^\infty(0, 1)$, $\|f_j\|_{\repFin{X}}\leq1$, such that
\begin{align}
&\sup_{\|f\|_{\repFin{X}}\leq1} \Big\| \chi_{(0,1/j)}(t) \int_{t^\frac{n}{n+\alpha}}^1 |f(s)| s^{-1+\frac{m}{n}} \dd s \Big\|_{\repFin{Y}} \nonumber\\
&\leq 2 \Big\| \chi_{(0,1/j)}(t) \int_{t^\frac{n}{n+\alpha}}^1 |f_j(s)| s^{-1+\frac{m}{n}} \dd s \Big\|_{\repFin{Y}}. \label{thm:embedding_on_balls_compactness_not_Linfty:eq3}
\end{align}
By combining \eqref{thm:embedding_on_balls_compactness_not_Linfty:eq2} and \eqref{thm:embedding_on_balls_compactness_not_Linfty:eq3}, we obtain
\begin{equation}\label{thm:embedding_on_balls_compactness_not_Linfty:eq4}
\sup_{\|g\|_{\repAsocFin{Y}}\leq1}\|g^*\chi_{(0,1/j)}\|_{\repAsocFin{Y_X}} \leq 2 \Big\| \chi_{(0,1/j)}(t) \int_{t^\frac{n}{n+\alpha}}^1 |f_j(s)| s^{-1+\frac{m}{n}} \dd s \Big\|_{\repFin{Y}}
\end{equation}
for every $j\in\N$. Furthermore, set
\begin{equation*}
a_j = j^{-n/(n+\alpha)} \quad \text{for every $j\in\N$}.
\end{equation*}
Thanks to~\cref{prop:construction_of_radially_symmetric_function_local_necessity_Hardy} with $a = a_j$ and $f = |f_j|\chi_{(0,a_{2j})}$, there is a function $u_j\in\wmrXB$ such that
\begin{align}
&u_j(x) = 0 \quad \text{for every $x\in B_R$ with $|x| \geq j^{-n/(n+\alpha)} R$} \label{thm:embedding_on_balls_compactness_not_Linfty:eq5}\\
&\|u_j\|_{\wmrXB} \leq C \label{thm:embedding_on_balls_compactness_not_Linfty:eq6}\\
&\Big\| \int_{t^\frac{n}{n+\alpha}}^1 |f_j(s)| s^{-1+\frac{m}{n}} \chi_{(0,a_{2j})}(s) \dd s \Big\|_{\repFin{Y}} \leq C_4 \|u_j\|_{Y(B_R, \aldx)}\label{thm:embedding_on_balls_compactness_not_Linfty:eq7}
\end{align}
for each $j\in\N$, where $C_4$ is the constant from \eqref{prop:construction_of_radially_symmetric_function_local_necessity_Hardy:eq:Hardy_bounded_by_u} and $C$ depends only on $m$, $n$, and $R$.

Next, combining the compactness of the embedding \eqref{E:embedding_on_balls_Sob_emb} with \eqref{thm:embedding_on_balls_compactness_not_Linfty:eq6}, we may without loss of generality assume that there is a function $u\in Y(B_R, \aldx)$ such that $\lim_{j\to \infty} u_j = u$ in $Y(B_R, \aldx)$ and $\lim_{j\to\infty} u_j(x) = u(x)$ for $\aldx$-a.e.\ $x\in B_R$. The latter together with~\eqref{thm:embedding_on_balls_compactness_not_Linfty:eq5} implies that $u = 0$. Hence, we have
\begin{equation*}
\lim_{j\to \infty} \Big\| \int_{t^\frac{n}{n+\alpha}}^1 |f_j(s)| s^{-1+\frac{m}{n}} \chi_{(0,a_{2j})}(s) \dd s \Big\|_{\repFin{Y}} = 0
\end{equation*}
thanks to \eqref{thm:embedding_on_balls_compactness_not_Linfty:eq7}. Therefore, there is $j_0\in\N$ such that
\begin{equation}\label{thm:embedding_on_balls_compactness_not_Linfty:eq8}
\Big\| \int_{t^\frac{n}{n+\alpha}}^1 |f_j(s)| s^{-1+\frac{m}{n}} \chi_{(0,a_{2j_0})}(s) \dd s \Big\|_{\repFin{Y}} \leq \varepsilon.
\end{equation}
Now, using \eqref{thm:embedding_on_balls_compactness_not_Linfty:eq8} and the H\"older inequality \eqref{prel:Holder} in the last step, we have
\begin{align*}
\Big\| \chi_{(0,1/j)}(t) \int_{t^\frac{n}{n+\alpha}}^1 |f_j(s)| s^{-1+\frac{m}{n}} \dd s \Big\|_{\repFin{Y}} &\leq \Big\| \chi_{(0,1/j)}(t) \int_{t^\frac{n}{n+\alpha}}^1 |f_j(s)| s^{-1+\frac{m}{n}} \chi_{(0,a_{2j_0})}(s) \dd s \Big\|_{\repFin{Y}} \\
&\quad+ \Big\| \chi_{(0,1/j)}(t) \int_{t^\frac{n}{n+\alpha}}^1 |f_j(s)| s^{-1+\frac{m}{n}} \chi_{(a_{2j_0},1)}(s) \dd s \Big\|_{\repFin{Y}} \\
&\leq \Big\| \int_{t^\frac{n}{n+\alpha}}^1 |f_j(s)| s^{-1+\frac{m}{n}} \chi_{(0,a_{2j_0})}(s) \dd s \Big\|_{\repFin{Y}} \\
&\quad+ \Big( \int_{a_{2j_0}^{n/(n+\alpha)}}^1 |f_j(s)| s^{-1+\frac{m}{n}} \dd s \Big)\fundXrep[Y](1/j) \\
&\leq \varepsilon +  \fundXrep[Y](1/j) a_{2j_0}^\frac{m-n}{n + \alpha}\fundXAsocrep[X](1)
\end{align*}
for every $j\in\N$, $j\geq 2j_0$. Combining this with \eqref{thm:embedding_on_balls_compactness_not_Linfty:fund_Y_vanishes_0}, we obtain
\begin{equation*}
\Big\| \chi_{(0,1/j)}(t) \int_{t^\frac{n}{n+\alpha}}^1 |f_j(s)| s^{-1+\frac{m}{n}} \dd s \Big\|_{\repFin{Y}} \leq 2\varepsilon
\end{equation*}
for every $j\in\N$ sufficiently large. Since $\varepsilon>0$ was arbitrary, this together with~\eqref{thm:embedding_on_balls_compactness_not_Linfty:eq4} implies that \eqref{thm:embedding_on_balls_compactness_not_Linfty:eq1} is true. Hence, so is \eqref{thm:embedding_on_balls_compactness_not_Linfty:ac} in turn.

Second, assume that \eqref{thm:embedding_on_balls_compactness_not_Linfty:ac} is valid. Let $\{u_j\}_{j = 1}^\infty\in\wmrXB$ be a bounded sequence. By \cref{cor:embedding_on_balls_optimal_target}, the sequence is also bounded in $Y_X(B_R, \aldx)$. Furthermore, thanks to~\eqref{prel:Linfty_emb_X_emb_L1_finite_measure}, it is bounded in $W^{1,1}(B_R) = W^1L^1(B_R)$. The classical Rellich--Kondrachov theorem (e.g., \cite[Theorem~6.3]{AFbook}) implies that there is a function $u\in L^1(B_R)$ and a subsequence of $\{u_k\}_{k = 1}^\infty$ converging to $u$ pointwise a.e.~in $B_R$, which is the same as $\aldx$-a.e.\ in $B_R$. Moreover, thanks to~\eqref{prel:Fatou_lemma}, we have $u\in Y_X(B_R, \aldx)$. Hence, \cite[Theorem~3.1]{S:12} (cf.~\cite[Exercise~8, p.~31]{BS}) together with \eqref{thm:embedding_on_balls_compactness_not_Linfty:ac} implies that the subsequence converges to $u$ in $Y(B_R, \aldx)$. It follows that the embedding \eqref{E:embedding_on_balls_Sob_emb} is compact.
\end{proof}

We conclude this section with three remarks.
\begin{remark}\label{rem:different_characterizations_of_ac_embedding_of_optimal_space}
The validity of~\eqref{thm:embedding_on_balls_compactness_not_Linfty:ac}, with \eqref{thm:embedding_on_balls_compactness_not_Linfty:fund_Y_vanishes_0} still in force, can be equivalently reformulated in several ways. In particular, it is equivalent to the validity of
\begin{equation*}
	\lim_{a\to 0_+} \sup_{\|f\|_{\repFin{X}}\leq1} \Big\| \chi_{(0,a)}(t) \int_{t^\frac{n}{n+\alpha}}^1 f^*(s) s^{-1 + \frac{m}{n}} \dd{s} \Big\|_{\repFin{Y}} = 0.
\end{equation*}
This and other equivalent reformulations can be obtained in essentially the same way as in \cite{CM:22, KP:06, S:15}.
\end{remark}

\begin{remark}\label{rem:embedding_on_balls_compactness_functions_with_zero_boundary_data}
It is worth noting that the functions $\{u_j\}_{j = 1}^\infty$ that we used for proving the necessity of \eqref{thm:embedding_on_balls_compactness_not_Linfty:ac} actually belong to the subspace (see~\eqref{thm:embedding_on_balls_compactness_not_Linfty:eq5})
\begin{equation}\label{rem:embedding_on_balls_compactness_functions_with_zero_boundary_data:def}
\wmrXBvan = \{u\in\wmrXB: u=u\chi_{B_{\tilde{R}}}\text{for some $0<\tilde{R} < R$}\}.
\end{equation}
\end{remark}

\begin{remark}\label{rem:embedding_on_balls_compactness:m_geq_n_not_interesting}
In view of \cref{rem:embedding_on_balls_into_Linfty} and \eqref{prel:Linfty_emb_X_emb_L1_finite_measure}, we have $Y_X(B_R,\mu) = L^\infty(B_R,\mu)$ for every \riSp{} $X(B_R)$ when $m\geq n$. Furthermore, characterizing the compactness of the Sobolev embedding \eqref{E:embedding_on_balls_Sob_emb} is rather simple when $m\geq n$. When $m>n$, the discussion before \cref{thm:embedding_on_balls_compactness_not_Linfty} (in particular, recall~\eqref{E:embedding_on_balls_compactness_into_Linfty_char_AsocX}) together with \eqref{prel:Linfty_emb_X_emb_L1_finite_measure} implies that \eqref{E:embedding_on_balls_Sob_emb} is always compact. When $m = n$, the fact that $Y_X(B_R,\mu) = L^\infty(B_R,\mu)$ combined with \cref{thm:embedding_on_balls_compactness_not_Linfty} (see also \cite[Theorem~5.2]{S:12}) and \eqref{E:embedding_on_balls_compactness_into_Linfty_char_AsocX} yields that \eqref{E:embedding_on_balls_Sob_emb} is compact if and only if $X$ is not (equivalent to) $L^1$ or $Y$ is not (equivalent to) $L^\infty$. 
\end{remark}

\section{Proof of~Theorem~\ref{thm:characterization_of_comp_emb_Y}}\label{sec:proof_of_main_result}
\begin{proof}[Proof of~\cref{thm:characterization_of_comp_emb_Y}]
\emph{Sufficiency.}
Let $\{u_j\}_{j = 1}^\infty\subseteq\wmrX$ be bounded. We will show that it contains a subsequence that is Cauchy in $Y(\rn)$. Let $\varepsilon > 0$. Thanks to \cref{prop:bounded_set_in_WmrX_tail_is_small}, there is $R>0$ so large that
\begin{equation}\label{thm:characterization_of_comp_emb_Y_locally_not_Linfty:eq:1}
\sup_{i,j\in\N}\|(u_i - u_j)\chi_{\rn\setminus B_R}\|_{Y(\rn)} \leq \varepsilon.
\end{equation}
We now distinguish between whether \eqref{thm:characterization_of_comp_emb_Y:eq:fund_Y_vanishes_0} or \eqref{thm:characterization_of_comp_emb_Y:eq:fund_Y_not_vanishing_at_0} is the case.

First, assume that \eqref{thm:characterization_of_comp_emb_Y:eq:fund_Y_vanishes_0} is true. We claim that \eqref{thm:characterization_of_comp_emb_Y_locally_not_Linfty:eq:Hardy_op_compact} is valid whether $m < n$ or $m\geq n$. When $m < n$, its validity is assumed, and so there is nothing to prove. When $m\geq n$ (and so $-1+m/n\geq0$), we have
\begin{align*}
&\sup_{\|f\chi_{(0, 1)}\|_{\rep{X}}\leq 1} \Big\| \chi_{(0, a)}(t) \int_t^1 (f\chi_{(0, 1)})^*(s) s^{-1 + \frac{m}{n}} \dd{s}  \Big\|_{\rep{Y}} \\
&\quad\qquad\leq \|t^{-1 + \frac{m}{n}}\chi_{(0,1)}(t)\|_{\repAsoc{X}} \varphi_Y(a) \leq \varphi_{X'}(1) \varphi_Y(a)
\end{align*}
for every $a>0$ thanks to the H\"older inequality~\eqref{prel:Holder}. The validity of \eqref{thm:characterization_of_comp_emb_Y_locally_not_Linfty:eq:Hardy_op_compact} then follows from the assumption~\eqref{thm:characterization_of_comp_emb_Y:eq:fund_Y_vanishes_0}. Hence, either way, \eqref{thm:characterization_of_comp_emb_Y_locally_not_Linfty:eq:Hardy_op_compact} is valid. By combining~\eqref{prel:representation_space_restriction} with $\alpha = 0$ and \cref{rem:different_characterizations_of_ac_embedding_of_optimal_space}, we see that the validity of \eqref{thm:characterization_of_comp_emb_Y_locally_not_Linfty:eq:Hardy_op_compact} is equivalent to that of \eqref{thm:embedding_on_balls_compactness_not_Linfty:ac} with $\alpha = 0$, where $Y_X(B_R)$ is the optimal space from \cref{cor:embedding_on_balls_optimal_target} with $\alpha = 0$ and $X = X(B_R)$. Hence, the embedding $\wmrXB \hookrightarrow Y(B_R)$ is compact (in fact, even $W^mX(B_R) \hookrightarrow Y(B_R)$ is by~\cite[Theorem~1.1]{KP:08}). It follows that the sequence $\{u_j\chi_{B_R}\}_{j = 1}^\infty\subseteq\wmrXB$ contains a subsequence converging in $Y(B_R)$. Combining this with \eqref{thm:characterization_of_comp_emb_Y_locally_not_Linfty:eq:1}, we obtain the fact that $\{u_j\}_{j = 1}^\infty$ contains a subsequence that is Cauchy in $Y(\rn)$. Therefore, the embedding~\eqref{thm:characterization_of_comp_emb_Y:eq:embedding} is compact.

Now, assume that \eqref{thm:characterization_of_comp_emb_Y:eq:fund_Y_not_vanishing_at_0} is true. As in the preceding case, it remains for us to show that $\{u_j\chi_{B_R}\}_{j = 1}^\infty$ contains a subsequence that is Cauchy in $Y(\rn)$. To this end, we note that each of the conditions~\ref{enum:cond:m_less_than_n_Y_is_locally_Linfty}--\ref{enum:cond:m_bigger_than_n} is equivalent to the compactness of the embedding $W^mX(B_R)\hookrightarrow L^\infty(B_R)$. This follows from the explanation before \cref{thm:embedding_on_balls_compactness_not_Linfty} (see also~\cref{rem:embedding_on_balls_compactness:m_geq_n_not_interesting}). Since (see~\eqref{prel:Linfty_emb_X_emb_L1_finite_measure})
\begin{equation*}
\|u\chi_{B_R}\|_{Y(\rn)} \leq \|u\|_{L^\infty(\rn)} \fundX[Y](\omega_n R^n) \quad \text{for every $u\in\M(\rn)$},
\end{equation*}
the compactness of $W^mX(B_R)\hookrightarrow L^\infty(B_R)$ implies that $\{u_j\chi_{B_R}\}_{j = 1}^\infty$ contains a subsequence that is Cauchy in $Y(\rn)$. Finally, putting it all together, we obtain the existence of a subsequence of~$\{u_k\}_{k = 1}^\infty$ that is Cauchy in $Y(B_R)$.


\emph{Necessity.}
Assume that the embedding $\wmrX\hookrightarrow Y(\rn)$ is compact. In particular, $\wmrX\hookrightarrow Y(\rn)$, and so there is a constant $C>0$ such that
\begin{equation}\label{thm:characterization_of_comp_emb_Y_locally_not_Linfty:eq:17}
\|u\|_{Y(\rn)} \leq C \|u\|_{\wmrX} \quad \text{for every $u\in\wmrX$}.
\end{equation}

We start by proving that \eqref{thm:characterization_of_comp_emb_Y:eq:globally_ac} is true. To this end, since every nonnegative nonincreasing function on $(0, \infty)$ is an a.e.\ pointwise limit of a nondecreasing sequence of nonnegative nonincreasing simple functions with bounded support, it is easy to see that
\begin{equation}\label{thm:characterization_of_comp_emb_Y_locally_not_Linfty:eq:18}
\sup_{\|f\|_{\rep{X}}\leq 1} \|f^*\chi_{(a,\infty)}\|_{\rep{Y}} = \sup_{\substack{\|f\|_{\rep{X}}\leq 1\\ f\in \mathcal S}} \|f\chi_{(a,\infty)}\|_{\rep{Y}}
\end{equation}
for every $a>0$, where
\begin{equation*}
\mathcal S = \{f\in\Mpl(0,\infty): \text{$f$ is a nonincreasing simple function with bounded support}\}.
\end{equation*}
Using \cref{prop:construction_of_radially_symmetric_function_vanishing_at_infinity}, we have
\begin{equation}\label{thm:characterization_of_comp_emb_Y_locally_not_Linfty:eq:19}
\sup_{\substack{\|f\|_{\rep{X}}\leq 1\\ f\in \mathcal S}} \|f\chi_{(a,\infty)}\|_{\rep{Y}} \leq C_2 \sup_{\substack{\|f\|_{\rep{X}}\leq 1\\ f\in \mathcal S}} \|u_{f,a}\|_{Y(\rn)}
\end{equation}
for every $a>0$, where $u_{f,a}\in \wmrX$ is the function from \cref{prop:construction_of_radially_symmetric_function_vanishing_at_infinity} assigned to $f\in\mathcal S$, and the constant $C_2$ is that from \eqref{prop:construction_of_radially_symmetric_function_vanishing_at_infinity:eq:tail_of_f_bounded_by_norm_of_u_fa}. Furthermore, combining \eqref{thm:characterization_of_comp_emb_Y_locally_not_Linfty:eq:17} and \eqref{prop:construction_of_radially_symmetric_function_vanishing_at_infinity:eq:sobolev_norm_bounded_by_norm_of_f}, we obtain
\begin{equation}\label{thm:characterization_of_comp_emb_Y_locally_not_Linfty:eq:20}
\sup_{\substack{\|f\|_{\rep{X}}\leq 1\\ f\in \mathcal S}} \|u_{f,a}\|_{Y(\rn)} \leq C C_1 \quad \text{for every $a>0$}.
\end{equation}
Now, note that
\begin{equation*}
\sup_{\|f\|_{\rep{X}}\leq 1} \|f^*\chi_{(a,\infty)}\|_{\rep{Y}} < \infty \quad \text{for every $a>0$}
\end{equation*}
thanks to \eqref{thm:characterization_of_comp_emb_Y_locally_not_Linfty:eq:19} and \eqref{thm:characterization_of_comp_emb_Y_locally_not_Linfty:eq:20}. In view of this and \eqref{thm:characterization_of_comp_emb_Y_locally_not_Linfty:eq:18}, we can find a sequence of functions $\{f_j\}_{j = 1}^\infty\subseteq \mathcal S$ such that $\|f_j\|_{\rep{X}}\leq 1$ and 
\begin{equation*}
\sup_{\|f\|_{\rep{X}}\leq 1} \|f^*\chi_{(j,\infty)}\|_{\rep{Y}} \leq 2 \|f_j\chi_{(j,\infty)}\|_{\rep{Y}} \quad \text{for every $j\in\N$}.
\end{equation*}
Furthermore, combining this with \eqref{prop:construction_of_radially_symmetric_function_vanishing_at_infinity:eq:tail_of_f_bounded_by_norm_of_u_fa}, we obtain
\begin{equation}\label{thm:characterization_of_comp_emb_Y_locally_not_Linfty:eq:21}
\sup_{\|f\|_{\rep{X}}\leq 1} \|f^*\chi_{(j,\infty)}\|_{\rep{Y}} \leq 2 C_2 \|u_j\|_{Y(\rn)} \quad \text{for every $j\in\N$},
\end{equation}
where $u_j = u_{f_j,j}$ for every $j\in\N$.

Since the embedding $\wmrX \hookrightarrow Y(\rn)$ is compact and the sequence $\{u_j\}_{j = 1}^\infty\subseteq \wmrX$ is bounded in $\wmrX$ thanks to \eqref{thm:characterization_of_comp_emb_Y_locally_not_Linfty:eq:20}, we may without loss of generality assume there is a function $v\in Y(\rn)$ such that
\begin{align*}
\lim_{j \to \infty} u_j &= u \quad \text{in $Y(\rn)$}\\
\intertext{and}
\lim_{j \to \infty} u_j(x) &= u(x) \quad \text{for a.e.\ $x\in\rn$}.
\end{align*}
Therefore, since the latter together with \eqref{prop:construction_of_radially_symmetric_function_vanishing_at_infinity:eq:function_vanishes_as_a_increases} implies that $u = 0$, we have
\begin{equation}\label{thm:characterization_of_comp_emb_Y_locally_not_Linfty:eq:22}
\lim_{j \to \infty} \|u_j\|_{Y(\rn)} = 0.
\end{equation}
Finally, combining \eqref{thm:characterization_of_comp_emb_Y_locally_not_Linfty:eq:21} and \eqref{thm:characterization_of_comp_emb_Y_locally_not_Linfty:eq:22}, we see that
\begin{equation*}
\lim_{j\to \infty} \sup_{\|f\|_{\rep{X}}\leq 1} \|f^*\chi_{(j,\infty)}\|_{\rep{Y}} = 0,
\end{equation*}
whence \eqref{thm:characterization_of_comp_emb_Y:eq:globally_ac} follows.

Next, if $m > n$, there is nothing else to prove, whether \eqref{thm:characterization_of_comp_emb_Y:eq:fund_Y_vanishes_0} or \eqref{thm:characterization_of_comp_emb_Y:eq:fund_Y_not_vanishing_at_0} is the case. For future reference, note that for every \riSp{} $Z(\rn)$ such that $\fundX[Z](0^+) > 0$, we have
\begin{equation}\label{thm:characterization_of_comp_emb_Y_is_locally_Linfty:eq:1}
\|v\|_{L^\infty(\rn)} = \lim_{t\to0^+} v^*(t) \leq \lim_{t\to0^+} \frac1{\fundX[Z](t)} \|v\|_{Z(\rn)} = C_{Z} \|v\|_{Z(\rn)}
\end{equation}
for every $v\in\M(\rn)$, where $C_Z = 1/\fundX[Z](0^+) \in (0 ,\infty)$. When $m = n$ and \eqref{thm:characterization_of_comp_emb_Y:eq:fund_Y_vanishes_0} is the case, there is also nothing else to prove. Let $m = n$ and \eqref{thm:characterization_of_comp_emb_Y:eq:fund_Y_not_vanishing_at_0} be the case. We need to show that \eqref{thm:characterization_of_comp_emb_Y_is_locally_Linfty:eq:X_locally_not_L1} is valid. Supposed that \eqref{thm:characterization_of_comp_emb_Y_is_locally_Linfty:eq:X_locally_not_L1} is not valid. That together with \eqref{prel:fund_of_X_times_of_Xasoc} means that
\begin{equation*}
\lim_{t\to 0^+} \fundX[X'](t) > 0.
\end{equation*}
Hence
\begin{equation}\label{thm:characterization_of_comp_emb_Y_is_locally_Linfty:eq:2}
\|v\|_{L^\infty(\rn)} \leq C_{X'} \|v\|_{X'(\rn)} \quad \text{for every $v\in\M(\rn)$}
\end{equation}
thanks to~\eqref{thm:characterization_of_comp_emb_Y_is_locally_Linfty:eq:1} with $Z = X'$, where $C_{X'}\in(0, \infty)$. Furthermore, we have
\begin{align}
\|u\|_{X(\rn)} &= \sup_{\|v\|_{X'(\rn)} \leq 1} \int_{\rn} |u(x)| |v(x)| \dd x \nonumber\\
& \leq \|u\|_{L^1(\rn)} \sup_{\|v\|_{X'(\rn)} \leq 1} \|v\|_{L^\infty(\rn)} \nonumber\\
& \leq C_{X'} \|u\|_{L^1(\rn)} \label{thm:characterization_of_comp_emb_Y_is_locally_Linfty:eq:3}
\end{align}
for every $u\in\M(\rn)$ thanks to \eqref{prel:double_associate_norm}, the H\"older inequality \eqref{prel:Holder}, and \eqref{thm:characterization_of_comp_emb_Y_is_locally_Linfty:eq:2}. Now, let $g\in\mathcal C^\infty_0(\R)$ be an arbitrary smooth function such that $g\equiv 1$ on $[0, 1/2]$ and $g\equiv0$ on $[3/4, \infty)$. Set $u(x) = g(|x|)$, $x\in\rn$. It is easy to see that $u\in\mathcal C^\infty_0(\rn)$. Furthermore, set $u_j(x) = u(jx)$, $x\in\rn$, for every $j\in\N$. Note that
\begin{equation}\label{thm:characterization_of_comp_emb_Y_is_locally_Linfty:eq:4}
\spt u_j \subseteq B_{1/j} \quad \text{for every $j\in\N$}.
\end{equation}
Furthermore, it can be readily verified that
\begin{align*}
\|\nabla^k u_j\|_{L^1(\rn)} = j^k \|(\nabla^k u)(j x)\|_{L^1(\rn)} = j^{k - n} \|\nabla^k u\|_{L^1(\rn)} \leq \|\nabla^k u\|_{L^1(\rn)} < \infty
\end{align*}
for all $j\in\N$ and $k = 0, \dots, n$. Hence, thanks to this and \eqref{thm:characterization_of_comp_emb_Y_is_locally_Linfty:eq:3}, we see that $\{u_j\}_{j = 1}^\infty$ is a bounded sequence in $\wmrX[n]$. Now, we clearly have
\begin{equation*}
\|u_j\|_{L^\infty(\rn)} = \|u\|_{L^\infty(\rn)} \quad \text{for every $j\in\N$}.
\end{equation*}
Therefore, we obtain
\begin{equation*}
\|u_j\|_{Y(\rn)} \geq \frac{\|u\|_{L^\infty(\rn)}}{C_Y} > 0 \quad \text{for every $j\in\N$}
\end{equation*}
thanks to \eqref{thm:characterization_of_comp_emb_Y:eq:fund_Y_not_vanishing_at_0} and \eqref{thm:characterization_of_comp_emb_Y_is_locally_Linfty:eq:1} with $Z = Y$, where $C_Y\in(0, \infty)$. It follows from this and \eqref{thm:characterization_of_comp_emb_Y_is_locally_Linfty:eq:4} that $\{u_j\}_{j = 1}^\infty$ does not contain a subsequence convergent in $Y(\rn)$. Hence, the embedding $\wmrX[n] \hookrightarrow Y(\rn)$ is not compact.

Finally, let $1\leq m < n$. First, assume that \eqref{thm:characterization_of_comp_emb_Y:eq:fund_Y_vanishes_0} is true. We need to show that~\eqref{thm:characterization_of_comp_emb_Y_locally_not_Linfty:eq:Hardy_op_compact} is valid. Suppose that \eqref{thm:characterization_of_comp_emb_Y_locally_not_Linfty:eq:Hardy_op_compact} is not true. By \eqref{prel:representation_space_restriction}, we have
\begin{equation}\label{thm:characterization_of_comp_emb_Y_locally_not_Linfty:eq:23}
\lim_{a\to 0^+} \sup_{\|f\|_{\repFin{X_1}}\leq1} \Big\| \chi_{(0,a)}(t) \int_t^1 f^*(s) s^{-1 + \frac{m}{n}} \dd s \Big\|_{\repFin{Y_1}} > 0,
\end{equation}
where $X_1 = X(B_1)$ and $Y_1 = Y(B_1)$. By~\cref{rem:different_characterizations_of_ac_embedding_of_optimal_space}, \eqref{thm:characterization_of_comp_emb_Y_locally_not_Linfty:eq:23} is equivalent to the fact that \eqref{thm:embedding_on_balls_compactness_not_Linfty:ac} is not valid with $X=X_1$, $Y=Y_1$, $R = 1$, and $\alpha = 0$. Hence, \cref{thm:embedding_on_balls_compactness_not_Linfty} combined with \cref{rem:embedding_on_balls_compactness_functions_with_zero_boundary_data} implies that the embedding $\wmrXBvan[m][X][1]\hookrightarrow Y(B_1)$ is not compact, where $\wmrXBvan[m][X][1]$ is defined by \eqref{rem:embedding_on_balls_compactness_functions_with_zero_boundary_data:def}. However, we have $\wmrXBvan[m][X][1]\hookrightarrow \wmrX$ and $\|u\|_{Y(B_1)} = \|u\|_{Y(\rn)}$ for every $u\in \wmrXBvan[m][X][1]$, provided that we extend the functions from $\wmrXBvan[m][X][1]$ by $0$ outside $B_1$. Therefore, it follows that the embedding \eqref{thm:characterization_of_comp_emb_Y:eq:embedding} is not compact. Now, assume that \eqref{thm:characterization_of_comp_emb_Y:eq:fund_Y_not_vanishing_at_0} is the case. We will show that \eqref{thm:characterization_of_comp_emb_Y_is_locally_Linfty:eq:Hardy_op_compact} is satisfied, which will conclude the entire proof of~\cref{thm:characterization_of_comp_emb_Y}. By~\eqref{prel:representation_space_restriction}, it is easy to see that \eqref{thm:characterization_of_comp_emb_Y_is_locally_Linfty:eq:Hardy_op_compact} is satisfied if and only if
\begin{equation}\label{thm:characterization_of_comp_emb_Y_is_locally_Linfty:eq:5} 
\lim_{a\to 0^+} \sup_{\|f\|_{\repFin{X_1}}\leq1} \Big\| \int_t^1 \big( f^*(s)\chi_{(0,a)}(s) \big) s^{-1+\frac{m}{n}} \dd s \Big\|_{L^\infty(0,1)} = 0,
\end{equation}
where $X_1 = X(B_1)$. Since $\min\{f^*\chi_{(0,a)},N\}\nearrow f^*\chi_{(0,a)}$ as $N\to\infty$, we have
\begin{align*}
&\sup_{\|f\|_{\repFin{X_1}}\leq1} \Big\| \int_t^1 \big( f^*(s)\chi_{(0,a)}(s) \big) s^{-1+\frac{m}{n}} \dd s \Big\|_{L^\infty(0,1)} \\
&= \sup_{\substack{\|f\|_{\repFin{X_1}}\leq1\\ f\in L^\infty(0,1)}} \Big\| \int_t^1 \big( f^*(s)\chi_{(0,a)}(s) \big) s^{-1+\frac{m}{n}} \dd s \Big\|_{L^\infty(0,1)}
\end{align*}
for every $a\in(0,1)$. Furthermore, thanks to \eqref{thm:characterization_of_comp_emb_Y:eq:fund_Y_not_vanishing_at_0} and \eqref{thm:characterization_of_comp_emb_Y_is_locally_Linfty:eq:1} with $Z = Y$, the compactness of \eqref{thm:characterization_of_comp_emb_Y:eq:embedding} implies that the embedding $\wmrX\hookrightarrow L^\infty(\rn)$ is also compact. In particular, it is bounded. Therefore, the suprema are finite. Hence, for each $j\in\N$, there is a function $f_j\in L^\infty(0,1)$ such that $\|f_j\|_{\repFin{X_1}}\leq1$ and
\begin{align}
&\sup_{\|f\|_{\repFin{X_1}}\leq1} \Big\| \int_t^1 \big( f^*(s)\chi_{(0,1/j)}(s) \big) s^{-1+\frac{m}{n}} \dd s \Big\|_{L^\infty(0,1)} \nonumber\\
&\leq 2 \Big\| \int_t^1 \big( f_j^*(s)\chi_{(0,1/j)}(s) \big) s^{-1+\frac{m}{n}} \dd s \Big\|_{L^\infty(0,1)}. \label{thm:characterization_of_comp_emb_Y_is_locally_Linfty:eq:6} 
\end{align}
Now, thanks to \cref{prop:construction_of_radially_symmetric_function_local_necessity_Hardy} with $Y = L^\infty$, $R = 1$, $\alpha = 0$, $f = f_j^*\chi_{(0,1/j)}$, and $a = 2/j$, there are functions $u_j = u_{f_j,1,2/j}$, $j\in\N$, $j\geq3$, such that (see also~\cref{rem:continuation_outside_by_zero})
\begin{align*}
&\text{the sequence $\{u_j\}_{j = 3}^\infty$ is bounded in $\wmrX$}; \\
&\lim_{j\to\infty} u_j(x) = 0 \quad \text{for every $x\in\rn\setminus\{0\}$}; \\
&\Big\| \int_t^1 \big( f_j^*(s)\chi_{(0,1/j)}(s) \big) s^{-1+\frac{m}{n}} \dd s \Big\|_{L^\infty(0,1)} \leq C_4 \|u_j\|_{L^\infty(\rn)},
\end{align*}
where $C_4$ is the constant from \eqref{prop:construction_of_radially_symmetric_function_local_necessity_Hardy:eq:Hardy_bounded_by_u}. The compactness of the embedding $\wmrX\hookrightarrow L^\infty(\rn)$ now implies that
\begin{equation*}
\lim_{j\to\infty} \Big\| \int_t^1 \big( f_j^*(s)\chi_{(0,1/j)}(s) \big) s^{-1+\frac{m}{n}} \dd s \Big\|_{L^\infty(0,1)} = 0.
\end{equation*}
Combining this with \eqref{thm:characterization_of_comp_emb_Y_is_locally_Linfty:eq:6}, we obtain the validity of \eqref{thm:characterization_of_comp_emb_Y_is_locally_Linfty:eq:5}. Hence, so is \eqref{thm:characterization_of_comp_emb_Y_is_locally_Linfty:eq:Hardy_op_compact} in turn, which finishes the proof.
\end{proof}


\section{Examples}\label{sec:examples}
In this section, we completely characterize the compactness of the Sobolev embedding \eqref{thm:characterization_of_comp_emb_Y:eq:embedding} and of the weighted Sobolev embedding \eqref{E:embedding_on_balls_Sob_emb} on balls for concrete examples of \riSps{}. More specifically, we consider the situation where both $X$ and $Y$ are from the class of \emph{Lorentz--Zygmund spaces}. This class of function spaces contains many customary function spaces. For example, it contains Lebesgue spaces, Lorentz space, or some Orlicz spaces\textemdash namely those of logarithmic and exponential type (we will get back to this in more detail soon). We start by briefly introducing the Lorentz--Zygmund spaces (the interested reader can find more information in~\cite{BR:80, OP:99}).

For $p,q\in[1, \infty]$ and $\A=(\alpha_0, \alpha_\infty)\in\R^2$, we define the functional $\|\cdot\|_{L^{p,q,\A}(0, \infty)}$ as
\begin{equation*}
\|f\|_{L^{p,q,\A}(0, \infty)} = \|t^{\frac1{p} - \frac1{q}}\ell^\A(t) f^*(t)\|_{L^q(0, \infty)},\ f\in\Mpl(0, \infty),
\end{equation*}
where the functions $\ell\colon(0, \infty)\to(0, \infty)$ and $\ell^\A$ are defined as
\begin{align}
\ell(t) &= 1 + |\log(t)|,\ t\in(0, \infty), \nonumber
\intertext{and}
\ell^\A(t) &= \begin{cases}
	\ell(t)^{\alpha_0} \quad &\text{if $t\in(0, 1]$},\\
	\ell(t)^{\alpha_\infty} \quad &\text{if $t\in(1, \infty)$}.
\end{cases} \label{EQ:broken_log_A}
\end{align}
The functional $\|\cdot\|_{L^{p,q,\A}(0, \infty)}$ is equivalent to an \riNorm{} if and only if one of the following is true:
\begin{equation}\label{EQ:LZ_equiv_RI}
\begin{cases}
p=q=1, \alpha_0\geq0, \text{ and } \alpha_\infty\leq0;\\
p\in(1, \infty) \text{ and } q\in[1, \infty];\\
p=\infty, q\in[1, \infty), \text{ and } \alpha_0 + \frac1{q} < 0;\\
p=q=\infty \text{ and } \alpha_0\leq 0.
\end{cases}
\end{equation}
When this is the case, we define the Lorentz--Zygmund space $L^{p,q,\A}(\rn)$ as
\begin{equation*}
L^{p,q,\A}(\rn) = \{f\in\M(\rn): \|f\|_{L^{p,q,\A}(\rn)} = \|f^*\|_{L^{p,q,\A}(0, \infty)} < \infty\}
\end{equation*}
and consider it an \riSp{}.

Note that we have $L^p(\rn) = L^{p,p,(0,0)}(\rn)$ and $L^{p,q}(\rn) = L^{p,q,(0,0)}(\rn)$. Furthermore, when $p=q\in[1, \infty]$, the Lorentz--Zygmund spaces $L^{p,p,\A}(\rn)$ coincide with the following well-known and oft-used Orlicz spaces, up to equivalent norms. When $p=q\in[1, \infty)$, we have $L^{p,p,\A}(\rn)=L^A(\rn)$ with $A$ being equivalent to the function $(0, \infty)\ni t\mapsto t^p\ell(t)^{(p\alpha_\infty, p\alpha_0)}$ (here we assume $\alpha_0\geq0$ and $\alpha_\infty\leq0$ when $p=q=1$). When $p=q=\infty$, $\alpha_0<0$, and $\alpha_\infty > 0$, we have $L^{\infty,\infty,\A}(\rn)=L^A(\rn)$ with $A$ being equivalent to the function $t\mapsto \exp(-t^{-1/\alpha_\infty})$ near $0$ and to the function $t\mapsto\exp(t^{-1/\alpha_0})$ near infinity.

Sometimes it is necessary to work with more tiers of logarithms\textemdash particularly to capture various delicate limiting situations. We define the function $\ell\ell\colon(0, \infty) \to (0, \infty)$ as
\begin{equation*}
\ell\ell(t) = 1+\log(\ell(t)),\ t\in(0, \infty),
\end{equation*}
and $\ell\ell^\B$, for $\B=(\beta_0, \beta_\infty)\in\R^2$, as in \eqref{EQ:broken_log_A} with $\ell$ and $\A$ replaced by $\ell\ell$ and $\B$, respectively. We can then define $\|f\|_{L^{p,q,\A,\B}(0, \infty)}$ as
\begin{equation*}
\|f\|_{L^{p,q,\A,\B}(0, \infty)} = \|t^{\frac1{p} - \frac1{q}}\ell^\A(t)\ell\ell^\B(t) f^*(t)\|_{L^q(0, \infty)},\ f\in\Mpl(0, \infty).
\end{equation*}
When even more tiers of logarithms are needed, they can be added in the obvious way.

Finally, when $(E, \mu)$ is a finite nonatomic measure space, we define, for $p,q\in[1, \infty]$ and $\alpha\in\R$, the Lorentz--Zygmund space $L^{p,q,\alpha}(E,\mu)$ as
\begin{equation*}
L^{p,q,\alpha}(E,\mu) = \{f\in\M(E,\mu): \|f\|_{L^{p,q,\alpha}(E,\mu)} = \|f^*\chi_{(0, \mu(E))}\|_{L^{p,q,(\alpha,0)}(0, \infty)} < \infty\}.
\end{equation*}
The functional $\|\cdot\|_{L^{p,q,\alpha}(E,\mu)}$ is equivalent to an \riNorm{} if and only if one of the conditions \eqref{EQ:LZ_equiv_RI} is satisfied with $\alpha_0 = \alpha$ and $\alpha_\infty$ disregarded. When this is the case, we consider $L^{p,q,\alpha}(E,\mu)$ an \riSp{}.
\subsection{Sobolev embeddings on the entire space}
Having introduce the Lorentz--Zygmund spaces, we now completely characterize the compactness of the Sobolev embedding
\begin{equation}\label{thm:compactness_entire_space_LZ_examples:Sob_emb}
\wmrX[m][L^{p,q,\A}] \hookrightarrow L^{r,s,\B}(\rn),
\end{equation}
which covers a large number of usual situations (see also \cref{thm:compactness_entire_space_LLogL_examples}).
\begin{theorem}\label{thm:compactness_entire_space_LZ_examples}
Let $m,n\in\N$, $n\geq2$. Let $p,q,r,s\in[1, \infty]$ and $\A = (\alpha_0, \alpha_\infty)$, $\B = (\beta_0, \beta_\infty)\in\R^2$. Assume that the parameters are such that $L^{p,q,\A}(\rn)$ and $L^{r,s,\B}(\rn)$ are (equivalent to) \riSps{} (recall~\eqref{EQ:LZ_equiv_RI}). Consider the following nine conditions:
\begin{enumerate}[label=(C\arabic*), ref=(C\arabic*)]
	\item\label{enum:thm:compactness_entire_space_LZ_examples_AC_near_infinity:1} $p < r$;
	\item\label{enum:thm:compactness_entire_space_LZ_examples_AC_near_infinity:2} $p = r < \infty$ and $\beta_\infty + \max\{1/s - 1/q, 0\} < \alpha_\infty$;
	\item $p = r = \infty$, $\alpha_\infty + 1/q > 0$, and $\alpha_\infty + 1/q > \beta_\infty + 1/s$;
	\item\label{enum:thm:compactness_entire_space_LZ_examples_AC_near_infinity:4} $p = r = \infty$, $q < \infty$, and $0 = \alpha_\infty + 1/q > \beta_\infty + 1/s$;
	\item\label{enum:thm:compactness_entire_space_LZ_examples_AC_near_infinity:5} $p = r = \infty$, $q < s$, and $\alpha_\infty + 1/q = \beta_\infty + 1/s = 0$;
	\item\label{enum:thm:compactness_entire_space_LZ_locally_compact:1} $p\leq \frac{n}{m}$ and $r<\frac{np}{n-mp}$;
	\item $p < \frac{n}{m}$, $r = \frac{np}{n-mp}$, and $\beta_0 < \alpha_0 + \min\Big\{ \frac1{q}-\frac1{s},0 \Big\}$;
	\item $p = \frac{n}{m}$, $r = \infty$, $\alpha_0 \leq 1 - \frac1{q}$, and $\beta_0 < \alpha_0 - 1 +\frac1{q}-\frac1{s}$;
	\item\label{enum:thm:compactness_entire_space_LZ_locally_compact:4} either $p = \frac{n}{m}$ and $\alpha_0 > 1 - \frac1{q}$ or $p > \frac{n}{m}$.
\end{enumerate}

When $m<n$, the Sobolev embedding \eqref{thm:compactness_entire_space_LZ_examples:Sob_emb} is compact if and only if one of the conditions \ref{enum:thm:compactness_entire_space_LZ_examples_AC_near_infinity:1}--\ref{enum:thm:compactness_entire_space_LZ_examples_AC_near_infinity:5} is satisfied and simultaneously so is one of \ref{enum:thm:compactness_entire_space_LZ_locally_compact:1}--\ref{enum:thm:compactness_entire_space_LZ_locally_compact:4}.

When $m=n$, the Sobolev embedding \eqref{thm:compactness_entire_space_LZ_examples:Sob_emb} is compact if and only if one of the conditions \ref{enum:thm:compactness_entire_space_LZ_examples_AC_near_infinity:1}--\ref{enum:thm:compactness_entire_space_LZ_examples_AC_near_infinity:5} is satisfied and simultaneously either $p=q=1$ and $\alpha_0=0$ is not true or $r = s = \infty$ and $\beta_0 = 0$ is not true.

When $m>n$, the Sobolev embedding \eqref{thm:compactness_entire_space_LZ_examples:Sob_emb} is compact if and only if one of the conditions \ref{enum:thm:compactness_entire_space_LZ_examples_AC_near_infinity:1}--\ref{enum:thm:compactness_entire_space_LZ_examples_AC_near_infinity:5} is satisfied.
\end{theorem}
\begin{proof}
The theorem follows from \cref{thm:characterization_of_comp_emb_Y} combined with \cref{thm:weighted_Sob_ball_LZ_examples_compactness} with $\alpha = 0$ and \cref{prop:ac_near_infinity_LZ_examples}, which will be proved in the rest of this section.
\end{proof}

Recalling \cref{thm:characterization_of_comp_emb_Y}, we see that the validity of \eqref{thm:characterization_of_comp_emb_Y:eq:globally_ac} plays an essential role in the question of whether the Sobolev embedding \eqref{thm:characterization_of_comp_emb_Y:eq:embedding} is compact. Therefore, in order to characterize the compactness of \eqref{thm:compactness_entire_space_LZ_examples:Sob_emb}, we need to know when \eqref{thm:characterization_of_comp_emb_Y:eq:globally_ac} is valid when both $X$ and $Y$ are Lorentz--Zygmund spaces. Before we characterize it, we establish the following simple auxiliary proposition, which will be useful later.
\begin{proposition}\label{prop:cutting_of_weighted_norm}
Let $X(0, \infty)$ and $Z(0, \infty)$ be \riSps{}. Let $\|\cdot\|_{Y(0, \infty)}$ be defined as
\begin{equation}\label{prop:cutting_of_weighted_norm:def_of_Y}
\|g\|_{Y(0, \infty)} = \|g^*(t) w(t)\|_{Z(0, \infty)},\ g\in\M(0, \infty),
\end{equation}
where $w\in\Mpl(0, \infty)$. Then
\begin{equation*}
\sup_{\|f\|_{X(0, \infty)}\leq1} \|f^*\chi_{(a, \infty)}\|_{Y(0, \infty)} \leq \frac{\fundX[Y](a)}{\fundX(a)} + \sup_{\|f\|_{X(0, \infty)}\leq1} \|f^*(t)w(t)\chi_{(a, \infty)}(t)\|_{Z(0, \infty)}
\end{equation*}
for every $a\in(0, \infty)$.
\end{proposition}
\begin{proof}
Fix $a>0$. Let $f\in\M(0, \infty)$. Note that
\begin{equation*}
f^*(a) = \frac1{\fundX(a)} \|f^*(a)\chi_{(0,a)}(t)\|_{X(0, \infty)} \leq \frac{\|f\|_{X(0, \infty)}}{\fundX(a)}.
\end{equation*}
Using this together with \eqref{prop:cutting_of_weighted_norm:def_of_Y}, we see that
\begin{align*}
\|f^*\chi_{(a, \infty)}\|_{Y(0, \infty)} &\leq \|(f^*(a) \chi_{(0,a)} + f^*\chi_{(a, \infty)})^*(t) w(t)\|_{Z(0, \infty)} \\
&= \|(f^*(a) \chi_{(0,a)}(t) + f^*(t)\chi_{(a, \infty)}(t)) w(t)\|_{Z(0, \infty)} \\
&\leq \|f^*(a) \chi_{(0,a)}(t) w(t)\|_{Z(0, \infty)} + \|f^*(t)w(t)\chi_{(a, \infty)}(t)\|_{Z(0, \infty)} \\
&\leq \frac{\|\chi_{(0,a)}(t) w(t)\|_{Z(0, \infty)}}{\fundX(a)} \|f\|_{X(0, \infty)} + \|f^*(t)w(t)\chi_{(a, \infty)}(t)\|_{Z(0, \infty)} \\
&= \frac{\|\chi_{(0,a)}\|_{Y(0, \infty)}}{\fundX(a)} \|f\|_{X(0, \infty)} + \|f^*(t)w(t)\chi_{(a, \infty)}(t)\|_{Z(0, \infty)},
\end{align*}
whence the desired inequality follows.
\end{proof}

\begin{remark}
Note that we abuse the notation of a fundamental function in \cref{prop:cutting_of_weighted_norm} because the functional $\|\cdot\|_{Y(0, \infty)}$ defined by \eqref{prop:cutting_of_weighted_norm:def_of_Y} need not be an \riNorm{} for a general weight $w\in\Mpl(0, \infty)$. Nevertheless, the functional is clearly rearrangement invariant, and so we can define the corresponding fundamental function in the obvious way. However, this function need not have the pleasant properties that the fundamental function of an \riSp{} has.
\end{remark}

\begin{proposition}\label{prop:ac_near_infinity_LZ_examples}
Let $p,q,r,s\in[1, \infty]$ and $\A = (\alpha_0, \alpha_\infty)$, $\B = (\beta_0, \beta_\infty)\in\R^2$. Assume that the parameters are such that $L^{p,q,\A}(\rn)$ and $L^{r,s,\B}(\rn)$ are (equivalent to) \riSps{} (recall~\eqref{EQ:LZ_equiv_RI}). Then
\begin{equation}\label{prop:ac_near_infinity_LZ_examples:eq1}
\lim_{a\to\infty}\sup_{\|f\|_{\LZ{p}{q}{\A}}\leq1} \|f^*\chi_{(a, \infty)}\|_{\LZ{r}{s}{\B}} = 0
\end{equation}
if and only if one of the conditions \ref{enum:thm:compactness_entire_space_LZ_examples_AC_near_infinity:1}--\ref{enum:thm:compactness_entire_space_LZ_examples_AC_near_infinity:5} from~\cref{thm:compactness_entire_space_LZ_examples} is satisfied.
\end{proposition}
\begin{proof}
Throughout the entire proof, we use the notation $\lesssim$ or $\gtrsim$ meaning that the left-hand side is bounded from above or below, respectively, by a positive constant multiple of the right-hand side with the multiplicative constant independent of $a\in(0, \infty)$ and $f$. We will also use $\approx$, meaning that both inequalities $\lesssim$ and $\gtrsim$ are true.

By \cref{prop:cutting_of_weighted_norm}, in order to show that \eqref{prop:ac_near_infinity_LZ_examples:eq1} is true, it is sufficient to show that
\begin{align}
&\lim_{a\to\infty} \frac{\fundX[{\LZ[]{r}{s}{\B}}](a)}{\fundX[{\LZ[]{p}{q}{\A}}](a)} = 0 \label{prop:ac_near_infinity_LZ_examples:eq12} \\
\intertext{and}
&\lim_{a\to\infty} \sup_{\|f\|_{\LZ{p}{q}{\A}}\leq1} \|t^{\frac1{r} - \frac1{s}} \ell(t)^{\beta_\infty} f^*(t)\|_{L^s(a, \infty)} = 0. \label{prop:ac_near_infinity_LZ_examples:eq13}
\end{align}
It is not hard to verify that each of the conditions~\ref{enum:thm:compactness_entire_space_LZ_examples_AC_near_infinity:1}--\ref{enum:thm:compactness_entire_space_LZ_examples_AC_near_infinity:5} implies the validity of~\eqref{prop:ac_near_infinity_LZ_examples:eq12} (see also~\cite[Lemma~3.7]{OP:99}). We now show that they also imply the validity of~\eqref{prop:ac_near_infinity_LZ_examples:eq13}.

First, assume that $p < r \leq \infty$, that is, we assume~\ref{enum:thm:compactness_entire_space_LZ_examples_AC_near_infinity:1}. Since $\LZ{p}{q}{\A} \hookrightarrow \LZ{p}{\infty}{\A}$ (see~\cite[Theorem~4.1]{OP:99}), in order to establish \eqref{prop:ac_near_infinity_LZ_examples:eq13}, it is sufficient to show that
\begin{equation*}
\lim_{a\to\infty}\sup_{\|f\|_{\LZ{p}{\infty}{\A}}\leq1} \|t^{\frac1{r} - \frac1{s}} \ell(t)^{\beta_\infty} f^*(t)\|_{L^s(a, \infty)} = 0.
\end{equation*}
To this end, we have
\begin{equation}\label{prop:ac_near_infinity_LZ_examples:eq3}
\|t^{\frac1{r} - \frac1{s}} \ell(t)^{\beta_\infty} f^*(t)\|_{L^s(a, \infty)} \leq \|f\|_{\LZ{p}{\infty}{\A}} \|t^{\frac1{r} - \frac1{p} - \frac1{s}} \ell(t)^{\beta_\infty - \alpha_\infty}\|_{L^s(a, \infty)}
\end{equation}
for every $f\in\M(0, \infty)$. We now distinguish between whether $s< \infty$ or $s=\infty$. Assume that $s<\infty$. Note that
\begin{equation*}
	\|t^{\frac1{r} - \frac1{p} - \frac1{s}} \ell(t)^{\beta_\infty - \alpha_\infty}\|_{L^s(1, \infty)} < \infty
\end{equation*}
thanks to the fact that $(1/r-1/p)s < 0$. Therefore, using \eqref{prop:ac_near_infinity_LZ_examples:eq3} and the dominated convergence theorem, we obtain the validity of \eqref{prop:ac_near_infinity_LZ_examples:eq13}. Now, assume that $s=\infty$. Note that the function $(0, \infty)\ni \tau \mapsto \tau^{\frac1{r} - \frac1{p}} \ell(\tau)^{\beta_\infty - \alpha_\infty}$ is equivalent to a decreasing function thanks to the fact that $1/r-1/p<0$. Hence
\begin{equation*}
\|t^{\frac1{r} - \frac1{p}} \ell(t)^{\beta_\infty - \alpha_\infty}\|_{L^\infty(a, \infty)} \approx a^{\frac1{r} - \frac1{p}} \ell(a)^{\beta_\infty - \alpha_\infty} \to 0 \qquad \text{as $a\to\infty$}.
\end{equation*}
Therefore, combining this with \eqref{prop:ac_near_infinity_LZ_examples:eq3}, we obtain the validity of \eqref{prop:ac_near_infinity_LZ_examples:eq13} again.

Second, assume one of the cases~\ref{enum:thm:compactness_entire_space_LZ_examples_AC_near_infinity:2}--\ref{enum:thm:compactness_entire_space_LZ_examples_AC_near_infinity:4}. By~\cite[Theorems~4.1 and 4.3]{OP:99} when $p=r<\infty$ and by~\cite[Theorems~4.2 and 4.4]{OP:99} when $p=r=\infty$, we have
\begin{equation}\label{prop:ac_near_infinity_LZ_examples:eq4}
\LZ{p}{q}{\A} \hookrightarrow \LZ{p}{s}{(\gamma_0, \gamma_\infty)},
\end{equation}
where
\begin{equation*}
(\gamma_0, \gamma_\infty) = \begin{cases}
( \alpha_0, \alpha_\infty ) \quad &\text{if $p=r<\infty$ and $q\leq s$},\\[3pt]
\big( \alpha_0 + 1/q - 1/s - \varepsilon, \alpha_\infty + 1/q - 1/s - \varepsilon \big) \quad &\text{if either $p=r<\infty$ and $q > s$},\\[3pt]
&\text{\hphantom{if }or $p=r = \infty$},
\end{cases} 
\end{equation*}
where, in the second case, $\varepsilon>0$ is (arbitrarily) chosen to be so small that we still have
\begin{equation*}
\beta_\infty + \frac1{s} < \alpha_\infty + \frac1{q} - \varepsilon.
\end{equation*}
Note that, with such a choice of $\varepsilon>0$, we also have
\begin{equation}\label{prop:ac_near_infinity_LZ_examples:eq5}
\beta_\infty - \gamma_\infty < 0
\end{equation}
in each of the cases~\ref{enum:thm:compactness_entire_space_LZ_examples_AC_near_infinity:2}--\ref{enum:thm:compactness_entire_space_LZ_examples_AC_near_infinity:4}. Hence, it follows that
\begin{equation}\label{prop:ac_near_infinity_LZ_examples:eq6}
\lim_{a\to\infty}\ell(a)^{\beta_\infty-\gamma_\infty} = 0.
\end{equation}
Now, thanks to the embedding \eqref{prop:ac_near_infinity_LZ_examples:eq4}, in order to show that \eqref{prop:ac_near_infinity_LZ_examples:eq13} is true, it is sufficient to show that
\begin{equation}\label{prop:ac_near_infinity_LZ_examples:eq7}
\lim_{a\to\infty} \sup_{\|f\|_{\LZ{p}{s}{(\gamma_0, \gamma_\infty)}}\leq1} \|t^{\frac1{p} - \frac1{s}} \ell(t)^{\beta_\infty} f^*(t)\|_{L^s(a, \infty)} = 0.
\end{equation}
To this end, using both \eqref{prop:ac_near_infinity_LZ_examples:eq5} and \eqref{prop:ac_near_infinity_LZ_examples:eq6}, we obtain
\begin{align*}
&\lim_{a\to \infty} \sup_{\|f\|_{\LZ{p}{s}{(\gamma_0, \gamma_\infty)}}\leq1} \|t^{\frac1{p} - \frac1{s}} \ell(t)^{\beta_\infty} f^*(t)\|_{L^s(a, \infty)} \\
&\leq \lim_{a\to \infty} \Big( \sup_{\tau\in(a, \infty)}\ell(\tau)^{\beta_\infty-\gamma_\infty} \sup_{\|f\|_{\LZ{p}{s}{(\gamma_0, \gamma_\infty)}}\leq1} \|t^{\frac1{p} - \frac1{s}} \ell(t)^{\gamma_\infty} f^*(t)\|_{L^s(a, \infty)} \Big) \\
&\leq \lim_{a\to \infty} \ell(a)^{\beta_\infty-\gamma_\infty} = 0.
\end{align*}
It follows that \eqref{prop:ac_near_infinity_LZ_examples:eq7} is true. Hence, so is \eqref{prop:ac_near_infinity_LZ_examples:eq13} in turn.

Next, assume~\ref{enum:thm:compactness_entire_space_LZ_examples_AC_near_infinity:5}, that is, $p = r = \infty$, $q < s$, and $\alpha_\infty + 1/q = 0$. Since $q \alpha_\infty = -1$ and $q\alpha_0 < - 1$ (recall~\eqref{EQ:LZ_equiv_RI}), we have
\begin{equation*}
\int_0^t \tau^{-1} \ell^{q\A}(\tau) \dd{\tau} \approx \ell\ell(t) \quad \text{for every $t\geq1$}.
\end{equation*}
Using this, we see that
\begin{align*}
f^*(t)^q &\approx \ell\ell(t)^{-1} f^*(t)^q \int_0^t \tau^{-1} \ell^{q\A}(\tau) \dd{\tau} \\
&\leq \ell\ell(t)^{-1} \int_0^t \tau^{-1} \ell^{q\A}(\tau) f^*(\tau)^q \dd{\tau} \\
&\leq \ell\ell(t)^{-1} \|f\|_{\LZ{\infty}{q}{\A}}^q
\end{align*}
for all $t\geq1$ and $f\in\M(0, \infty)$. Hence, using this, we obtain
\begin{equation*}
\lim_{a\to\infty} \sup_{\|f\|_{\LZ{\infty}{q}{\A}}\leq1} \|t^{- \frac1{s}} \ell(t)^{\beta_\infty} f^*(t)\|_{L^s(a, \infty)} \leq \lim_{a\to\infty} \|t^{- \frac1{s}} \ell(t)^{\beta_\infty} \ell\ell(t)^{-\frac1{q}}\|_{L^s(a, \infty)}.
\end{equation*}
Since $\beta_\infty + 1/s = 0$ and $1\leq q < s \leq \infty$, it is easy to see that
\begin{equation*}
\|t^{- \frac1{s}} \ell(t)^{\beta_\infty} \ell\ell(t)^{-\frac1{q}}\|_{L^s(a, \infty)} \approx \ell\ell(a)^{\frac1{s} - \frac1{q}} \quad \text{for every $a\geq1$}.
\end{equation*}
Hence, it follows that \eqref{prop:ac_near_infinity_LZ_examples:eq13} is also true when \ref{enum:thm:compactness_entire_space_LZ_examples_AC_near_infinity:5} is the case.

Finally, we will show that the conditions \ref{enum:thm:compactness_entire_space_LZ_examples_AC_near_infinity:1}--\ref{enum:thm:compactness_entire_space_LZ_examples_AC_near_infinity:5} are also necessary for the validity of \eqref{prop:ac_near_infinity_LZ_examples:eq1}. When
\begin{align*}
&\text{$p>r$,} \\
&\text{$p=r<\infty$ and $\beta_\infty \geq \alpha_\infty$}, \\
&\text{$p=r=\infty$ and $0 < \alpha_\infty + \frac1{q} \leq \beta_\infty + \frac1{s}$}, \\
&\text{$p=r=\infty$ and $0 = \alpha_\infty + \frac1{q} < \beta_\infty + \frac1{s}$}, \\
&\text{$p=r=\infty$, $q<\infty$, and $\alpha_\infty + \frac1{q} < 0$}, \\
&\text{$p=r=\infty$, $q=\infty$, and $\alpha_\infty \leq 0$}, \\
\intertext{or}
&\text{$p=r=\infty$, $s\leq q$, and $\alpha_\infty + \frac1{q} = \beta_\infty + \frac1{s} = 0$},
\end{align*}
it can be verified by straightforward computations (see also~\cite[Lemma~3.7]{OP:99}) that
\begin{equation*}
\lim_{a\to\infty} \frac{\fundX[{\LZ[]{r}{s}{\B}}](a)}{\fundX[{\LZ[]{p}{q}{\A}}](a)} > 0.
\end{equation*}
Hence, it follows from \cref{prop:decay_at_infinity_implies_decay_of_fund_func} that \eqref{prop:ac_near_infinity_LZ_examples:eq1} is not valid when one of these conditions is satisfied. Furthemore, comparing them with \ref{enum:thm:compactness_entire_space_LZ_examples_AC_near_infinity:1}--\ref{enum:thm:compactness_entire_space_LZ_examples_AC_near_infinity:5}, we see that it only remains for us to show that if
\begin{equation}\label{prop:ac_near_infinity_LZ_examples:eq8}
\text{$p=r<\infty$, $q>s$, and $\beta_\infty< \alpha_\infty \leq \beta_{\infty} + \frac1{s} - \frac1{q}$},
\end{equation}
then \eqref{prop:ac_near_infinity_LZ_examples:eq1} is not valid. For future reference,  note that
\begin{equation}\label{prop:ac_near_infinity_LZ_examples:eq9}
\lim_{t\to\infty} \frac{t+a}{t} = \lim_{t\to\infty} \frac{\ell(t+a)}{\ell(t)} = 1 \quad \text{for every $a>0$}.
\end{equation}
First, assume \eqref{prop:ac_near_infinity_LZ_examples:eq8} and that $q = \infty$. Set
\begin{equation*}
f_0(t) = t^{-\frac1{p}} \ell^{-\A}(t),\ t\in(0, \infty).
\end{equation*}
Since $p<\infty$, it is easy to see that $f_0^*(t)\approx f_0(t)$ for every $t\in(0, \infty)$. Furthermore, note that $\|f_0\|_{\LZ{p}{q}{\A}}\approx 1$. It follows from \eqref{prop:ac_near_infinity_LZ_examples:eq9} that
for each $a>0$ there is $t_a\geq1$ such that
\begin{equation*}
(t+a)^{-\frac1{p}} \ell^{-\A}(t+a) \geq \frac1{2}t^{-\frac1{p}} \ell^{-\A}(t) \quad \text{for every $t\geq t_a$}.
\end{equation*}
Hence, for every $a>0$, we have
\begin{align*}
\sup_{\|f\|_{\LZ{p}{q}{\A}}\leq1} \|f^*\chi_{(a, \infty)}\|_{\LZ{r}{s}{\B}} &\gtrsim \|f_0\chi_{(a, \infty)}\|_{\LZ{r}{s}{\B}} \\
&\approx \|t^{\frac1{r} - \frac1{s}} \ell^\B(t) (t+a)^{-\frac1{p}} \ell^{-\A}(t+a)\|_{L^s(0, \infty)} \\
&\gtrsim \|t^{- \frac1{s}} \ell^\B(t) \ell^{-\A}(t)\|_{L^s(t_a, \infty)} \\
&= \|t^{- \frac1{s}} \ell(t)^{\beta_\infty - \alpha_\infty}\|_{L^s(t_a, \infty)} \\
&= \infty
\end{align*}
thanks to the fact that $\beta_\infty + 1/s \geq \alpha_\infty$. Therefore, \eqref{prop:ac_near_infinity_LZ_examples:eq1} is not valid.

At last, assume \eqref{prop:ac_near_infinity_LZ_examples:eq8} and that $q< \infty$. To this end, note that
\begin{align}
\sup_{\|f\|_{\LZ{p}{q}{\A}}\leq1} \|f^*\chi_{(a, \infty)}\|_{\LZ{r}{s}{\B}} &= \sup_{\|f\|_{\LZ{p}{q}{\A}}\leq1} \|t^{\frac1{r} - \frac1{s}} \ell^\B(t)f^*(t+a)\|_{L^s(0 ,\infty)} \nonumber\\
&= \sup_{\|f\|_{\LZ{p}{q}{\A}}\leq1} \|(t-a)^{\frac1{r} - \frac1{s}} \ell^\B(t-a)f^*(t)\|_{L^s(a ,\infty)}. \label{prop:ac_near_infinity_LZ_examples:eq10}
\end{align}
By \cite[Theorem~1]{S:90}, we have
\begin{align}
&\sup_{\|f\|_{\LZ{p}{q}{\A}}\leq1} \|(t-a)^{\frac1{r} - \frac1{s}} \ell^\B(t-a)f^*(t)\|_{L^s(a ,\infty)} \approx \nonumber\\
&\Bigg( \int_0^\infty \Bigg( \int_t^\infty \frac{(\tau - a)^{\frac{s}{p} - 1} \ell^{s\B}(\tau - a)}{\tau^{\frac{q}{p}} \ell^{q\A}(\tau)} \chi_{(a, \infty)}(\tau) \dd{\tau} \Bigg)^{\frac{q}{q-s}} t^{\frac{q}{p} - 1} \ell^{q\A}(t) \dd{t} \Bigg)^\frac{q-s}{qs}. \label{prop:ac_near_infinity_LZ_examples:eq11}
\end{align}
Using \eqref{prop:ac_near_infinity_LZ_examples:eq9}, we see that for each $a>0$ there is $t_a\geq a$ such that
\begin{equation*}
(t - a)^{\frac{s}{p} - 1} \ell^{s\B}(t - a) \geq \frac1{2} t^{\frac{s}{p} - 1} \ell^{s\B}(t) \quad \text{for every $t\geq t_a$}.
\end{equation*}
Therefore, for each $a\geq1$ and every $t\geq t_a\geq a$, we have
\begin{align*}
\int_t^\infty \frac{(\tau - a)^{\frac{s}{p} - 1} \ell^{s\B}(\tau - a)}{\tau^{\frac{q}{p}} \ell^{q\A}(\tau)} \chi_{(a, \infty)}(\tau) \dd{\tau} &=
\int_t^\infty \frac{(\tau - a)^{\frac{s}{p} - 1} \ell(\tau - a)^{s\beta_\infty}}{\tau^{\frac{q}{p}} \ell(\tau)^{q\alpha_\infty}} \dd{\tau} \\
&\gtrsim \int_t^\infty \tau^{\frac{s}{p} - \frac{q}{p} - 1} \ell(\tau)^{s\beta_\infty - q\alpha_\infty} \dd{\tau} \\
&\approx t^{\frac{s}{p} - \frac{q}{p}} \ell(t)^{s\beta_\infty - q\alpha_\infty}
\end{align*}
thanks to the fact that $s/p - q/p < 0$. Furthermore, we have
\begin{equation*}
\ell(t)^{s\beta_\infty - q\alpha_\infty} \geq \ell(t)^{(s-q)\alpha_\infty + \frac{s-q}{q}} \quad\text{for every $t\geq1$}
\end{equation*}
thanks to the last inequality in~\eqref{prop:ac_near_infinity_LZ_examples:eq8}. Hence, for each $a\geq1$, we obtain
\begin{align*}
&\int_0^\infty \Bigg( \int_t^\infty \frac{(\tau - a)^{\frac{s}{p} - 1} \ell^{s\B}(\tau - a)}{\tau^{\frac{q}{p}} \ell^{q\A}(\tau)} \chi_{(a, \infty)}(\tau) \dd{\tau} \Bigg)^{\frac{q}{q-s}} t^{\frac{q}{p} - 1} \ell^{q\A}(t) \dd{t} \\
&\gtrsim \int_{t_a}^\infty \big( t^{\frac{s}{p} - \frac{q}{p}} \ell(t)^{s\beta_\infty - q\alpha_\infty} \big)^{\frac{q}{q-s}} t^{\frac{q}{p} - 1} \ell(t)^{q\alpha_\infty} \dd{t} \\
&\geq \int_{t_a}^\infty t^{-1} \ell(t)^{-1} \dd{t} = \infty.
\end{align*}
It follows from this combined with \eqref{prop:ac_near_infinity_LZ_examples:eq10} and \eqref{prop:ac_near_infinity_LZ_examples:eq11} that \eqref{prop:ac_near_infinity_LZ_examples:eq1} is not valid, which finishes the proof.
\end{proof}

\begin{remark}\label{rem:decay_of_fund_not_imply_ac_near_infinity_LZ_example}
Let $p\in(1, \infty)$, $1\leq s < q\leq \infty$, and $\A=(\alpha_0, \alpha_\infty)$, $\B=(\beta_0, \beta_\infty)\in\R^2$. Whereas
\begin{equation*}
\lim_{a\to\infty} \frac{\fundX[{\LZ[]{p}{s}{\B}}](a)}{\fundX[{\LZ[]{p}{q}{\A}}](a)} = 0
\end{equation*}
is true if and only if $\beta_\infty < \alpha_\infty$,
\begin{equation*}
\lim_{a\to\infty}\sup_{\|f\|_{\LZ{p}{q}{\A}}\leq1} \|f^*\chi_{(a, \infty)}\|_{\LZ{p}{s}{\B}} = 0
\end{equation*}
is true if and only if $\beta_\infty < \alpha_\infty + 1/q - 1/s$. Hence, $X=L^{p,q,\A}$ and $Y=L^{p,s,\B}$ with $\beta_\infty\in[\alpha_\infty + 1/q - 1/s, \alpha_\infty)$ is a concrete example of \riSps{} for which the decay of the ratio $\fundX[Y]/\fundX$ is not sufficient for the validity of \eqref{thm:characterization_of_comp_emb_Y:eq:globally_ac}.
\end{remark}

As a special case of \cref{thm:compactness_entire_space_LZ_examples}, we obtain a complete characterization of the compactness of \eqref{thm:characterization_of_comp_emb_Y:eq:embedding} with $X$ and $Y$ being ``logarithmic Orlicz spaces'', which are sometimes called Zygmund spaces. In fact, with a bit of extra effort, we can obtain the characterization even when the Young functions generating these spaces have different behavior near 0  and near $\infty$. We omit its proof because it is essentially the same as that of~\cref{thm:compactness_entire_space_LZ_examples}. Recall that, loosely speaking, the behavior of a Young function near $\infty$ influences the ``local part'' of the corresponding Orlicz space, whereas its behavior near $0$ influences the ``global part'' of the Orlicz space.
\begin{theorem}\label{thm:compactness_entire_space_LLogL_examples}
Let $m,n\in\N$, $n\geq2$. Let $p_1,p_2,r_1,r_2\in[1, \infty)$ and $\gamma_1,\gamma_2, \delta_1,\delta_2\in\R$. Assume that if $p_1 = 1$ or $r_1 = 1$, then $\gamma_1\leq0$ or $\delta_1\leq0$; and if $p_2 = 1$ or $r_2 = 1$, then $\gamma_2\geq0$ or $\delta_2 \geq 0$, respectively. Let $A$ and $B$ be Young functions that are equivalent to the functions
\begin{align*}
t\mapsto\begin{cases}
t^{p_1} \ell(t)^{\gamma_1} \quad \text{near $0$},\\
t^{p_2} \ell(t)^{\gamma_2} \quad \text{near $\infty$},
\end{cases}
\intertext{and}
t\mapsto\begin{cases}
t^{r_1} \ell(t)^{\delta_1} \quad \text{near $0$},\\
t^{r_2} \ell(t)^{\delta_2} \quad \text{near $\infty$},
\end{cases}
\end{align*}
respectively.

When $m<n$, the Sobolev embedding 
\begin{equation}\label{thm:compactness_entire_space_LLogL_examples:Sob_emb}
\wmrX[m][L^{A}] \hookrightarrow L^{B}(\rn)
\end{equation}
is compact if and only if one of the conditions
\begin{enumerate}[label=(D\arabic*), ref=(D\arabic*)]
	\item\label{enum:compactness_entire_space_LLogL_examples_globally_AC:1} $p_1 < r_1$;
	\item\label{enum:compactness_entire_space_LLogL_examples_globally_AC:2} $p_1 = r_1$ and $\delta_1 < \gamma_1$
\end{enumerate}
is satisfied and simultaneously so is one of the following:
\begin{enumerate}[label=(D\arabic*), ref=(D\arabic*), resume]
	\item $p_2 < \frac{n}{m}$ and $r_2 < \frac{np_2}{n-mp_2}$;
	\item $p_2 < \frac{n}{m}$, $r_2 = \frac{np_2}{n-mp_2}$, and $\delta_2 < \frac{n}{n-mp_2} \gamma_2$;
	\item $p_2 \geq \frac{n}{m}$.
\end{enumerate}

When $m \geq n$, the Sobolev embedding \eqref{thm:compactness_entire_space_LLogL_examples:Sob_emb} is compact if and only if one of the conditions \ref{enum:compactness_entire_space_LLogL_examples_globally_AC:1}--\ref{enum:compactness_entire_space_LLogL_examples_globally_AC:2} is satisfied.
\end{theorem}

\subsection{Weighted Sobolev embeddings on balls}
We conclude this paper by describing the optimal target \riSp{} in the weighted Sobolev inequality on balls \eqref{E:embedding_on_balls_Sob_emb} when $X$ is a Lorentz--Zygmund space, and by characterizing when the weighted Sobolev embedding
\begin{equation}\label{thm:weighted_Sob_ball_LZ_examples:Sob_emb_LZ_both_sides}
\wmrXB[m][L^{p,q,\gamma_1}] \hookrightarrow L^{r,s,\gamma_2}(B_R,\aldx)
\end{equation}
is compact.
\begin{theorem}\label{thm:weighted_Sob_ball_LZ_examples_optimal_spaces}
Let $m,n\in\N$, $m<n$, $R\in(0, \infty)$, and $\alpha\geq0$. Let $p,q\in[1, \infty]$ and $\gamma_1\in\R$ be such that $L^{p,q,\gamma_1}(B_R)$ is (equivalent to) an \riSp{} (recall~\eqref{EQ:LZ_equiv_RI} with $\alpha_0 = \gamma_1$ and $\alpha_\infty$ disregarded).

When $p\in[1, \frac{n}{m})$, the optimal \riSp{} $Y_X(B_R, \aldx)$ for $X = L^{p,q,\gamma_1}(B_R)$ in the embedding
\begin{equation}\label{thm:weighted_Sob_ball_LZ_examples_optimal_spaces:Sob_emb}
\wmrXB[m][L^{p,q,\gamma_1}] \hookrightarrow Y(B_R,\aldx)
\end{equation}
is
\begin{equation*}
Y_X(B_R, \aldx) = L^{\frac{(n+\alpha)p}{n-mp}, q, \gamma_1}(B_R, \aldx).
\end{equation*}

When $p = \frac{n}{m}$, the optimal \riSp{} satisfies
\begin{equation*}
Y_X(B_R, \aldx) = \begin{cases}
L^{\infty, q, \gamma_1 - 1}(B_R, \aldx) \quad &\text{when $\gamma_1 < 1 - \frac1{q}$},\\
L^{\infty, q, -\frac1{q}, -1}(B_R, \aldx) \quad &\text{when $q\in(1, \infty]$ and $\gamma_1 = 1 - \frac1{q}$},\\
L^\infty(B_R, \aldx) \quad &\text{when either $q\in(1, \infty]$ and $\gamma_1 > 1 - \frac1{q}$} \\
\quad &\text{\hphantom{when }or $q = 1$ and $\gamma_1 \geq 0$}.
\end{cases}
\end{equation*}

Finally, when $p > \frac{n}{m}$, we have
\begin{equation*}
Y_X(B_R, \aldx) = L^\infty(B_R, \aldx).
\end{equation*}
\end{theorem}
\begin{proof}
First, we use \cref{cor:embedding_on_balls_optimal_target} to obtain a description of the optimal target space $Y_X(B_R, \aldx)$ for $X = L^{p,q,\gamma_1}(B_R)$ in \eqref{thm:weighted_Sob_ball_LZ_examples_optimal_spaces:Sob_emb}. More precisely, we obtain a description of $Y_X'(B_R, \aldx)$. With just minor, obvious changes, we can proceed as in \cite[Proof of Theorem~5.1]{CP:16} or \cite[Proof of Theorem~5.1]{CM:22} to obtain
\begin{equation}\label{thm:weighted_Sob_ball_LZ_examples:optimal_assoc}
Y_X'(B_R, \aldx) = \begin{cases}
L^{\frac{(n+\alpha)p}{(n+\alpha+m)p - n}, q', -\gamma_1}(B_R, \aldx)\quad &\text{if $p\in[1, \frac{n}{m})$}, \\
L^{(1,q',-\gamma_1)}(B_R, \aldx)\quad &\text{if $p=\frac{n}{m}$}, \\
L^1(B_R, \aldx)\quad &\text{if $p>\frac{n}{m}$}.
\end{cases}
\end{equation}
Here, $L^{(1,q',-\gamma_1)}(B_R, \aldx)$ is a Lorentz--Zygmund space defined by means of the functional $\|f\|_{L^{(1,q',-\gamma_1)}(B_R, \aldx)} = \|t^{1-\frac1{q'}}\ell(t)^{-\gamma_1}f^{**}(t)\|_{L^{q'}(0,\aldx(B_R))}$. The desired description of $Y_X(B_R, \aldx)$ then follows from the description of the associate spaces of the Lorentz--Zygmund spaces in \eqref{thm:weighted_Sob_ball_LZ_examples:optimal_assoc} (see~\cite[Section~6]{OP:99}).
\end{proof}

Finally, we characterize when the embedding \eqref{thm:weighted_Sob_ball_LZ_examples:Sob_emb_LZ_both_sides} is compact. In view of~\cref{rem:embedding_on_balls_compactness:m_geq_n_not_interesting}, we only need to focus on the case where $m < n$.
\begin{theorem}\label{thm:weighted_Sob_ball_LZ_examples_compactness}
Let $m,n\in\N$, $m<n$, $R\in(0, \infty)$, and $\alpha\geq0$. Let $p,q,r,s\in[1, \infty]$ and $\gamma_1,\gamma_2\in\R$ be such that $L^{p,q,\gamma_1}(B_R)$ and $L^{r,s,\gamma_2}(B_R, \aldx)$ are (equivalent to) \riSps{}. Then the Sobolev embedding \eqref{thm:weighted_Sob_ball_LZ_examples:Sob_emb_LZ_both_sides} is compact if and only if one of the following is true:
\begin{enumerate}[label=(E\arabic*), ref=(E\arabic*)]
	\item\label{enum:thm:weighted_Sob_ball_LZ_examples_compactness:1} $p\leq \frac{n}{m}$ and $r<\frac{(n+\alpha)p}{n-mp}$;
	\item $p < \frac{n}{m}$, $r = \frac{(n+\alpha)p}{n-mp}$, and $\gamma_2 < \gamma_1 + \min\Big\{ \frac1{q}-\frac1{s},0 \Big\}$;
	\item $p = \frac{n}{m}$, $r = \infty$, $\gamma_1 \leq 1 - \frac1{q}$, and $\gamma_2 < \gamma_1 - 1 +\frac1{q}-\frac1{s}$;
	\item\label{enum:thm:weighted_Sob_ball_LZ_examples_compactness:4} either $p = \frac{n}{m}$ and $\gamma_1 > 1 - \frac1{q}$ or $p > \frac{n}{m}$.
\end{enumerate}
\end{theorem}
\begin{proof}
First, assume that $L^{r,s,\gamma_2}(B_R,\aldx)\neq L^\infty(B_R,\aldx)$\textemdash that is, it is not true that $r=s=\infty$ and $\gamma_2=0$. Thanks to \cref{thm:embedding_on_balls_compactness_not_Linfty}, the Sobolev embedding \eqref{thm:weighted_Sob_ball_LZ_examples:Sob_emb_LZ_both_sides} is compact if and only if
\begin{equation}\label{thm:weighted_Sob_ball_LZ_examples:optimal_ac_to_target}
Y_X(B_R, \aldx) \ac L^{r,s,\gamma_2}(B_R,\aldx),
\end{equation}
where $Y_X(B_R, \aldx)$ is the optimal target space in \eqref{thm:weighted_Sob_ball_LZ_examples_optimal_spaces:Sob_emb} from \cref{thm:weighted_Sob_ball_LZ_examples_optimal_spaces}. Now, the validity of almost compact embeddings between two Lorentz--Zygmund spaces is well known (e.g., see~\cite[Proposition~7.12]{S:15}). It follows that \eqref{thm:weighted_Sob_ball_LZ_examples:optimal_ac_to_target} is valid if and only if one of the conditions \ref{enum:thm:weighted_Sob_ball_LZ_examples_compactness:1}--\ref{enum:thm:weighted_Sob_ball_LZ_examples_compactness:4} is true.

Finally, assume that $L^{r,s,\gamma_2}(B_R,\aldx) = L^\infty(B_R,\aldx)$\textemdash that is, $r=s=\infty$ and $\gamma_2=0$. In view of the discussion above \cref{thm:embedding_on_balls_compactness_not_Linfty} (in particular, recall~\eqref{E:embedding_on_balls_compactness_into_Linfty_char_domain_ac_to_Lorentz}), the Sobolev embedding
\begin{equation*}
\wmrXB[m][L^{p,q,\gamma_1}] \hookrightarrow L^\infty(B_R,\aldx)
\end{equation*}
is compact if and only if
\begin{equation*}
L^{p,q,\gamma_1}(0,1) \ac L^{\frac{n}{m},1}(0,1).
\end{equation*}
Straightforward computations show (see~\cite[Proposition~7.12]{S:15} again) that this is the case if and only if \ref{enum:thm:weighted_Sob_ball_LZ_examples_compactness:4} is true, which concludes the proof.
\end{proof}

\section*{Acknowledgment}
The author would like to thank the referee for their valuable comments.

\end{document}